%% file: cmguide.tex
\documentclass[noams]{compositio}

\usepackage[latin1]{inputenc}
\usepackage[T1]{fontenc}
\usepackage[francais]{babel}
\usepackage[mathscr]{eucal}

\usepackage{amsmath}
\usepackage{amssymb}
\usepackage{tabularx}
\usepackage{epsfig, floatflt,amssymb}

\usepackage[all]{xy}

\newcommand{\OO}{\mathcal{O}}

\newcommand{\QQ}{\mathcal{Q}}

\newcommand{\X}{\mathcal{X}}

\newcommand{\Pic}{\mathrm{Pic}}

\newtheorem{theorem}{Théorème}[subsubsection]
\newtheorem{prop}[theorem]{Proposition}
\newtheorem{rem}[theorem]{Remarque}
\newtheorem{lemma}[theorem]{Lemme}
\newtheorem{corollary}[theorem]{Corollaire}
\newtheorem{exam}[theorem]{Exemple}
\newtheorem{definition}[theorem]{Définition}
\newtheorem{notation}[theorem]{Notation}

\newtheorem{theoreme}{Théorème}[subsection]
\newtheorem{defn}[theoreme]{Définition}

\newtheorem{quest}[theoreme]{Question}
\newtheorem{proposition}[theoreme]{Proposition}

\newtheorem{theoreme1}{Théorème}[section]
\newtheorem{defn1}[theoreme1]{Définition}

\newtheorem{quest1}[theoreme1]{Question}
\newtheorem{proposition1}[theoreme1]{Proposition}
\newtheorem{lemme1}[theoreme1]{Lemme}
\newtheorem{remarque1}[theoreme1]{Remarque}

\newtheorem{corollaire}[theoreme]{Corollaire}
\newtheorem{remarques}[theoreme]{Remarque}
\newtheorem{lemme}[theoreme]{Lemme}

\newenvironment{remark}{\begin{rem} \upshape}{\end{rem}}

\newenvironment{example}{\begin{exam}\upshape}{\end{exam}}

\newenvironment{remarque}{\begin{remarques}\upshape}{\end{remarques}}

\newenvironment{theo}{\vspace{2mm}\noindent \textbf{Théorème.~}\itshape}{\hfill\vspace{2mm}}

\renewcommand{\contentsname}{Contents\\{\footnotesize\normalfont(A table
of contents should normally not be included)}}

\begin{document}

\title{Diviseur thêta et formes différentielles}
\author{Jilong Tong}
\email{jilong.tong@uni-due.de}
\address{Universität Duisburg-Essen, Fachbereich Mathematik, Campus Essen 45117 Essen, Germany}


\begin{abstract}This article concerns the geometry of algebraic
curves in characteristic $p>0$. We study the geometric and
arithmetic properties of the theta divisor $\Theta$ associated to
the vector bundle of locally exact differential forms of a curve.
Among other things, we prove that, for a generic curve of genus
$\geq 2$, this theta divisor $\Theta$ is always geometrically
normal. We give also some results in the case where either $p$ or
the genus of the curve is small. In the last part, we apply our
results on $\Theta$ to the study of the variation of fundamental
group of algebraic curves. In particular, we refine a recent
result of Tamagawa on the specialization homomorphism between
fundamental groups at least when the special fiber is
supersingular.
\end{abstract}

\maketitle


\input{introduction}
\input{chapitre1}
\input{chapitre2}
\input{chapitre3}
\input{chapitre4}
\input{chapitre5}

\begin{acknowledgements}
Je suis très reconnaissant à M. Raynaud, qui m'a proposé ce sujet.
Je remercie aussi O. Debarre pour une simplification de la preuve
du lemme \ref{lemme sur effectif} et pour sa longue liste de
commentaires. Je remercie enfin le(a) referee pour sa lecture
attentive et ses remarques.
\end{acknowledgements}

\addcontentsline{toc}{section}{Bibliographie}
\bibliographystyle{alpha}

\end{document}

%% file: introduction.tex
\section*{Introduction}
\addcontentsline{toc}{section}{Introduction}
\markboth{Introduction}{Introduction} Le but de ce travail est
l'étude géométrique du diviseur thêta associé au faisceau des
formes différentielles localement exactes sur une courbe, en
caractéristique $p>0$. Rappelons brièvement la situation. Soient
$k$ un corps algébriquement clos de caractéristique $p>0$, $X$ une
courbe projective lisse connexe de genre $g\geq 1$ sur $k$. Notons
$X_1$ l'image réciproque de $X$ par le Frobenius absolu de $k$,
$J_1$ sa jacobienne. Soit $F:X\rightarrow X_1$ le Frobenius
relatif. Définissons le faisceau $B$ par la suite exacte suivante:
$$
0\rightarrow \OO_{X_1}\rightarrow F_{\ast}\OO_{X}\rightarrow
B\rightarrow 0.
$$
Alors, $B$ est un fibré vectoriel sur $X_1$ de rang $p-1$ et de
pente $g-1$. Le point de départ de cet article est le théorème
suivant dû à Raynaud (\cite{R1}): pour un faisceau inversible $L$
de degré $0$ assez général, $H^{0}(X_1, B\otimes L)=0$. Par
conséquent, le fibré $B$ possède un diviseur thêta
$\Theta=\Theta_B$, défini comme déterminant du complexe
$\mathrm{R}f_{\ast}(B\otimes \mathcal{P})$, où $f:X_1\times
J_1\rightarrow J_1$ est le morphisme de projection, et
$\mathcal{P}$ est un faisceau de Poincaré sur $X_1\times J_1$.
C'est un diviseur effectif sur la jacobienne $J_1$ de $X_1$ dont
le support est l'ensemble des $L\in J_1$ tels que $H^{0}(X_1,
B\otimes L)\neq 0$. Le diviseur $\Theta$ est canoniquement attaché
à la courbe $X$ et algébriquement équivalent à
$(p-1)\Theta_{\mathrm{class}}$, où $\Theta_{\mathrm{class}}$
désigne la polarisation principale naturelle de la jacobienne
$J_1$.

Pour $p=2$, $B$ est un faisceau inversible, racine carrée
canonique du faisceau des formes différentielles sur $X_1$, et
$\Theta$ est un diviseur thêta classique; en particulier il est
normal irréductible et ses singularités ont été intensivement
étudiées.

Par contre, pour $p$ impair, on sait très peu de choses sur
$\Theta$, mise à part son existence, connue depuis plus de $25$
ans. La littérature sur les diviseurs thêta des fibrés est récente
et concerne surtout le rang $2$. Pour $p>0$, nous prouvons que
$\Theta$ est normal lorsque $X$ est une courbe assez générale.
Lorsque $p=3$ et $g\geq 2$, nous montrons que $\Theta$ est réduit,
et qu'il est intègre si $g=2$ et $p=3$. Nous donnons aussi des
exemples où $\Theta$ n'est pas normal. Pour $p>2$, nous ignorons
si $\Theta$ est toujours irréductible.

A côté de ses propriétés géométriques, le diviseur thêta présente
un intérêt arithmétique: comme l'a observé Raynaud, les points de
torsion d'ordre $n$ premier à $p$, contenus dans $\Theta$
contrôlent les revêtements étales cycliques $Y\rightarrow X$ de
degré $n$, tels que la partie nouvelle de la jacobienne de $Y$ ne
soit pas ordinaire. Du coup, le diviseur $\Theta$ a été utilisé
pour étudier la variation du groupe fondamental en caractéristique
$p>0$, en fonction de la variation de $X$ (\cite{Pop-Saidi},
\cite{R2}, \cite{Tamagawa}). Dans ce problème, on rencontre de
sérieuses difficultés, s'il existe des composantes irréductibles
de $\Theta$ qui sont des translatés de sous-variétés  abéliennes.
C'est la raison pour laquelle nous nous intéressons à cette
question, en particulier au chapitre $4$. Nous pouvons alors
améliorer les résultats  de Raynaud et Tamagawa, du moins si la
fibre spéciale est supersingulière.

Analysons brièvement le contenu des divers chapitres.

Le chapitre $1$ commence par la définition du diviseur $\Theta$.
Soit $X\rightarrow S$ une courbe lisse. Lorsque elle possède une
section, on dispose d'un faisceau de Poincaré $\mathcal{P}$ sur
$X_1\times_S J_1$ (où $J_1$ est la jacobienne de $X_1/S$). On
considère $\mathrm{R}f_{\ast}(B\otimes \mathcal{P})$, comme un
objet de $\mathrm{D}^{b}_{c}(J_1)$ la catégorie dérivée des
$\OO_{J_1}$-modules à cohomologie cohérente et à degrés bornées.
D'après \cite{R1}, ce complexe est concentré en degré $1$, par
suite, $\QQ:=\mathrm{R}^{1}f_{\ast}(B\otimes \mathcal{P})$ est un
$\OO_{J_1}$-module de dimension projective $1$. Le diviseur
$\Theta$ est alors défini comme le déterminant de $\QQ$, et sa
construction ne dépend pas du choix de faisceau de Poincaré
$\mathcal{P}$. Par conséquent, par descente étale, on peut
attacher canoniquement le diviseur $\Theta$ à n'importe quelle
$S$-courbe lisse. Afin d'employer les techniques de
dégénérescence, on présente aussi une définition de $\Theta$ pour
une $S$-courbe semi-stable en suivant une stratégie analogue.

Ensuite, on rassemble les premières propriétés de $\Theta$ à
partir de sa définition. Pour une $S$-courbe lisse, le fibré $B$
est équipé d'un accouplement alterné non-dégénéré à valeurs dans
$\Omega^{1}_{X_1/S}$ le faisceau canonique de $X_1/S$. Il en
résulte que $B$ est auto-dual pour la dualité de Serre. Le
diviseur $\Theta$ est donc symétrique, et même totalement
symétrique au-sens de Mumford \cite{Eq} pour $p$ impair. La
transformation de Fourier-Mukai intervient pour établir un lien
entre $B$ et le faisceau $\QQ$, qui sera utile pour montrer
certaines propriétés du diviseur $\Theta$ pour la courbe
générique. La dualité de Serre-Grothendieck nous permet de montrer
que les faisceaux $\QQ$ et $(-1)^{\ast}\QQ$ sont en dualité à
valeurs dans le faisceau canonique de $\Theta$.

Une particularité du diviseur thêta, liée au noyau du
Verschiebung, est la {\textquotedblleft propriété de
Dirac\textquotedblright}. Cette propriété sera utilisée très
souvent dans l'étude du diviseur $\Theta$. Comme dans le cas
classique, dès que $g\geq 4$, le faisceau $\QQ$ n'est jamais
inversible, en particulier, le diviseur $\Theta$ a toujours des
singularités lorsque la courbe est de genre $g\geq 4$. Puis, nous
revisitons la démonstration de Joshi \cite{Joshi} du fait que le
faisceau $B$ est stable lorsque $g\geq 2$, en en donnant une
preuve plus directe, qui fournit aussi un contrôle effectif du
degré des sous-fibrés de $B$. A la fin de ce chapitre, on
rassemble quelques questions qui se posent naturellement à propos
de l'étude géométrique et arithmétique du diviseur $\Theta$, et
qui feront aussi l'objet des chapitres suivants.

Le chapitre 2 est consacré à l'étude différentielle de $\Theta$.
On rappelle d'abord le travail de Laszlo (\cite{Laszlo}), sur la
multiplicité du diviseur thêta universel. Puis, on applique ce
résultat à l'étude de la multiplicité de $\Theta$, en particulier
aux points d'ordre $p$, et met en évidence le rôle des formes de
Cartier. Nous faisons également l'étude différentielle du schéma
de Hilbert $\mathcal{H}$ des faisceaux inversibles de degré $0$
plongés dans $B$. On examine notamment les points $x\in
\mathcal{H}$ où $\mathcal{H}$ est lisse de dimension $g-1$, ce
résultat sera utile dans l'étude de $\Theta$ en caractéristique
$p=3$.

Le but principal du chapitre 3 est de montrer que pour la courbe
générique de genre $\geq 2$ dans l'espace de modules des courbes,
le diviseur $\Theta=\Theta_{\mathrm{gen}}$ est géométriquement
intègre et normal. La preuve est un peu technique. Elle utilise:
(i) les courbes réductibles dégénérées, (ii) l'action de
monodromie, (iii) les déformations des points doubles ordinaires
(de codimension $1$), (iv) le fait que le groupe de Néron-Severi
de la jacobienne de la courbe générique est isomorphe à
$\mathbf{Z}$, et (v) une propriété de GAGA formel pour un schéma
qui n'est pas forcément propre (qui est montrée dans SGA2).
Examinons par exemple le cas du genre $2$. On dégénère la courbe
générique $X_{\eta}$ en une courbe $X_s$ semi-stable constituée de
$2$ courbes elliptiques (ordinaires) génériques se coupant
transversalement, telles que l'action de la monodromie sur les
points d'ordre $p$ soit maximale. Alors le diviseur $\Theta$
relatif est une courbe semi-stable, à fibre spéciale géométrique
$\Theta_s$ formée de courbes elliptiques telles que la monodromie
permute transitivement les points doubles ordinaires. Alors, ou
bien la fibre générique $\Theta_{\eta}$ est lisse, et l'on gagne,
ou bien sa normalisation est formée de courbes elliptiques. Ce
dernier cas est exclu car la jacobienne $J_{\eta}$ de la courbe
générique $X_{\eta}$ ne contient pas de courbes elliptiques
(\cite{Mori}). La preuve pour les cas $g\geq 3$ se fait par
récurrence sur le genre $g$, en utilisant une stratégie analogue
et les résultats de SGA2.

Au chapitre 4, nous étudions $\Theta$ lorsque $g$ est petit ou
lorsque $p=3$. Pour $p=3$, $B$ est un fibré vectoriel de rang $2$,
et l'on a une certaine prise sur le lieu singulier de $\Theta$.
Cette particularité nous permet de montrer que $\Theta$ est réduit
et ne contient pas de composantes irréductibles translatées d'une
sous-variété abélienne. Si on suppose de plus $g=2$, on montre que
$\Theta$ est intègre. Par contre, pour une courbe de genre $g=2$
en caractéristique $p\geq 5$, on ne sait pas si un tel énoncé
reste valable. Toutefois, supposons la courbe $X$ ordinaire (de
genre $2$), nous montrons qu'aucune composante irréductible de
$\Theta$ n'est le translaté d'une courbe elliptique.

Dans le chapitre 5, on donne des applications du diviseur $\Theta$
à l'étude du groupe fondamental d'une courbe en caractéristique
$p>0$. Le point de départ est le fait suivant: les points de
torsion, d'ordre premier à $p$, situés sur le diviseur thêta
déterminent un quotient métabélien du groupe fondamental $\pi_1$,
le $\pi^{\mathrm{new,ord}}_{1}$ qui rend compte des revêtements
cycliques finis étales $Y\rightarrow X$ de degré premier à $p$,
tels que la {\textquotedblleft partie nouvelle\textquotedblright}
de la jacobienne de $Y$ soit ordinaire. Supposons que
$\pi_{1}^{\mathrm{new,ord}}$ soit constant quand on passe de la
fibre générique géométrique $X_{\bar{\eta}}$ à la fibre spéciale
géométrique $X_{\bar{s}}$. Alors la courbe $X/S$ est-elle
constante, du moins lorsque $X_s$ est définissable sur un corps
fini? Une réponse positive permettait de préciser les résultats de
Tamagawa. En fait, faute de renseignements suffisants sur le
diviseur $\Theta$ et sur la saturation de la torsion, on rencontre
des difficultés très sérieuses. Dans cet article, on ne peut
donner une réponse positive à cette question que dans le cas où la
jacobienne de $X/S$ a une fibre spéciale supersingulière.

%% file: chapitre1.tex

\section{Le fibré $B$ des formes différentielles localement exactes
et son diviseur thêta} \markboth{chapitre1}{chapitre1}

\subsection{Définition du diviseur thêta}

Sauf mention du contraire, $S$ désigne un schéma localement
noethérien de caractéristique $p>0$ (c'est-à-dire, $p\cdot
1_{\OO_S}=0$).

\subsubsection{Le cas relatif lisse.}\label{construction, cas lisse}

Soit $X/S$ une courbe relative propre lisse à fibres géométriques
connexes de genre $g\geq 0$. Définissons $X_1$ par le diagramme
cartésien suivant
$$
\xymatrix{X_1\ar[r]\ar[d]& X\ar[d] \\ S\ar[r]^{Fr_S}& S}
$$
où $Fr_S:S\rightarrow S$ désigne le Frobenius absolu. Soit
$F:X\rightarrow X_1$ le Frobenius relatif. On dispose de la suite
exacte courte suivante:
$$
0\rightarrow \mathcal{O}_{X_1}\rightarrow
F_{\ast}\mathcal{O}_{X}\rightarrow B\rightarrow 0,
$$
où le faisceau $B$ est, par définition, \emph{le faisceau des
formes différentielles localement exactes sur $X_1$}. C'est un
fibré vectoriel de rang $p-1$ sur $X_1$.

\begin{remark}
Si $S=\mathrm{Spec}(k)$, et si $X$ est la droite projective sur
$S$. Alors $B \simeq \OO_{X_1}(-1)^{\oplus^{p-1}}$.
\end{remark}

Supposons désormais dans ce chapitre que $g\geq 1$. Notons $J$
(resp. $J_1$) la jacobienne de $X/S$ (resp. $X_1/S$). Alors $J_1$
est l'image réciproque de $J/S$ par le Frobenius absolu
$F_S:S\rightarrow S$. Formons le carré cartésien suivant
$$
\xymatrix{X_1\times_S J_1\ar[r]^<<<<{pr_{X_1}}\ar[d]^{pr_{J_1}}& X_1\ar[d] \\
J_1\ar[r] & S}.
$$

Supposons d'abord que $X/S$ admette une section $\varepsilon\in
X(S)$, d'où une section $\varepsilon_1\in X_1(S)$. Soit
$\mathcal{P}_{\varepsilon,1}$ le faisceau de Poincaré rigidifié
sur $X_1\times_S J_1$ associé à $\varepsilon$ (\cite{BLR1} 8.2/1
et 8.2/4). Par définition, $\mathcal{P}_{\varepsilon,1}$ est un
faisceau inversible sur $X_1\times_S J_1$ muni d'un isomorphisme
$u:\varepsilon_{1,J_1}^{\ast}(\mathcal{P}_{\varepsilon,1})\simeq
\OO_{J_1}$ (où $\varepsilon_{1,J_1}:J_1\rightarrow X_1\times_S
J_1$ est le changement défini par la section $\varepsilon_1\in
X_1(S)$). Considérons le complexe
$\mathrm{R}pr_{J_1,\ast}(pr^{\ast}_{X_1}B\otimes
\mathcal{P}_{\varepsilon,1})$ qui est un objet de
$\mathrm{D}^{b}_{c}(J_1)$, la catégorie dérivée des complexes de
$\OO_{J_1}$-modules à degrés bornés et à cohomologie cohérente.
Comme $B$ est un faisceau cohérent sur $X_1$, plat sur $S$ et à
fibres de caractéristique d'Euler-Poincaré nulle, d'après
\cite{AV} (chap. 2, $\S$ 5 page 46 Theorem)
$\mathrm{R}pr_{J_1,\ast}(pr^{\ast}_{X_1}B\otimes
\mathcal{P}_{\varepsilon,1})\in \mathrm{D}^{b}_{c}(J_1)$ peut être
localement sur $J_1$ représenté par un complexe de faisceaux en
degré $0$ et $1$:
$$
\xymatrix{\cdots\ar[r] & 0\ar[r]& \mathcal{M}^0\ar[r]^{u}&
\mathcal{M}^{1}\ar[r] & 0\ar[r] & \cdots}
$$
où $\mathcal{M}^{0},\mathcal{M}^1$ sont libres de type fini et de
même rang.

Rappelons le résultat suivant, qui est notre point de départ pour
la théorie du diviseur $\Theta$.

\begin{theorem}\label{existence}(\cite{R1} théorème 4.1.1.) Si $S=\mathrm{Spec}(k)$
avec $k$ un corps algébriquement clos de caractéristique $p>0$.
Alors $h^{0}(B\otimes L)=0$ pour $L$ général dans $J_1(k)$.
\end{theorem}

 Il en résulte que la cohomologie de
$\mathrm{R}pr_{J_1,\ast}(pr^{\ast}_{X_1}B\otimes
\mathcal{P}_{\varepsilon,1})$ est concentrée en degré $1$. Soit
$\QQ=\mathrm{R}^{1}pr_{J_1,\ast}(pr^{\ast}_{X_1}B\otimes
\mathcal{P}_{\varepsilon,1})$, alors la formation de $\QQ$ commute
aux changements de base $S'\rightarrow S$ et $\QQ$ est plat sur
$S$. De plus, pour chaque $s\in S$, $\mathrm{det}(u)_s$ est
injectif et définit un diviseur de Cartier effectif relatif sur
$J_1/S$, que l'on note $\Theta_B$ ou simplement $\Theta$. Alors
$\Theta$ est le sous-schéma fermé de $J_1$ défini par le $0$-ième
idéal de Fitting de $\QQ$ (pour la définition d'idéal de Fitting,
voir \cite{Eisenbud} definition 20.4.).

 Le faisceau $\QQ$ dépend du choix de la section
$\varepsilon$. Plus précisément, si $\varepsilon'\in X(S)$ est une
autre section de $X/S$, $\varepsilon_1'\in X_1(S)$ la section de
$X_1$ associée. Notons $\mathcal{P}_{\varepsilon',1}$ le faisceau
de Poincaré rigidifié associé à $\varepsilon_1'$, et
$\QQ'=\mathrm{R}^{1}pr_{J_1,\ast}\left(pr_{X_1}^{\ast}B\otimes
\mathcal{P}_{\varepsilon',1}\right)$. Par définition du faisceau
de Poincaré rigidifié,
$\mathcal{P}_{\varepsilon',1}=\mathcal{P}_{\varepsilon,1}\otimes
pr_{J_1}^{\ast}\varepsilon_1'^{\ast}\mathcal{P}_{\varepsilon,1}^{-1}$.
On obtient donc
$$
\QQ'=\mathrm{R}^{1}pr_{J_1,\ast}\left(pr_{X_1}^{\ast}B\otimes
\mathcal{P}_{\varepsilon',1}\right)=\mathrm{R}^{1}pr_{J_1,\ast}\left(pr_{X_1}^{\ast}B\otimes
\mathcal{P}_{\varepsilon,1}\otimes
pr_{J_1}^{\ast}\varepsilon_1'^{\ast}\mathcal{P}_{\varepsilon,1}^{-1}\right)=\QQ\otimes
\varepsilon_1'^{\ast}\mathcal{P}_{\varepsilon,1}^{-1}.
$$
Donc $\QQ$ et $\QQ'$ diffèrent par tensorisation par un faisceau
inversible algébriquement équivalent à zéro de $J_1$.

 Par suite, $\Theta$ ne dépend pas du choix de la
section $\varepsilon\in X(S)$. Par descente étale, on en déduit
l'existence du diviseur $\Theta$ en général (dans le cas où on ne
suppose plus que $X/S$ admette une section). De plus, la formation
de $\Theta/S$ est compatible aux changements de base
$S'\rightarrow S$.

\subsubsection{Le cas relatif semi-stable.}

 Soit maintenant $X$ une courbe semi-stable sur $S$,
c'est-à-dire, une courbe relative propre plate sur $S$ dont les
fibres géométriques sont réduites connexes, avec pour seules
singularités des points doubles ordinaires. Comme dans le cas
lisse, notons $X_1$ l'image réciproque de $X$ par le Frobenius
absolu $F_S:S\rightarrow S$, et $F:X\rightarrow X_1$ le Frobenius
relatif.

Comme $X/S$ est semi-stable, le morphisme naturel
$\OO_{X_1}\rightarrow F_{\ast}(\OO_X)$ est universellement
injectif pour les changements de base $S'\rightarrow S$.
Définissons le faisceau cohérent $B$ par la suite exacte suivante:
$$
0\rightarrow \OO_{X_1}\rightarrow F_{\ast}(\OO_X)\rightarrow
B\rightarrow 0.
$$
Alors $B$ est un $\OO_{X_1}$-module cohérent, plat sur $\OO_S$
dont la formation est compatible aux changements de base
$S'\rightarrow S$.

 Notons $\mathrm{Pic}_{X/S}$ (resp.
$\mathrm{Pic}_{X_1/S}$) le foncteur de Picard de $X/S$ (resp. de
$X_1/S$), d'après M. Artin (\cite{BLR1} 8.3/2),
$\mathrm{Pic}_{X/S}$ (resp. $\mathrm{Pic}_{X_1/S}$) est
représentable par un espace algébrique en groupes lisse sur $S$.
Notons $J$ (resp. $J_1$) la composante neutre de
$\mathrm{Pic}_{X/S}$ (resp. de $\mathrm{Pic}_{X_1/S}$), d'après
Deligne (\cite{BLR1} 9.4/1), $J$ (resp. $J_1$) est représentable
par un schéma semi-abélien sur $S$.

  Ensuite, formons le diagramme cartésien suivant:
$$
\xymatrix{X_1\times_S J_1\ar[r]^<<<<{pr_{X_1}}\ar[d]^{pr_{J_1}}& X_1\ar[d] \\
J_1\ar[r] & S}.
$$
Procédons comme dans le paragraphe précédent ($\S$
\ref{construction, cas lisse}) pour associer à $B$ un diviseur
thêta $\Theta$. Lorsque $X/S$ admet une section $\varepsilon\in
X(S)$, notons $\varepsilon_1\in X_1(S)$ la section de $X_1/S$
associée à $\varepsilon$. Nous disposons d'un faisceau de Poincaré
rigidifié $\mathcal{P}_{\varepsilon,1}$ sur $X_1\times J_1$. On
peut donc considérer le complexe
$\mathrm{R}pr_{J_1,\ast}(pr_{X_1}^{\ast}\mathcal{B}\otimes
\mathcal{P}_{\varepsilon,1})\in \mathrm{D}^{b}_{c}(J_1)$. Là
encore, il y a un seul objet de cohomologie non nul en degré $1$
(\cite{R2} la remarque après la proposition 1.1.2)
$\QQ=\mathrm{R}^{1}pr_{J_1,\ast}(pr_{X_1}^{\ast}\mathcal{B}\otimes
\mathcal{P}_{\varepsilon,1})$. Le diviseur de Cartier $\Theta$ sur
$J_1$ est défini comme le $0$-ième idéal de Fitting du faisceau
$\QQ$. La définition de $\QQ$ dépend du choix de la section
$\varepsilon\in X(S)$: soient $\varepsilon'\in X(S)$ une autre
section, $\QQ'$ le faisceau analogue, alors $\QQ$ et $\QQ'$
diffèrent entre eux par un faisceau inversible sur $J_1$. Donc la
définition de $\Theta$ ne dépend pas du choix de $\varepsilon$.
Par suite, dans le cas général où nous ne supposons plus que $X/S$
admette une section, nous pouvons encore définir le diviseur thêta
$\Theta$, et sa formation est compatible aux changements de base
$S'\rightarrow S$.

\begin{remark}\label{theta pour les courbes stables} Dans cette remarque, $S=\mathrm{Spec}(k)$ avec $k$
un corps algébriquement clos de caractéristique $p>0$. Soient
$X/k$ une courbe semi-stable, $\pi: \widetilde{X}\rightarrow X$ sa
normalisée. Alors l'image réciproque
$\pi_1:\widetilde{X}_1\rightarrow X_1$ de $\pi$ par le Frobenius
absolu $Fr_k:\mathrm{Spec}(k)\rightarrow \mathrm{Spec}(k)$ est la
normalisée de $X_1$. Notons $\widetilde{J}_1$ la jacobienne de
$\widetilde{X}_1$, d'après \cite{BLR1} 9.2/8, $J_{1}$ est une
extension de $\widetilde{J}_1$ par un tore $T$:
$$
\xymatrix{0\ar[r]& T\ar[r]& J_1\ar[r]^{\nu}&
\widetilde{J}_{1}\ar[r] & 0}.
$$
avec $\nu:J_1\rightarrow \widetilde{J}_1=J_{\widetilde{X}_1/k}$ le
morphisme naturel induit par $\pi_1:\widetilde{X}_1\rightarrow
X_1$. Notons $\widetilde{B}$ le faisceau des formes
différentielles localement exactes sur $\widetilde{X}_1$. D'après
la fonctorialité du Frobenius relatif (\cite{SGA5}, Exposé XV,
$\S$ 1, proposition 1(a)), on a un diagramme commutatif
$$
\xymatrix{\widetilde{X}\ar[d]^{\pi}\ar[r]^{\widetilde{F}} &
\widetilde{X}_1\ar[d]^{\pi_1}\\ X\ar[r]^{F}& X_1}.
$$
Alors $B=\pi_{1,\ast}\widetilde{B}$ (\cite{R2} la remarque après
la proposition 1.1.2). Notons $\widetilde{\Theta}\subset
\widetilde{J}_1$ le diviseur thêta sur $\widetilde{J}_1$ associé à
$\widetilde{B}$. Alors $\Theta=\nu^{\ast}(\widetilde{\Theta})$
comme diviseur.
\end{remark}

\subsection{Les premières propriétés de $\Theta$}

Dans cette section, sauf mention du contraire, $S$ est un schéma
localement noethérien de caractéristique $p>0$, $X/S$ est une
courbe propre lisse à fibres géométriques connexes de genre $g\geq
1$.

\subsubsection{Auto-dualité.}

 Rappelons que l'on a la suite exacte suivante:
$$
\xymatrix{0\ar[r]& B\ar[r] & F_{\ast}\Omega^{1}_{X/S}\ar[r]^{c}&
\Omega^{1}_{X_1/S}\ar[r]&  0},
$$
dans laquelle la flèche $c:F_{\ast}(\Omega^{1}_{X/S})\rightarrow
\Omega^{1}_{X_1/S}$ désigne l'opérateur de Cartier de $X/S$.
L'application $(f,g)\mapsto c(fdg)$ de $F_{\ast}(\OO_{X})\times
F_{\ast}(\OO_{X})\rightarrow \Omega_{X_1/S}^{1}$ définit, par
passage au quotient, une application bilinéaire alternée:
$$
(\cdot,\cdot):B\otimes_{\OO_{X_1}}B\rightarrow \Omega^{1}_{X_1/S}
$$

\begin{prop}[\cite{R1}, $\S$ 4.1] \label{autodual} Cet accouplement est non-dégénéré.
Donc $B$ est auto-dual sous la dualité de Serre.
\end{prop}

\subsubsection{Le cas où $p=2$.}\label{cas p=2}

Supposons dans ce numéro que $S=\mathrm{Spec}(k)$ avec $k$ un
corps algébriquement clos de caractéristique $p=2$. Comme $p=2$,
le faisceau $B$ des formes différentielles localement exactes est
inversible sur $X_1$. Par l'auto-dualité de $B$ (\ref{autodual}),
$B$ est une racine carrée canonique de $\Omega^{1}_{X_1/k}$
(c'est-à-dire, une thêta caractéristique canonique de $X_1/k$). Le
diviseur $\Theta$ est donc un diviseur thêta classique symétrique
de $J_1$. Le résultat suivant découle de la théorie classique des
courbes algébriques.

\begin{theorem}[\cite{Laszlo}]Gardons les notations ci-dessus. Alors (1)
le diviseur $\Theta$ est irréductible et normal; (2) pour tout
$x\in \Theta\subset J_1$, la multiplicité de $\Theta$ en $x$,
notée $\mathrm{mult}_{x}(\Theta)$, est égale à
$h^{0}(X_1,B\otimes_{\OO_{X_1}}L_x)$, où $L_x$ est le faisceau
inversible de degré $0$ correspondant à $x\in J_1$.
\end{theorem}

\subsubsection{La classe d'équivalence linéaire de $\Theta$ (cas
$p\geq 3$).}

 Supposons $p\geq 3$ dans ce numéro. Pour $d\geq
1$ un entier, notons $X_{1}^{(d)}$ le produit symétrique de
$X_1/S$, comme $X_1/S$ est lisse, $X_{1}^{(d)}$ l'est aussi. Pour
$m$ un entier, notons $J_{1}^{[m]}$ le schéma qui classifie les
faisceaux inversibles de degré $m$ sur $X_1/S$.

 Considérons l'application
$$
X_{1}^{(g-1)}\rightarrow J_{1}^{[g-1]}
$$
définie par $(x_1,\cdots,x_{g-1})\mapsto
\OO_{X}(x_1+\cdots+x_{g-1})$. Son image schématique est un
diviseur $\Delta$ de $J^{[g-1]}$, réalisation canonique du
diviseur thêta classique $\Theta_{\mathrm{class}}$. Comme $p\geq
3$, localement sur $S$ pour la topologie étale, la courbe $X_1/S$
admet une thêta caractéristique, donc un diviseur thêta classique
symétrique.

\begin{prop}[\cite{R2} proposition 1.1.4]Si $p\geq 3$,
$\Theta$ est un diviseur symétrique de $J_1$. Si de plus, $X_1/S$
admet une thêta caractéristique $L$, alors $\Theta$ est
rationnellement équivalent au translaté de $(p-1)\Delta$ par
$L^{-1}$.
\end{prop}\label{theta est sym}

\begin{corollary} Le diviseur $\Theta$ est algébriquement
équivalent à $(p-1)\Theta_{\mathrm{class}}$.
\end{corollary}

 Supposons $p\geq 3$. Notons
$\mathrm{Kum}=\mathrm{Kum}_{J_1/S}$ le $S$-schéma, variété de
Kummer, quotient de $J_1$ par $\{\pm 1\}$. Alors $\mathrm{Kum}$
est un $S$-schéma propre et plat, à fibres géométriques normales,
dont la formation commute aux changements de base $S'\rightarrow
S$. Soit $U$ l'ouvert de $J_1$ complément des points d'ordre
divisant $2$, et $V$ son image dans $\mathrm{Kum}$. Alors le
morphisme de quotient $U\rightarrow V$ est fini étale de degré
$2$.


\begin{prop}\label{theta est fortement symetrique}
Gardons les notations ci-dessus, et supposons $p\geq 3$. Alors
$\Theta$ se descend canoniquement sur $\mathrm{Kum}_{J_1/S}$ en
$\widetilde{\Theta}$ qui est un diviseur de Cartier positif.
\end{prop}

\begin{corollary} Supposons $p\geq 3$, alors $\Theta$ est
totalement symétrique au sens de Mumford (\cite{Eq}, Page 305
definition). Pour un point $x\in J_1$ d'ordre divisant $2$, notons
$\mathrm{mult}_x(\Theta)$ la multiplicité de $\Theta$ en $x$,
alors $\mathrm{mult}_x(\Theta)\equiv 0(\mathrm{mod}~2)$.
\end{corollary}

\begin{proof}[Démonstration de la proposition \ref{theta est
fortement symetrique}] Pour le cas $g=1$, la conclusion résulte de
la description explicite de $\Theta$ pour les courbes elliptiques
(corollaire \ref{thetaforell}). On peut donc supposer $g\geq 2$.
Comme $\Theta$ est symétrique, $\Theta|_{U}$ se descend en
$\widetilde{\Theta}_V$ diviseur de Cartier relatif sur $V$, par
descente étale. Aux points de $x\in \mathrm{Kum}-V$, on a
$\mathrm{prof}_{x}(\OO_{\mathrm{Kum}})\geq 2$. Donc si le diviseur
de Cartier $\widetilde{\Theta}$ existe, son idéal de définition
sera égal à
$i_{\ast}\left(\OO_{V}(-\widetilde{\Theta}_{V})\right)$ ($\subset
i_{\ast}\OO_{V}\simeq \OO_{\mathrm{Kum}}$, d'après SGA2), où
$i:V\rightarrow \mathrm{Kum}$ est l'inclusion. Pour montrer que ce
faisceau d'idéaux est inversible, on peut faire une extension
étale de $S$, et donc supposer que $X_1$ possède une thêta
caractéristique, et donc un diviseur thêta classique symétrique
$\Theta_{\mathrm{class,sym}}$. Alors la norme
$\mathrm{Nm}_{J_1/\mathrm{Kum}}(\Theta_{\mathrm{class,sym}})=\widetilde{\Theta}_1$
est un diviseur de Cartier sur $\mathrm{Kum}$ qui descend
$2\Theta_{\mathrm{class,sym}}$. On a vu (proposition \ref{theta
est sym}) que $\Theta$ et $(p-1)\Theta_{\mathrm{class,sym}}$ sont
linéairement équivalents. Alors
$\left(\frac{p-1}{2}\right)\widetilde{\Theta}_1$ est un diviseur
de Cartier sur $\mathrm{Kum}$, qui sur $V$ est linéairement
équivalent à $\widetilde{\Theta}_V$. Finalement,
$i_{\ast}\left(\OO_{V}(-\widetilde{\Theta}_{V})\right)\simeq
\OO_{\mathrm{Kum}}\left(-\left(\frac{p-1}{2}\widetilde{\Theta}_1\right)\right)$
est inversible.
\end{proof}

\subsubsection{Le faisceau $\QQ$ et la transformation de
Fourier-Mukai.}\label{accouplement}

Pour un schéma abélien $A/S$, notons
$\mathcal{F}:\mathrm{D}^{b}_{\mathrm{c}}(A)\rightarrow
\mathrm{D}^{b}_{\mathrm{c}}(A^{\vee})$ le foncteur de
Fourier-Mukai pour $A$ défini par
$$
\mathcal{F}(\cdot)=\mathrm{R}\pi_{2,\ast}(\mathcal{L}
\otimes^{\mathrm{L}}\mathrm{L}\pi_{1}^{\ast}(\cdot)),
$$
avec $\mathcal{L}\in \mathrm{Pic}(A\times_S A^{\vee})$ le faisceau
de Poincaré normalisé et $\pi_1:A\times_S A^{\vee}\rightarrow A$,
$\pi_2:A\times_S A^{\vee}\rightarrow A^{\vee}$ les deux
projections canoniques. Pour un entier $i$, notons
$\mathcal{F}^{i}(L)$ le $i$-ième faisceau de cohomologie de
$\mathcal{F}(L)$ pour un objet $L\in
\mathrm{D}^{b}_{\mathrm{c}}(A)$. On renvoie le lecteur à
\cite{Mukai} pour les propriétés fondamentales de la
transformation de Fourier-Mukai, ou à \cite{Laumon} $\S$ 1 pour un
résumé.

Supposons dans ce numéro que $X/S$ admette une section
$\varepsilon\in X(S)$. D'où une section $\varepsilon_1\in X_1(S)$,
et le faisceau de Poincaré rigidifié
$\mathcal{P}_1=\mathcal{P}_{\varepsilon,1}$ associé sur
$X_1\times_S J_1$ à $\varepsilon_1$. Par conséquent, $X_1/S$ admet
un faisceau de Poincaré rigidifié
$\mathcal{P}_1=\mathcal{P}_{1,\varepsilon_1}$ pour la section
$\varepsilon_{1, J_1}:J_1\rightarrow X_1\times_S J_1$. Soit
$i_1:X_1\rightarrow \mathrm{Alb}(X_1/S)=J_{1}^{\vee}$ le
plongement d'Albanese tel que $i_1(\varepsilon_1)=0\in
J_{1}^{\vee}$. Comme d'habitude, notons
$pr_{J_1}:X_1\times_{S}J_1\rightarrow J_1$ et
$pr_{X_1}:X_1\times_S J_1\rightarrow X_1$ les deux projections
canoniques.

\subsubsection*{Le faisceau $\mathcal{Q}$.}

Rappelons d'abord que, par définition,
$\QQ=\mathrm{R}^{1}pr_{J_1,\ast}(pr_{X_1}^{\ast}B\otimes\mathcal{P}_1)$
($\S$ 1.1.1).

\begin{lemma}\label{lemme trivial} (1) Le faisceau $\QQ$ est en fait un
$\OO_{\Theta}$-module.

(2) Soit $x\in \Theta$ un point et notons
$\QQ(x):=\QQ\otimes_{\OO_{\Theta,x}}k(x)$. Alors
$\mathrm{dim}_{k(x)}(\QQ(x))=1$ si et seulement si $\QQ$ est
inversible comme $\OO_{\Theta}$-module dans un voisinage de $x\in
\Theta$.

(3) Le faisceau $\QQ$ est un $\OO_{J_1}$-module de Cohen-Macaulay
dès que $S$ est régulier.
\end{lemma}

\begin{proof} Les assertions (1) et (2) sont immédiates. Pour (3),
il suffit de remarquer que $\QQ$ admet, localement pour la
topologie de Zariski, une résolution libre de longueur $1$ sur
$J_1$ (on renvoie à $\S$ 1.1.1 pour l'existence d'une telle
résolution).
\end{proof}

Formons le diagramme commutatif suivant
$$
\xymatrix{X_1\times_S J_1\ar[r]^{i_1\times
1_{J_1}}\ar[d]^{pr_{X_1}} &
J_{1}^{\vee}\times_S J_1\ar[r]^{\pi_2}\ar[d]^{\pi_1} & J_1\ar[d]\\
X_1\ar[r]^{i_1} & J_{1}^{\vee} \ar[r]&S  }
$$
dont les carrés sont cartésiens. Soit $\mathcal{L}_1\in
\mathrm{Pic}(J_1\times_S J_{1}^{\vee})$ le faisceau de Poincaré
normalisé. Considérons la transformation de Fourier-Mukai du
faisceau $i_{1,\ast}B$ sur $J_{1}^{\vee}$, alors on a les
isomorphismes canoniques suivants (rappelons que l'on a une
identification canonique $(J_{1}^{\vee})^{\vee}\simeq J_1$):
\begin{eqnarray*} \mathcal{F}(i_{1,\ast}B)&=& \mathrm{R}\pi_{2,\ast}(\pi_{1}^{\ast}i_{1,\ast}B\otimes
\mathcal{L}_1)\\&\simeq & \mathrm{R}\pi_{2,\ast}((i_1\times
1_{J_1})_{\ast}pr_{X_1}^{\ast}B\otimes \mathcal{L}_1) \\ &\simeq &
\mathrm{R}pr_{J_1,\ast}(pr_{X_1}^{\ast}B\otimes\mathcal{L}_1).
\end{eqnarray*}
D'après le théorème \ref{existence}, $B$ admet un diviseur thêta,
on en déduit que $\mathcal{F}(i_{1,\ast}B)\simeq
\mathrm{R}^{1}pr_{J_1,\ast}(pr_{X_1}^{\ast}B\otimes
\mathcal{P}_1)[-1]=\QQ[-1]$ dans $\mathrm{D}^{b}_{c}(J_1)$.

\subsubsection*{Un accouplement sur $\QQ$.}

On traduit ci-après l'auto-dualité de $B$ (proposition
\ref{autodual}) pour la dualité de Serre en un accouplement entre
$(-1)^{\ast}\QQ$ et $\QQ$.

Fixons un isomorphisme $\alpha$ entre $B$ et
$\mathcal{H}om_{\OO_{X_1}}(B,\Omega^{1}_{X_1/S})$ via
l'accouplement anti-symétrique sur $B$. Par exemple, on peut
prendre
$$
\alpha:~~B\rightarrow
\mathcal{H}om_{\OO_{X_1}}(B,\Omega^{1}_{X_1/S})
$$
donné par $b\in B\mapsto (b,\cdot) ~\in
\mathcal{H}om_{\OO_{X_1}}(B,\Omega^{1}_{X_1/S})$. D'après le
théorème de dualité de Serre-Grothendieck (\cite{RD}), on a les
isomorphismes suivants:
\begin{eqnarray*} \mathrm{R}\mathcal{H}om(\QQ[-1],\OO_{J_1}) &\simeq &
\mathrm{R}\mathcal{H}om(\mathrm{R}pr_{J_1,\ast}(pr_{X_1}^{\ast}B\otimes \mathcal{P}_1),\mathcal{O}_{J_1}) \\
&\simeq &
\mathrm{R}pr_{J_1,\ast}\mathrm{R}\mathcal{H}om(pr_{X_1}^{\ast}B\otimes
\mathcal{P}_1, pr_{J_1}^{!}\mathcal{O}_{J_1})
\\ &\simeq &\mathrm{R}pr_{J_1,\ast}(\mathcal{H}om(pr_{X_1}^{\ast}B\otimes
\mathcal{P}_1, pr_{X_1}^{\ast}\Omega^{1}_{X_1/S}[1])
\\ &\simeq &
\mathrm{R}pr_{J_1,\ast}(pr_{X_1}^{\ast}\mathcal{H}om(B,\Omega^{1}_{X_1/S})\otimes
\mathcal{P}_{1}^{-1})[1] \\ & \simeq &
\mathrm{R}pr_{J_1,\ast}(pr_{X_1}^{\ast}B\otimes
\mathcal{P}_1^{-1})[1] \\
&\simeq &
(-1)^{\ast}\mathrm{R}pr_{J_1,\ast}(pr_{X_1}^{\ast}B\otimes
\mathcal{P}_1)[1]\\ &\simeq & (-1_{J_1})^{\ast}\QQ,
\end{eqnarray*}
où le cinquième isomorphisme est induit par $\alpha$. Notons
$\omega_{\Theta}=\OO_{J_1}(\Theta)|_{\Theta}$ le faisceau
canonique de $\Theta$. On a donc un isomorphisme
$\mathcal{H}om_{\OO_{\Theta}}(\QQ,\omega_{\Theta})\simeq
\mathcal{E}xt^{1}_{\OO_{J_1}}(\QQ,\OO_{J_1})\simeq
(-1)^{\ast}\QQ$, et un accouplement de $\OO_{\Theta}$-modules:
$$
\phi:~~~(-1_{\Theta})^{\ast}\QQ\otimes_{\OO_{\Theta}}\QQ\rightarrow
\omega_{\Theta}.
$$

\begin{remark} Cet accouplement dépend du choix de l'isomorphisme $\alpha:B\rightarrow
\mathcal{H}om_{\OO_{X_1}}(B,\Omega^{1}_{X_1/S})$. En fait, soit
$\alpha'$ l'isomorphisme entre $B$ et
$\mathcal{H}om_{\OO_{X_1}}(B,\Omega^{1}_{X_1/S})$ donné par $b\in
B\mapsto (\cdot,b)\in
\mathcal{H}om_{\OO_{X_1}}(B,\Omega^{1}_{X_1/S})$. L'accouplement
$\phi':(-1_{\Theta})^{\ast}\QQ\otimes_{\OO_{\Theta}}\QQ\rightarrow
\omega_{\Theta}$ défini à partir de $\alpha'$, est alors égal à
$-\phi$.
\end{remark}

\subsubsection{Le schéma de {\textquotedblleft Hilbert\textquotedblright} $\mathcal{H}$.}\label{le
schema H}

Dans ce numéro, supposons que la courbe propre lisse $X/S$ est
\emph{projective}, ce qui est automatique lorsque $X/S$ admet une
section ou lorsque $g\geq 2$. Fixons un faisceau inversible
relativement très ample $\OO_{X_1}(1)$ sur $X$. Supposons $S$
\textit{connexe}, et notons $d'$ le degré de
$\mathcal{O}_{X_1}(1)$ sur une fibre géométrique. De plus, pour
$m$ un entier, rappelons que $J^{[m]}$ est le schéma qui paramètre
les faisceaux inversibles de degré $m$ sur $X/S$, en particulier,
$J^{[m]}$ est un torseur sous $J$.

\subsubsection*{Définition du schéma $\mathcal{H}$.}

Soit $d$ un entier, posons $h_d(x)=(p-2)d'x+g-1-d\in
\mathbf{Z}[x]$. Considérons le foncteur
$$
\mathcal{Q}uot_{X_{1}/S,B}^{h_d(x)}: (\mathbf{Sch}/S)\rightarrow
\mathfrak{Ens}
$$
défini de la manière suivante: soit $T$ un $S$-schéma,
$\mathcal{Q}uot_{X_1/S,B}^{h_d(x)}(T)$ est alors l'ensemble des
classes d'isomorphisme de faisceaux quotients cohérents de
$B\otimes_{\OO_S}\OO_T$ plat sur $T$, ayant $h_d(x)$ comme
polynôme de Hilbert pour le faisceau relativement très ample
$\OO_{X_1}(1)$. D'après Grothendieck \cite{Quot}, ce foncteur est
représentable par un schéma propre
$\mathcal{H}_d=\mathrm{Quot}_{X_1/S,B}^{h_d(x)}$ sur $S$. Soient
$p_1:X_1\times_S\mathcal{H}_d\rightarrow X_1$,
$p_2:X_1\times_S\mathcal{H}_d\rightarrow \mathcal{H}_d$ les deux
projections naturelles. Soient $\mathcal{G}_d$ le faisceau
quotient cohérent universel de $p_{1}^{\ast}B$, $\mathcal{L}_d$ le
faisceau défini par la suite exacte suivante:
$$
0\rightarrow \mathcal{L}_d\rightarrow p_{1}^{\ast}B\rightarrow
\mathcal{G}_d\rightarrow 0.
$$
Comme $\mathcal{G}_d$ est plat sur le schéma $\mathcal{H}_d$,
$\mathcal{L}_d$ est plat sur $\mathcal{H}_d$, et sa formation
commute aux changements de base $T\rightarrow \mathcal{H}_d$.
Alors $\mathcal{L}_d$ est un faisceau inversible sur
$X_1\times_S\mathcal{H}_d$, de degré $d$ sur les fibres de
$X_1\times_S\mathcal{H}_d/\mathcal{H}_d$. D'où un morphisme
$\alpha_d:\mathcal{H}_d\rightarrow J_{1}^{[-d]}$ défini par
$\mathcal{L}_d$. On notera $\mathcal{H}=\mathcal{H}_{0}$
$\mathcal{G}=\mathcal{G}_0$, $\alpha=\alpha_0$ et
$\mathcal{L}=\mathcal{L}_0$ dans la suite.

\begin{remark} \label{fibre de H}Pour $x\in J_{1}^{[-d]}$
correspondant à un faisceau inversible $L$ de degré $-d$, la fibre
$\mathcal{H}_{d,x}$ de $\mathcal{H}_d$ au-dessus de $x$
s'identifie canoniquement à l'espace projectif des droites de
$H^{0}(X_1,B\otimes L)=\mathrm{Hom}(L^{-1},B)$.
\end{remark}

\subsubsection*{Lien entre $\mathcal{H}$ et le faisceau $\QQ$.}

Rappelons le point de vue fonctoriel introduit par Grothendieck.
Soient $S$ un schéma localement noethérien, $\mathcal{G}$ un
faisceau cohérent. On peut considérer le foncteur
$V_{\mathcal{G}}$ de la catégorie de $S$-schémas vers la
catégories d'ensembles, défini par $S'\in (\mathbf{Sch}/S)\mapsto
\mathrm{Hom}_{\OO_{S'}}(\mathcal{G}\otimes_{\OO_S}\OO_{S'},\OO_{S'})$.
Il est représenté par le $S$-schéma affine, spectre de l'algèbre
symétrique de $\mathcal{G}$. Nous l'appelons le {\textquotedblleft
fibré vectoriel\textquotedblright} associé à $\mathcal{G}$, bien
que $\mathcal{G}$ ne soit pas nécessairement localement libre.

Rappelons le résultat suivant de Grothendieck.

\begin{prop}[EGA $\mathrm{III}_{\mathrm{2}}$ 7.7.6] Soient $S$
un schéma localement noethérien, $f:X\rightarrow S$ un $S$-schéma
propre. Soit $\mathcal{F}$ un faisceau cohérent sur $X$ qui est
plat sur $S$. Considérons le foncteur $T$ de la catégorie de
$S$-schémas $Sch/S$ vers la catégorie des ensembles défini par
$T(S')=\Gamma(X\times_S S',\mathcal{F}\otimes_{\OO_S}\OO_{S'})$.

(1) Ce foncteur $T$ est représentable par un fibré vectoriel
$V_{\mathcal{R}}$ associé à un $\OO_S$-module cohérent
$\mathcal{R}$ sur $S$.

(2) Plus précisément, soit
$$
\xymatrix{L^{\bullet}: \cdots\ar[r]&  0\ar[r]& L^{0}\ar[r]^{u}&
L^{1}\ar[r]& L^{2}\ar[r] & \cdots}
$$
un complexe parfait de $\OO_S$-modules localement libres de rang
fini qui calcule universellement $\mathrm{R}f_{\ast}\mathcal{F}$,
et soit $\breve{u}:L^{1,\vee}\rightarrow L^{0,\vee}$ le dual de
$u$, alors on peut prendre
$\mathcal{R}=\mathrm{coker}(\breve{u})$. En particulier, si
$f:X\rightarrow S$ est une courbe plate et propre,
$\mathcal{R}=\mathcal{E}xt^{1}(\mathrm{R}^{1}f_{\ast}(\mathcal{F}),\OO_S)$.
\end{prop}

On en déduit l'interprétation du fibré projectif (au sens de
Grothendieck, \cite{EGA} EGA II définition 4.1.1)
$\mathbf{P}(\mathcal{R})$ associé à $\mathcal{R}$.

\begin{corollary} Avec les notations précédentes. Le fibré
projectif sur $S$ associé (au-sens de Grothendieck) au faisceau
cohérent $\mathcal{R}$ représente le foncteur
$T':(\mathbf{Sch}/S)\rightarrow \mathfrak{Ens}$, qui associe à un
$S$-schéma $S'$ l'ensemble des classes d'équivalence de couples
$(\mathcal{L}',i')$, où (i) $\mathcal{L}'$ est un faisceau
inversible sur $S'$, (ii) $i':f'^{\ast}\mathcal{L}'\rightarrow
\mathcal{F}'$ est universellement injectif (où $f':X\times_S
S'\rightarrow S'$ le morphisme naturel); (iii) On dit que
$(\mathcal{L}_1', i_1')$ et $(\mathcal{L}_2',i_2')$ sont
équivalents s'il existe un $\OO_{S'}$-morphisme
$u:\mathcal{L}_1'\rightarrow \mathcal{L}_2'$ rendant le diagramme
suivant commutatif
$$
\xymatrix{f'^{\ast}\mathcal{L}_1'\ar[d]_{f'^{\ast}u}\ar[r]^{i_1'}&
\mathcal{F}\otimes_{\OO_S}\OO_{S'}\\
f'^{\ast}\mathcal{L}_2'\ar[ru]_{i_2'}&}.
$$
\end{corollary}

Revenons au schéma $\mathcal{H}$ et supposons que $X/S$ admette
une section $\varepsilon\in X(S)$. Appliquons ces considérations
avec $S=J_1$, $X=X_1\times_S J_1$ et $\mathcal{F}=B\times
\mathcal{P}$ où $\mathcal{P}$ est un faisceau de Poincaré sur
$X_1\times_S J_1$, on trouve

\begin{prop} \label{H est fibre}Gardons les notations ci-dessus. Supposons que
$X/S$ admette une section $\varepsilon\in X(S)$.

(1) Le foncteur $T\in (\mathbf{Sch}/J_1)\mapsto
\Gamma(T\times_{J_1}(X_1\times_S J_1),B\otimes \mathcal{P}\otimes
\OO_{T})$ est représenté par le fibré vectoriel sur $J_1$ associé
au faisceau $\mathcal{R}=\mathcal{E}xt^{1}(\QQ,\OO_{J_1})\simeq
\mathcal{H}om(\QQ,\omega_{\Theta})$ ($\S$ \ref{accouplement}).

(2) Le foncteur $\mathcal{H}$, vu comme $J_1$-foncteur via le
morphisme $\alpha:\mathcal{H}\rightarrow J_1$, est représentable
par le fibré projectif associé à $\mathcal{R}$.
\end{prop}

\begin{corollary} Le morphisme $\alpha:\mathcal{H}\rightarrow J_1$
se factorise à travers $\Theta$.
\end{corollary}
\begin{proof} Par descente étale, on peut supposer que $X/S$
admette une section. On dispose donc du faisceau $\QQ$, et
$\mathcal{H}$ est le fibré projectif associé à $\QQ$. Comme $\QQ$
est à support dans $\Theta$, le morphisme
$\alpha:\mathcal{H}\rightarrow J_1$ se factorise à travers
$\Theta\hookrightarrow J_1$. D'où le résultat.
\end{proof}

\begin{remark} Lorsque $X/S$ n'a pas de section, le foncteur
$\mathcal{Q}uot_{X_{1}/S,B}^{h_d(x)}$ est toujours représentable
par un $J_1$-schéma, mais au lieu de trouver un fibré projectif,
on trouve un fibré de {\textquotedblleft
Severi-Brauer\textquotedblright}.
\end{remark}

\subsubsection{Majoration de la multiplicité des composantes de
$\Theta$.} ~\newline

Le résultat suivant est bien connu.

\begin{lemma}\label{lemme sur effectif} Soient $A$ une variété abélienne sur un corps $k$
algébriquement clos, $\mathcal{L}$ un faisceau inversible sur $A$.
Supposons qu'il existe un entier $n>0$ tel que $\mathcal{L}^{n}$
soit algébriquement équivalent à un faisceau inversible
$\OO_{A}(\Sigma)$ avec $\Sigma$ un diviseur effectif ou nul de
$A$, alors $\mathcal{L}$ est lui-même algébriquement équivalent à
un faisceau inversible $\OO_{A}(\Sigma')$ avec $\Sigma'$ effectif
ou nul.
\end{lemma}

\begin{proof} Si $\mathcal{L}^{n}$ est algébriquement équivalent à $\OO_{A}$,
$\mathcal{L}$ l'est aussi. On peut supposer que $\Sigma$ est un
diviseur effectif non nul. Comme $\mathrm{Pic}^{\circ}_{A/k}$ est
$n$-divisible, quitte à remplacer $\mathcal{L}$ par un faisceau
inversible algébriquement équivalent à $\mathcal{L}$, on peut
supposer que $\mathcal{L}^{n}\simeq \OO_A(\Sigma)$. Soit $B=\{x\in
A| T_{x}^{\ast}\mathcal{L}\simeq
\mathcal{L}\}^{\circ}_{\mathrm{red}}$, qui est une sous-variété
abélienne de $A$. D'après lemme 1 de \cite{Kempf} (et la preuve du
théorème 1 de \cite{Kempf}), on sait que
$\mathcal{L}^{n}|_{B}\simeq \OO_B$. Comme $B$ est une sous-variété
abélienne de $A$, le morphisme $\Pic^{\circ}_{A/k}\rightarrow
\Pic^{\circ}_{B/k}$ est surjectif. Quitte à tensoriser
$\mathcal{L}$ par un faisceau inversible d'ordre $n$ de $A$, on
peut supposer que $\mathcal{L}|_{B}\simeq \OO_B$. Il y a donc un
faisceau inversible non dégénéré $\mathcal{M}$ sur $C:=A/B$ tel
que $\mathcal{L}=\pi^{\ast}\mathcal{M}$ (ici, $\pi:A\rightarrow C$
est le morphisme canonique). Or $\mathcal{M}$ et $\mathcal{M}^{n}$
ont le même indice (\cite{AV} page $159$, corollaire), on déduit
du théorème 1 de \cite{Kempf} que $H^{i}(A,\mathcal{L})\neq 0$ si
et seulement si $H^{i}(A,\mathcal{L}^{n})\neq 0$. En particulier,
$H^{0}(A,\mathcal{L})\neq 0$. D'où le résultat.

\end{proof}

\begin{prop}\label{polarisation} Gardons les notations dans
la section précédente. Soit $\alpha$ (resp. $\beta$) la classe
dans $\mathrm{NS}(J_1)$ d'un diviseur effectif non trivial de
$J_1$. Notons $\mathrm{cl}(\Theta)$ la classe de $\Theta$ dans
$\mathrm{NS}(J_1)$. Supposons $\mathrm{cl}(\Theta)=a\alpha+b\beta$
avec $a,~b$ deux entiers positifs. Alors (1) $a\leq p-1$ et $b\leq
p-1$; (2) Si $a=p-1$ (resp. si $b=p-1$), alors
$\alpha=\mathrm{cl}(\Theta_{\mathrm{class}})$ (resp.
$\beta=\mathrm{cl}(\Theta_{\mathrm{class}})$) et $b=0$ (resp.
$a=0$), où $\mathrm{cl}(\Theta_{\mathrm{class}})$ désigne la
classe du diviseur thêta classique $\Theta_{\mathrm{class}}$ dans
$\mathrm{NS}(J_1)$.
\end{prop}

\begin{proof} Supposons $a\geq p$, alors $a-p+1\geq 1$.
Comme
$\mathrm{cl}(\Theta)=(p-1)\cdot\mathrm{cl}(\Theta_{\mathrm{class}})$
dans $\mathrm{NS}(J_1)$, on a
$$
(p-1)(\mathrm{cl}(\Theta_{\mathrm{class}})-\alpha)=(a-p+1)\alpha+b\beta.
$$
Comme $a-p+1\geq 1$, $(a-p+1)\alpha+b\beta$ est la classe d'un
diviseur effectif non trivial dans $\mathrm{NS}(J_1)$. Donc
$\mathrm{cl}(\Theta_{\mathrm{class}})-\alpha$ peut être représenté
par $\OO_{J_1}(D)$ avec $D$ un diviseur effectif non trivial de
$J_1$. On en déduit que $\Theta_{\mathrm{class}}$ est
algébriquement équivalent à $D+D'$ avec $D$ et $D'$ deux diviseurs
effectifs non triviaux de $J_1$. Donc un translaté de
$\Theta_{\mathrm{class}}$ s'écirt comme une somme de deux
diviseurs effectifs non triviaux. D'où une contradiction avec le
fait que la polarisation principale $\Theta_{\mathrm{class}}$ pour
une courbe lisse propre et connexe est irréductible (en fait, il
existe un morphisme birationnel
$X^{(g-1)}\rightarrow\Theta_{\mathrm{class}}$, où $X^{(g-1)}$
désigne le $(g-1)$-ieme produit symétrique de $X/k$. Par suite,
$\Theta_{\mathrm{class}}$ est intègre, même normal). Ceci montre
que $a\leq p-1$. De même, on a $b\leq p-1$. Si on suppose $a=p-1$,
le même raisonnement nous montre que $b=0$, la deuxième assertion
s'en déduit.
\end{proof}

\begin{remark}\label{cas p=3 facile} Supposons $p=3$. Alors on déduit de proposition
\ref{polarisation} que ou bien $\Theta$ est réduit, ou bien
$\Theta$ est le double d'un diviseur thêta classique. On verra
plus loin ($\S$ \ref{Theta est reduit}) que pour $g\geq 2$, ce
dernier cas est exclu.
\end{remark}

\subsubsection{{\textquotedblleft La propriété de Dirac\textquotedblright}.}\label{propriete de Dirac}

Dans cette section, $k$ est un corps de caractéristique $p>0$.

\subsubsection*{Généralités.}

Soit $R$ une $k$-algèbre locale de longueur finie. Le \emph{socle}
de $R$ est le plus grand idéal de $A$ annulé par l'idéal maximal
de $R$. L'algèbre $R$ est de Gorenstein si et seulement si son
socle est engendré par un élément (\cite{Eisenbud} proposition
21.5).

Soient $G$ un schéma en groupes fini commutatif sur un corps $k$,
et $R$ son algèbre de fonctions. Alors $R$ est une intersection
complète, et donc est de Gorenstein (\cite{Eisenbud} corollary
21.19). Par exemple, si $k$ est parfait et si $G$ est
infinitésimal, il existe des entiers $n_1,\cdots,n_r\geq 1$ tels
que $R\simeq
k[X_1,\cdots,X_r]/(X_{1}^{p^{n_1}},\cdots,X_{r}^{p^{n_r}})$
(\cite{Demazure-Gabriel} III $\S$ 3, $n^{\circ}$ 6, corollaire
6.3). Le socle de $R$ est alors l'idéal principal engendré par
$\prod_{i=1}^{r}x_{i}^{p^{n_i}-1}\in R$, où $x_i$ est l'image de
$X_i$ dans $R$.

\begin{definition}[\cite{R2}, définition 1.1.1] Une
\emph{{\textquotedblleft fonction de Dirac\textquotedblright}} sur $G$ est un
élément $\lambda\in R$ tel que (i) l'image de $\lambda$ dans
$R_{x}$ soit nulle pour $x\in G(k)-\{0\}$ et (ii) $\lambda$ induit
un élément non nul du socle de $G$ à l'origine.
\end{definition}

\begin{remark} Si $f$ est une fonction de Dirac sur $G$, et si $k'$ est une
extension de $k$, posons $G'=G\times_k k'$, alors $f$ est aussi
une fonction de Dirac sur $G'$.
\end{remark}

Soit $f$ une telle fonction, elle peut être vue comme un morphisme
de schémas $f:G\rightarrow \mathbf{A}^{1}_{k}$. Alors
$Z=V(f)\subset G$ est le plus grand sous-schéma fermé de $G$ où
$f$ s'annule. De plus, il est clair que $Z-\{0\}=G-\{0\}$, et si
$W$ est un sous-schéma fermé de $G$ tel que $Z\subset W\subset G$,
alors ou bien $Z=W$, ou bien $W=G$, c'est-à-dire, $Z$ est un
sous-schéma fermé maximal de $G$ distinct de $G$. De plus, ces
deux propriétés caractérisent $f$ à une multiplication par un
$k-$scalaire près. Si $G=\mathrm{Spec}(R)$ avec $R$ une
$k$-algèbre finie, on a
$\mathrm{dim}_k(R/(f))=\mathrm{dim}_k(R)-1$.

\begin{lemma}\label{critere de Dirac} Soient $G$ un schéma en groupes fini commutatif sur $k$, $f$ une
fonction de Dirac sur $G$, alors le plus grand sous-schéma en
groupes qui laisse invariant $f$ est trivial.
\end{lemma}

\begin{proof} Soit $N=\mathrm{Spec}(R')$ le plus grand sous-schéma
en groupes de $G$ qui laisse invariant $f$. Alors $Z=V(f)\subset
G$ est invariant par $N$. Donc $\mathrm{dim}_k R/(f)$ est un
multiple de $\mathrm{dim}_k N$. Or
$\mathrm{dim}_k(R/(f))=\mathrm{dim}_k(R)-1$, et
$\mathrm{dim}_k(R)$ et un multiple de $\mathrm{dim}_k R'$. Ce qui
implique que $\mathrm{dim}_k(R')=1$ sauf dans le cas où
$\mathrm{dim}_k(R)=1$, mais la conclusion est triviale dans ce
cas.
\end{proof}

\begin{definition}\label{def de dirac} Soit $A$ une variété semi-abélienne sur un corps $k$
parfait de caractéristique $p>0$. Soit $D$ un diviseur effectif de
$A$. On dit que $D$ satisfait à \emph{la {\textquotedblleft
propriété de Dirac\textquotedblright}} si toute équation locale de
$D|_{\ker(V)}$ est une fonction de Dirac sur $\ker(V)$, noyau du
Verschiebung $V:A\rightarrow A_{-1}$ (où $A_{-1}$ est l'image
réciproque de $A$ par l'inverse de Frobenius absolu sur
$\mathrm{Spec}(k)$).
\end{definition}

\begin{prop} \label{Dirac implique ample}Soient $A$ une variété abélienne sur un corps
$k$ parfait de caractéristique $p>0$, et $D$ un diviseur effectif
de $A$ satisfaisant à la propriété de Dirac, alors $D$ est ample.
\end{prop}

\begin{proof} Soit $A'$ la plus grande sous-variété abélienne de
$A$ qui laisse stable $D$. Il suffit de montrer que $A'$ est
triviale. Comme $\ker(V|_{A'})$ laisse stable $D|_{\ker(V)}$,
$\ker(V|_{A'})$ est trivial, d'après le lemme \ref{critere de
Dirac}.
\end{proof}

\begin{corollary}\label{Dirac implique ordinaire} Soient $A$
une variété abélienne sur un corps parfait $k$, $D$ un diviseur
effectif de $A$ satisfaisant à la propriété de Dirac. Notons $D_0$
la somme des composantes de $D$ passant par $0\in A$. Supposons
que $D_0$ n'est pas trivial. Soit $A_0$ la plus grande
sous-variété laissant stable $D_0$, alors $A_0$ est ordinaire.
\end{corollary}

\begin{proof}
En fait, $(\ker(V|_{A_0}))^{\circ}$ est un sous-groupe de $(\ker
V)^{\circ}$ qui laisse stable $D|_{\left(\ker
(V)\right)^{\circ}}=D_0|_{\left(\ker(V)\right)^{\circ}}$, d'après
le lemme précédent, on en déduit que $(\ker
(V|_{A_0}))^{\circ}=0$, c'est-à-dire que $\ker( V|_{A_{0}})$ est
étale, donc $A_0$ est ordinaire.
\end{proof}

Revenons au diviseur $\Theta$,

\begin{theorem}[Propriété de Dirac pour $\Theta$, \cite{R2}
$\S$ 1.1] \label{Dirac}Soient $k$ un corps algébriquement clos de
caractéristique $p>0$, $X/k$ une courbe propre semi-stable. Le
diviseur thêta $\Theta\subset J_1$ satisfait à la propriété de
Dirac.
\end{theorem}

\begin{remark} Rappelons que c'est en quelque sorte cette  {\textquotedblleft propriété de
Dirac\textquotedblright} qui permet de prouver l'existence de
$\Theta$.
\end{remark}

\subsubsection*{Application au cas de genre $1$.}

Pour les courbes elliptiques, on a les résultats suivants dus à
Tango:

\begin{theorem}[\cite{Tango} corollary 8] Soit $E$ une courbe elliptique sur $k$ algébriquement
clos de caractéristique $p>0$.

(1) Si $E$ est ordinaire, on a un isomorphisme $B\simeq
\oplus_{1}^{p-1}\mathcal{L}_i$, où $\{\mathcal{L}_{i}| 1\leq i\leq
p-1\}\cup \{\mathcal{O}_{E_1}\}=J_{E_1}[p](k)$.

(2) Si $E$ n'est pas ordinaire, $B\simeq F_{p-1}$, où $F_{p-1}$
est le seul fibré semi-stable de rang $p-1$ et de pente $0$ de
$E_1$ admettant une section non-triviale (qui est obtenu comme
extension successive de $\OO_{E_1}$).
\end{theorem}

Comme corollaire, on obtient

\begin{corollary}\label{thetaforell} Soit $E/k$ une courbe elliptique sur $k$
algébriquement clos de caractéristique $p>0$, alors

(1) Si $E$ est ordinaire, $\Theta$ est un schéma réduit concentré
aux points d'ordre $p$ de $\mathrm{Pic}^{0}_{E_1/k}$.

(2) Si $E$ est supersingulière, $\Theta$ est un schéma concentré
au seul point $0\in \mathrm{Pic}^{0}_{E_1/k}$ avec une
multiplicité $p-1$.
\end{corollary}

On donne une démonstration directe suivante, en utilisant la
propriété de Dirac pour $\Theta$.\newline

\begin{proof}[Démonstration de \ref{thetaforell}] On sait que $\Theta$
est un diviseur effectif de degré $p-1$. (1) Si $E$ est ordinaire,
$\ker(V:E_1\rightarrow E)$ est un schéma en groupes fini étale de
degré $p$. Donc $\Theta=\ker(V)-\{0\}$. (2) Si $E$ est
supersingulière, on sait que $\Theta$ passe par l'origine. De
plus, soit $f\in \OO_{J_{E_1},0}$ une fonction locale de $\Theta$
en origine $0\in J_{E_1}$. Sa restriction à $\OO_{\ker(V)}\simeq
k[t]/(t^{p})$ engendre le socle de $k[t]/(t^{p})$ (où $t$ est une
uniformisante de $\OO_{J_{E_1},0}$). Donc $f\equiv t^{p-1}
(\mathrm{mod}~t^{p})$. En particulier, $\Theta$ est de
multiplicité $p-1$ en $0\in J_{E_1}$. Donc $\Theta$ est un schéma
à support le seul point $0\in J_{E_1}$ avec une multiplicité
$p-1$.
\end{proof}

\begin{remark}
On peut décrire la variation de $\Theta$ au voisinage d'un point
supersingulier dans l'espace de modules des courbes elliptiques.
Soit $E$ une $k$-courbe elliptique supersingulière,
$\mathcal{E}/k[[t]]$ une déformation verselle de $E/k$ sur
$k[[t]]$. D'après Igusa (\cite{KatzMazur} cor. 12.4.4), l'algèbre
de Lie de $\mathcal{E}/k[[t]]$ admet une base $e$ telle que
$e^{(p)}=te$. Soit $G=\ker(F:\mathcal{E}\rightarrow
\mathcal{E}_1)$, $G^{\vee}$ son dual de Cartier, alors $G^{\vee}$
admet pour équation $x^{p}-tx=0$. Et on a une suite exacte
$$
0\rightarrow G\rightarrow \mathcal{E}[p]\rightarrow
G^{\vee}\rightarrow 0.
$$
Alors $G^{\vee}\simeq \ker(V)\subset \mathcal{E}_1$ admet pour
équation $x^{p}-tx=0$. Par suite $\Theta$ a pour équation
$x^{p-1}-t=0$.
\end{remark}

\subsubsection{Un énoncé de pureté.}

Le résultat suivant est bien connu pour $\Theta_{\mathrm{class}}$,
et s'étend aux diviseurs thêta de fibrés vectoriels.

\begin{prop} Considérons $X/S$ une courbe lisse admettant
une section $\varepsilon\in X(S)$, avec $S$ un schéma local
noethérien régulier de point fermé $s$. Soit $\xi$ un point
générique du fermé $W$ de $\Theta$ où $\QQ$ n'est pas inversible
(on renvoie à $\S$ \ref{accouplement} pour la définition de
$\QQ$). Alors on a

(1)
$\mathrm{Codim}_{\xi}(W,\Theta)=\mathrm{dim}(\OO_{\Theta,\xi})\leq
3$;

(2) Si de plus $\xi\in \Theta_s$, et $\Theta_s$ est
géométriquement réduit en $\xi$, on a
$\mathrm{dim}(\OO_{\Theta,\xi})\geq 1$;

(3) Si de plus $\xi\in \Theta_s$, et $\Theta_s$ est
géométriquement normal en $\xi$, on a
$\mathrm{dim}(\OO_{\Theta,\xi})\geq 2$.
\end{prop}

\begin{proof} (1) On raisonne par l'absurde. Supposons que
$\mathrm{dim}(\OO_{\Theta,\xi})\geq 4$. Comme $S$ est régulier,
$\QQ$ est de Cohen-Macaulay (lemme \ref{lemme trivial}). Donc
$\mathrm{prof}_{\xi}(\QQ)\geq 4\geq 2$. De plus, comme $S$ est
régulier, $\OO_{\Theta,\xi}$ est un anneau local d'intersection
complète de dimension $\geq 4$. D'après le théorème
d'Auslander-Buchsbaum (SGA $2$, exposé XI, théorème $3.13(2)$), le
couple $(\mathrm{Spec}(\OO_{\Theta,\xi}),\xi)$ est parafactoriel.
Donc le faisceau inversible
$\QQ|_{\mathrm{Spec}(\OO_{\Theta,\xi})-\{\xi\}}$ est trivial.
Compte tenu de la propriété de profondeur de $\QQ$ en $\xi$, on en
déduit que $\QQ$ est inversible en $\xi$, d'où une contradiction.
Montrons (2), supposons que $\mathrm{dim}(\OO_{\Theta,\xi})<1$,
alors $\mathrm{dim}(\OO_{\Theta_s,\xi})\leq
\mathrm{dim}(\OO_{\Theta,\xi})=0$. Donc $\xi$ est un point
générique de $\Theta_s$. Or $\Theta_s$ est géométriquement réduit,
$\QQ_s$ est génériquement inversible sur $\Theta_s$ en $\xi$ en
vertu du critère de lissité (lemme \ref{critere de lissite},
utilisé dans le sens trivial) dans chapitre $2$, donc $\QQ$ est
inversible sur $\Theta$ en $\xi$ (\ref{lemme trivial}), d'où une
contradiction. Un raisonnement similaire nous montre (3). Ceci
finit la preuve.
\end{proof}

\begin{prop}\label{purete} Soit $X/k$ une courbe lisse de genre $g\geq 4$.
Alors $\QQ$ n'est pas inversible. En particulier, $\Theta$ est
singulier, et le lieu singulier est de codimension $\leq 3$ dans
$\Theta$.
\end{prop}

\begin{proof} D'après la proposition précédente et le critère de lissité (\ref{critere de lissite}),
il suffit de montrer que $W\neq \emptyset$, et il suffit donc de
traiter le cas où $X/k$ est générique. On raisonne par l'absurde.
Sinon $\QQ$ est inversible sur $\Theta$. Comme $g\geq 4$, le
morphisme naturel $\mathrm{Pic}(J_1)\rightarrow
\mathrm{Pic}(\Theta)$ est un isomorphisme (\cite{SGA2} Exposé XII,
corollaire 3.16). Donc $\QQ$ provient d'un faisceau inversible
$\mathcal{L}$ sur $J_1$. Comme $(-1_{\Theta})^{\ast}\QQ\otimes
\QQ\simeq \omega_{\Theta}$, on obtient que
$(-1)^{\ast}\mathcal{L}\otimes \mathcal{L}\simeq
\OO_{J_1}(\Theta)$. Par ailleurs, on a la suite exacte suivante:
$$
0\rightarrow \OO_{J_1}(-\Theta)\rightarrow \OO_{J_1}\rightarrow
\OO_{\Theta}\rightarrow 0,
$$
d'où la suite exacte
$$
0\rightarrow \OO_{J_1}(-\Theta)\otimes \mathcal{L}\rightarrow
\mathcal{L}\rightarrow \mathcal{Q}\rightarrow 0. ~~~~~~~~(1)
$$
Comme $X$ est une courbe générique, le groupe de Néron-Severi
$\mathrm{NS}(J_1)\simeq \mathbf{Z}$ avec la classe du diviseur
thêta classique $\Theta_{\mathrm{class}}$ comme générateur
(\cite{Mori}). Supposons que $\mathcal{L}$ est algébriquement
équivalent à $\OO_{J_1}(r\cdot \Theta_{\mathrm{class}})$ pour un
entier $r\in \mathbf{Z}$, alors $(-1)^{\ast}\mathcal{L}$ est
algébriquement équivalent à $\mathcal{O}_{J_1}(r\cdot
\Theta_{\mathrm{class}})$, donc $2r=p-1$, d'où une contradiction
dans le cas où $p=2$. Supposons désormais $p\geq 3$, alors
$r=(p-1)/2$. Donc $(\OO_{J_1}(-\Theta)\otimes \mathcal{L})^{-1}$
et $\mathcal{L}$ sont amples. Appliquant la transformation de
Fourier-Mukai à la suite exacte (1) ci-dessus, et puisque $J_1$
est de dimension $g\geq 4$,  on obtient une suite exacte:
$$
0=\mathcal{F}^{0}(\OO_{J_1}(-\Theta)\otimes
\mathcal{L})\rightarrow \mathcal{F}^{0} \mathcal{L}\rightarrow
\mathcal{F}^0 \QQ\rightarrow
\mathcal{F}^{1}(\OO_{J_1}(-\Theta)\otimes \mathcal{L})=0.
$$
En particulier, $\mathcal{F}^{0}(\QQ)\neq 0$. Or par la propriété
de la transformation de Fourier-Mukai (\cite{Mukai} theorem 2.2),
$\mathcal{F}\circ \mathcal{F} (B)=\mathcal{F}(\QQ)[-1]\simeq
(-1)^{\ast} B[-g]$. Donc $\mathcal{F}(\QQ)=(-1)^{\ast}B[1-g]$.
Puisque $g\geq 4$, $\mathcal{F}^{0}(\QQ)=0$, d'où une
contradiction.
\end{proof}

\subsubsection{Les sous-fibrés de $B$.}

Cette section est consacrée à l'étude des sous-fibrés de $B$, en
particulier, on donne une nouvelle preuve du fait que $B$ est un
fibré vectoriel stable lorsque $g\geq 2$, qui est plus directe que
celle de Joshi (\cite{Joshi}).

Dans cette section, $k$ désigne un corps algébriquement clos de
caractéristique $p>0$.

\subsubsection*{La stabilité de $B$ pour $g\geq 2$.}

L'existence du diviseur thêta entraîne déjà la semi-stabilité du
fibré vectoriel $B$ pour une courbe $X/k$ lisse de genre $g\geq
1$. De plus, si le diviseur thêta $\Theta$ (pour $B$) est
géométriquement intègre, le fibré $B$ est alors stable. Mais, on
ignore si $\Theta$ est intègre en toute généralité. K. Joshi a
donné une preuve directe de la stabilité de $B$ dans \cite{Joshi}.
On donne ci-après une nouvelle preuve, qui fournit aussi une borne
naturelle sur la pente des sous-fibrés de $B$.

\begin{lemma} Soient $K$ un corps de caractéristique $p$,
$K'$ une extension radicielle de degré $p$. Posons $E=K'/K$, qui
est donc un $K$-espace vectoriel de dimension $p-1$. Soit $I$ le
noyau du morphisme $K'\otimes_K K'\rightarrow K'$ défini par
$a\otimes b\mapsto ab$.

(1): L'application $K$-linéaire composée $\phi:I\rightarrow
K'\otimes_K K'\rightarrow E\otimes_K K'$ est bijective (ici, pour
$\alpha=a\otimes b\in K'\otimes_K K'$, et $\lambda\in K'$,
$\alpha\cdot \lambda:=a\otimes \lambda b$).

(2): Notons $\mathrm{Fil}_{i}=\phi(I^{i+1})$ pour
$i=0,\cdots,p-1$, les $\mathrm{Fil}_{i}$ forment une filtration
décroissante de $E\otimes_K K'$:
$$
0=\mathrm{Fil}_{p-1}\subset \mathrm{Fil}_{p-2}\subset\cdots\subset
\mathrm{Fil}_1\subset \mathrm{Fil}_0=E\otimes_K K'.
$$
Soit $V\subset E$ un $K$-sous-espace de $E$ de dimension $n$,
alors le morphisme naturel $V\otimes_K K'\rightarrow E\otimes_K
K'\rightarrow  (E\otimes_{K}K')/\mathrm{Fil}_{n}$ est injectif.
\end{lemma}

\begin{proof} (1): Soient $e\in E=K'/K$ un élément non nul de
$E$, $\widetilde{e}$ un relèvement de $e$ dans $K'$. Considérons
$\widetilde{e}\otimes 1-1\otimes \widetilde{e}\in I$, alors
$\phi(\widetilde{e}\otimes 1-1\otimes \widetilde{e})=e\otimes 1$.
Donc $\phi$ est surjectif. Or, vu comme $K'$-espace vectoriel, $I$
et $E\otimes_K K'$ sont de dimension $p-1$, ce qui implique que le
morphisme surjectif $\phi$ est en fait bijectif.

(2) Clairement, les $\mathrm{Fil}_{i}$ ($1\leq i\leq p$) forment
une filtration décroissante de $E\otimes_K K'$. Il reste à
vérifier la deuxième assertion. On raisonne par récurrence sur
$n$. Commençons avec le cas où $n=1$. Soient $e\in E=K'/K$ un
élément non nul, $\tilde{e}\in K'$ un relèvement de $e$ dans $K'$.
Alors il engendre $K'$ comme $K$-algèbre. Comme
$\phi^{-1}(\tilde{e}\otimes 1)=\tilde{e}\otimes 1-1\otimes
\tilde{e}$, il s'agit de montrer $\tilde{e}\notin I^{2} $. Puisque
$\tilde{e}$ est un générateur de $K'$ comme $K$-algèbre, les
éléments $\tilde{e}^{i}\otimes 1-1\otimes \tilde{e}^{i} ~(1\leq
i\leq p-1)$ forment une $K'$ base de $I$. Si $\tilde{e}\otimes
1-1\otimes \tilde{e}\in I^{2}$, ceci implique que
$\tilde{e}^{i}\otimes -1\otimes \tilde{e}^{i}\in (\tilde{e}\otimes
1-1\otimes \tilde{e})(K'\otimes_K K')\subset I^{2}$. On obtient
donc $I^{2}=I$, ce qui implique que $I=0$, d'où une contradiction.
Supposons ensuite que l'assertion est vérifiée pour $0\leq n\leq
r$ avec $r\geq 0$ un entier. Soit $\{e_1,\cdots,e_{r+1}\}$ une
$K$-base de $E$. Soit $\widetilde{e_i}$ un relèvement de $e_i$
dans $K'$. Comme $K'$ est une extension de degré $p$,
$\widetilde{e_{1}}$ est un générateur de $K'$ comme $K$-algèbre.
Posons $t=\widetilde{e_1}$. Les $\widetilde{e_i}$ pour $i>1$
peuvent donc s'écrire sous la forme $\widetilde{e_i}=P_i(t)$ avec
$P_{i}(T)\in K[T]$ un polynôme de degré $2\leq d_i \leq p-1$.
Quitte à changer la base $\{e_i\}$ de $V$, on peut supposer que
les $d_i$ sont deux à deux distincts. Soit
$e=\sum_{i=1}^{r+1}e_i\lambda_i$ un élément de $E\otimes_K K'$ qui
est dans le noyau du morphisme naturel $V\otimes_K K'\rightarrow
(E\otimes_K K')/\mathrm{Fil}_{r+1}$. Par hypothèse de récurrence,
on peut supposer que $\lambda_1\neq 0$.
$\phi^{-1}(e)=\sum_{i=1}^{r+1}(e_i\otimes 1-1\otimes
e_i)\lambda_i$ peut donc s'écrire sous la forme suivante:
$$
\phi^{-1}(e)=\sum_{j=1}^{m}\left(\prod_{k=1}^{r+2}(a_k\otimes
1-1\otimes a_k)\right)\cdot \mu_j
$$
avec des $a_k\in K'$ Soit $\partial:K'\rightarrow K'$ la
$K$-dérivation telle que $\partial(t)=1$. Alors $\partial$ s'étend
par extension des scalaires en une $K'$-dérivation
$\partial\otimes K': K'\otimes_K K'\rightarrow K'\otimes_K K'$. En
appliquant $\partial\otimes K'$, on obtient que $(\partial\otimes
K')(e)\in I^{r+1}$. D'ailleurs, $(\partial\otimes
K')(e)=\lambda_1+\sum_{i=2}^{r+1}P'_{i}(t)\otimes \lambda_i$. Donc
$\phi((\partial\otimes
K')(e))=\sum_{i=2}^{r+1}\overline{P'_{i}(t)}\otimes \lambda_i$ (où
$\overline{P_{i}'(t)}$ désigne la réduction de $P_{i}'(t)$ modulo
$K$). Comme les $P'_{i}[T]$ sont des polynômes de degré $1\leq
d_i'\leq p-1$, et les $d_i'$ sont deux à deux distincts, les
$\overline{P_{i}'[t]}$ (pour $2\leq i\leq r+1$) engendrent un
$K$-sous-espace de dimension $r$. On a donc
$\lambda_2=\cdots=\lambda_{r+1}=0$ en vertu de l'hypothèse de
récurrence. Finalement $e_1\in \mathrm{Fil}_{r+1}\subset
\mathrm{Fil}_1$, d'où une contradiction.
\end{proof}
\vspace{2mm}

Soient $k$ un corps algébriquement clos de caractéristique $p>0$,
$X$ une courbe propre lisse connexe sur $k$. Comme d'habitude,
soit $X_1$ l'image réciproque de $X$ par le Frobenius absolu de
$\mathrm{Spec}(k)$, $F:X\rightarrow X_1$ le Frobenius relatif.
Alors $F^{\ast}B$ admet une filtration canonique décroissante.
Rappelons la construction donnée dans \cite{R1}. Considérons le
diagramme cartésien:
$$
\xymatrix{X\times_{X_1}X\ar[r]^>>>>>{p_1}\ar[d]^{p_2} & X\ar[d]^{F}\\
X\ar[r]^{F} & X_1}.
$$
L'image réciproque de la suite exacte
$$
0\rightarrow \OO_{X_1}\rightarrow F_{\ast}\OO_{X}\rightarrow
B\rightarrow 0
$$
par $F$ donne la suite exacte suivante:
$$
0\rightarrow \OO_{X}\rightarrow
p_{1,\ast}(\OO_{X\times_{X_1}X})\rightarrow F^{\ast}B\rightarrow
0.
$$
Mais $X\times_{X_1}X$ est le voisinage infinitésimal d'ordre $p-1$
de la diagonale de $X\times_k X$ et $F^{\ast}(B)$ est l'idéal
augmentation définissant la diagonale $X$. Par suite, $F^{\ast}B$
admet une filtration canonique par des sous-fibrés sur $X$:
$$
0=B_p\subset B_{p-1}\subset \cdots\subset B_1\subset B_0=B,
$$
de quotients successifs $B_i/B_{i+1}\simeq
\Omega^{\otimes^i}_{X_1/k}$. En particulier, si $g\geq 2$,
$F^{\ast}B$ n'est pas semi-stable.

\begin{theorem} Soit $V\subset B$ un sous-faisceau cohérent de rang $r$ de
$B$. Alors $F^{\ast}V$ se plonge dans $F^{\ast}B/B_r$ via le
morphisme naturel $F^{\ast}V\rightarrow F^{\ast}B\rightarrow
F^{\ast}B/B_{r}$. Et par suite, $\mathrm{deg}(V)\leq
\frac{r(r+1)(g-1)}{p}$.
\end{theorem}

\begin{proof} Comme $X$ est une courbe réduite, et que $V$ est un fibré
sur $X_1$, il suffit de montrer que le morphisme
$F^{\ast}V\rightarrow F^{\ast}B/B_{n}$ est injectif en point
générique de $X_1$, ce qui résulte du lemme précédent.
\end{proof}

\begin{corollary} Le fibré vectoriel $B$ est stable pour $g\geq 2$.
\end{corollary}

\begin{proof} Soit $V$ un sous-fibré de rang $r\leq p-2$ de $B$. D'après
le théorème ci-dessus, on sait que $\mathrm{deg}(V)\leq
\frac{r(r+1)(g-1)}{p}$. Notons $\lambda(V)$ la pente de $V$. Comme
$g\geq 2$, on a $\lambda(V)\leq \frac{(r+1)(g-1)}{p}\leq
\frac{(p-1)(g-1)}{p}< g-1$. D'où le résultat.
\end{proof}

\subsubsection*{Utilisation des puissances symétriques.}\label{puissance sym}

Soit $L\subset B$ un sous-faisceau inversible de $B$, on va
construire des sous-faisceaux de $B$ de rang supérieur à l'aide de
la structure multiplicative de $F_{\ast}(\OO_{X})$. Soit
$\pi:F_{\ast}(\OO_X)\rightarrow B$ le morphisme de quotient, et
soit $E=\pi^{-1}(L)\subset F_{\ast}(\OO_X)$, qui est un
sous-faisceau de rang $2$ de $F_{\ast}(\OO_X)$ sur $X_1$. On a une
suite exacte
\begin{eqnarray}\label{fibre E} 0\rightarrow \OO_{X_1}\rightarrow
E\rightarrow L\rightarrow 0
\end{eqnarray}
où $E$ est un fibré vectoriel de rang $2$ sur $X_1$. Or
$F_{\ast}(\OO_{X})$ est un faisceau en algèbres, la structure
multiplicative nous donne une flèche
$\mathrm{Sym}^{i}_{\OO_{X_1}}E\rightarrow F_{\ast}(\OO_{X})$, qui
est une injection pour $0\leq i\leq p-1$. Notons $E_i$ l'image de
$\mathrm{Sym}^{i}_{\OO_{X}}E$ dans $F_{\ast}(\OO_X)$, qui est donc
un sous-faisceau de rang $i+1$. D'où une filtration de
$F_{\ast}(\OO_X)$ par ses sous-faisceaux
$$
0\subset E_0\subset E_1\subset E_2\subset \cdots\subset
E_{p-1}\subset F_{\ast}(\OO_{X})
$$
de quotients successifs $E_{i}/E_{i-1}\simeq
L^{\otimes^{i}}~(0\leq i\leq p-1)$. Par passage au quotient, on
obtient des sous-faisceaux $\mathcal{F}_i:=\pi(E_i)$ de $B$ et une
filtration
$$
0=\mathcal{F}_0\subset \mathcal{F}_1(=L)\subset
\mathcal{F}_2\subset \cdots\subset \mathcal{F}_{p-1}\subset B
$$
de quotients successifs $\mathcal{F}_{i}/\mathcal{F}_{i-1}\simeq
L^{\otimes^{i}} (1\leq i\leq p-1)$

\subsubsection*{Courbes de Tango.}

On va donner des exemples de courbes qui montrent que la borne
donnée ci-dessus pour les degrés des sous-faisceaux de $B$ est la
meilleure possible.
\begin{definition} Une courbe lisse connexe $X/k$ est dite \emph{de
Tango} s'il existe un sous-fibré inversible $L$ de $B$ avec
$p\cdot \mathrm{deg}(L)=2g-2$.
\end{definition}

\begin{lemma}[\cite{Tango} lemme 12]Pour qu'il existe un diviseur $D$ sur $X_1$
tel que $\mathcal{L}=\OO_{X_1}(D)$ soit un sous-faisceau de $B$,
il faut et il suffit qu'il existe $f\in K$ avec
$\mathrm{div}(df)\geq pD$ et $df\neq 0$ (où $K$ est le corps de
fonctions de $X_1$).
\end{lemma}

\begin{corollary}Une courbe $X/k$ est de Tango si et seulement s'il
existe un diviseur $D$ sur $X_1$ et une $f\in K$ avec
$\mathrm{div}(df)=pD$ et $df\neq 0$.
\end{corollary}

\begin{example}\label{courbe de tango} On se place sur un corps $k$ de
caractéristique $3$. Soit $\mathbf{P}=\mathbf{P}_{k}^{1}$ la
droite projective et $\mathbf{A}=\mathbf{A}_{k}^{1}$ l'ouvert
droite affine de coordonnée $t$. Soit $d>0$ un entier pair.
Considérons un polynôme de degré $3d$ en $t$, $F_{3d}$, dont la
dérivée $F_{3d}'$ a $3d-2$ racines simples. Par exemple $t^{3d} -
t^{3d-1} +t$. Soient $a_i$ les racines simples de $F_{3d}'$. Soit
$X$ la courbe hyperelliptique, d'équation  $y^2=F_{3d}'$. Notons
$\pi:X\rightarrow \mathbf{P}$ le revêtement double et $b_i$ le
point de $X$ au-dessus de $a_i$. Comme $d$ est pair, $\{b_i\}$ est
l'ensemble des points de ramification de $\pi$, et $X$ est de
genre $(3d-4)/2$. Soit $G_d$ un polynôme en $t$ de degré $d$,
disons à racines simples $c_j$ distinctes des $a_i$ (mais ce n'est
pas important). La fonction rationnelle $f:= F_{3d}/ (G_d)^{3}$ a
pour différentielle $df$, de diviseur $\sum_ia_{i}-3\sum_j c_j$.
Calculée sur $X$, $df$ a pour diviseur $3(\sum_ib_i-
\pi^{-1}(\sum_j c_j))$ (rappelons que $\pi$ n'est pas ramifié
au-dessus de $\infty\in \mathbf{P}$). Notons $M:=\OO_X(\sum_i b_i-
\pi^{-1}(\sum_j c_j))$. Alors $X$ est de Tango avec $M$ le
faisceau inversible sur $X$ tel que $p\cdot\mathrm{deg}(M)=2g-2$.
On renvoie à \cite{R4} pour un autre exemple de courbe de Tango
qui utilise le revêtement d'Artin-Schreier de la droite affine.
\end{example}


\begin{remark} Soient $X/k$ une courbe de Tango, $L$ un sous-fibré
inversible de $B$ tel que $p\cdot\mathrm{deg}(L)=2g-2$. En
utilisant les puissances symétriques, on voit qu'il existe des
sous-faisceaux $\mathcal{F}_{i}~(0\leq i\leq p-1)$ de $B$:
$$ 0=\mathcal{F}_0\subset
\mathcal{F}_1(=L)\subset \mathcal{F}_2\subset \cdots\subset
\mathcal{F}_{p-1}\subset B
$$
de quotients successifs $\mathcal{F}_{i}/\mathcal{F}_{i-1}\simeq
L^{\otimes^{i}} (1\leq i\leq p-1)$, les $\mathcal{F}_i$ sont de
rang $i$ et de degré
$\mathrm{deg}(\mathcal{F}_{i})=\frac{i(i+1)\mathrm{deg}(L)}{2p}=
\frac{i(i+1)(g-1)}{p}$.
\end{remark}

\subsubsection*{Utilisation d'un résultat d'Hirschowitz.}

Rappelons le théorème suivant dû à Hirschowitz.
\begin{theorem}[Hirschowitz, \cite{Hirschowitz} théorème 4.4] Soient $X/k$ une courbe lisse connexe de genre $g\geq
2$,  $F$ un fibré vectoriel stable générique de pente $g-1$. Alors
$F$ contient un sous-faisceau inversible de degré $\delta$ dès que
$\delta\leq (g-1)/(p-1)$.
\end{theorem}

Ceci étant, par déformation et spécialisation, tout fibré
vectoriel de pente $g-1$ sur $X_1$ contient un sous-faisceau
inversible de degré $\delta$ pour $\delta\leq (g-1)/(p-1)$.

Suivant Tango \cite{Tango}, on note $n(X)$ le degré maximal d'un
sous-fibré inversibles de $B$, on a donc:

\begin{prop} $\frac{g-1}{p-1}\leq n(X)\leq \frac{2g-2}{p}$.
\end{prop}

\subsection{Questions sur le diviseur thêta}

Soit $X$ une courbe propre, lisse, connexe, de genre $g\geq 2$,
définie sur un corps $k$ algébriquement clos de caractéristique
$p>0$. On rassemble dans ce $\S$ quelques questions qui se posent
naturellement dans l'étude géométrique et arithmétique de $\Theta$
(aussi \cite{R2} $\S$ 1.4).

\begin{quest} Le diviseur $\Theta$ est-il intègre?\end{quest}

A défaut d'être irréductible, on peut se demander:

\begin{quest} Le diviseur $\Theta$ est-il réduit? \end{quest}

Si $\Theta$ n'est pas irréductible, il y a lieu de répartir les
composantes irréductibles en deux types:

\begin{defn}\label{composante principale}
Soit $\Theta_i$ une composante irréductible de $\Theta$, on dit
que $\Theta_i$ est \textit{principale}, si $\Theta_i\cap
\ker(V)\neq \emptyset$ (où $V:J_1\rightarrow J$ est le
Verschiebung de $J$). Sinon, on dit que $\Theta_i$ est
\textit{secondaire}.
\end{defn}

D'après la propriété de Dirac ($\S$ \ref{propriete de Dirac}), on
sait que $\Theta$ contient des composantes principales.

\begin{quest}
Le diviseur $\Theta$ peut-il avoir des composantes secondaires?
\end{quest}

En étudiant les liens entre $\Theta$ et le groupe fondamental de
$X$, les questions suivantes se présentent:

\begin{quest}
Toute composante $\Theta_i$ de $\Theta$ est-elle ample?
\end{quest}

A défaut d'être ample, on peut se demander

\begin{quest}
Le diviseur $\Theta$ peut-il avoir des composantes irréductibles
qui sont des translatés de sous-variétés abéliennes de $J_1$?
\end{quest}

On peut aussi s'intéresser à la propriété de Dirac:

\begin{quest} Etant donnée une variété abélienne principalement
polarisée par une polarisation $\tau$, existe-t-il au plus un
diviseur positif algébriquement équivalent à $(p-1)\tau$ qui
satisfait à la propriété de Dirac?
\end{quest}

Les résultats suivants sont maintenant connus. D'abord, on sait
que si $p=2$, le diviseur $\Theta$ est intègre et normal ($\S$
\ref{cas p=2}). Supposons $p\geq 3$, alors

\begin{itemize}

\item Le diviseur $\Theta$ est normal si $X$ est une courbe
générale (théorème \ref{irr de theta gen} et théorème \ref{nor de
theta gen}), mais pas nécessairement normal si $X$ est spéciale
($\S$ \ref{Theta est sans TAS}).

\item $\Theta$ possède au moins une composante principale qui
n'est pas un translaté d'une sous-variété abélienne (\cite{R2}
prop. 1.2.1).

\item Si $p=3$, le diviseur $\Theta$ est réduit ($\S$ \ref{Theta
est reduit}), et ne possède pas de composante qui soit le
translatés d'une sous-variété abélienne ($\S$ \ref{Theta est sans
TAS}).

\item Si $p=3$ et $g=2$, $\Theta$ est intègre (théorème \ref{cas
p=3 g=2} ).

\item Si $g=2$, et $X$ est ordinaire, toute composante de $\Theta$
est ample (corollaire \ref{theta est ample si g=2}).
\end{itemize}

%% file: chapitre2.tex
\section{Etude différentielle du diviseur thêta}
\markboth{chapitre2}{chapitre2}

Dans ce chapitre, $k$ désigne un corps algébriquement clos de
caractéristique $p>0$.

\subsection{Rappel des résultats de Laszlo} Soit $X/k$ une courbe
lisse connexe de genre $g$, $E$ un fibré vectoriel sur $X$ de
pente $g-1$ et de rang $r$ avec $h^{0}(X,E)\neq 0$. Soit
$\mathcal{E}$ une déformation verselle de $E$ sur une base $S$, en
particulier, $S$ est un schéma local essentiellement lisse sur
$k$, d'espace tangent
$\left(\mathfrak{m}_s/\mathfrak{m}_{s}^{2}\right)^{\vee}\simeq
H^{1}(X,\mathcal{E}nd(E))$ au point fermé $s\in S$. Par
définition, $\mathcal{E}$ est un fibré vectoriel sur $X\times_k S$
qui relève $E$. Si $E$ est simple, c'est-à-dire si
$H^{0}(X,\mathcal{E}nd(E))\simeq k$ (par exemple, si $E$ est un
fibré stable sur $X$), $S$ est donc de dimension $r^{2}(g-1)+1$.
Alors il existe un diviseur positif $\Theta_{\mathrm{univ},E}$ sur
$S$ défini comme le diviseur associé au déterminant de
$\mathrm{R}f_{\ast}(\mathcal{E})$, où $f:X\times_k S\rightarrow S$
est le morphisme naturel.

\begin{theoreme}[Laszlo,\cite{Laszlo}]\label{theta universel} La multiplicité de
$\Theta_{\mathrm{univ},E}$ en $s$ est $h^{0}(X,E)$.
\end{theoreme}

Laszlo étudie l'application donnée par le cup-produit
$$
H^{1}(X,\mathcal{E}nd(E))\otimes_k H^{0}(X,E)\rightarrow
H^{1}(X,E)
$$
et montre que pour $\alpha$ général dans
$H^{1}(X,\mathcal{E}nd(E))$, l'application
$$
\alpha\cup: ~~H^{0}(X,E)\rightarrow H^{1}(X,E)
$$
est injective donc bijective.

Considérons ensuite $\mathcal{U}(r,g-1)$ l'espace de modules des
fibrés vectoriels stables de rang $r$ et de pente $g-1$. Il existe
un changement de base étale $\mathcal{V}\rightarrow
\mathcal{U}(r,g-1)$, et une certaine {\textquotedblleft famille de
Poincaré\textquotedblright} de fibrés vectoriels stables
$\mathfrak{E}$ de rang $r$ et de pente $g-1$ au-dessus de $X$. Le
déterminant de $\mathrm{R}f_{\ast}'(\mathfrak{E})$ (où $f':X\times
\mathcal{V}\rightarrow \mathcal{V}$ est le morphisme naturel)
définit un diviseur positif sur $\mathcal{V}$ qui ne dépend pas du
choix de $\mathfrak{E}$. On obtient donc, par descente étale, le
diviseur thêta universel $\Theta_{\mathrm{univ}}$ sur
$\mathcal{U}(r,g-1)$.

Supposons maintenant $E$ stable, et que $h^{0}(X,E\otimes L)=0$
pour $L$ inversible de degré $0$ général dans la jacobienne $J$ de
$X$. Alors $E$ possède un diviseur thêta $\Theta_E$, qui est un
diviseur positif de $J$ défini comme déterminant de
$\mathrm{R}pr_{J,\ast}(E\otimes_{\OO_{X}}\mathcal{P})$ avec
$pr_J:X\times_k J\rightarrow J$ le morphisme naturel et
$\mathcal{P}$ un faisceau de Poincaré sur $X\times_k J$.
L'application $L\mapsto E\otimes L$ définit une application de $J$
dans $\mathcal{U}(r,g-1)$, et $\Theta_E$ est l'image réciproque du
diviseur $\Theta_{\mathrm{univ}}$ par cette application. Fixons un
point $x$ de $\Theta_E$ correspondant au faisceau inversible $L$
de degré $0$. On a une application naturelle
$$
\alpha:\OO_{X}\rightarrow \mathcal{E}nd(E\otimes L)
$$
qui envoie une section locale $t$ de $\OO_X$ sur l'homothétie de
$E\otimes L$ associée à $t$. D'où une application
$$
H^{1}(X,\OO_{X})\rightarrow H^{1}(X,\mathcal{E}nd(E\otimes L)).
$$
On déduit de \ref{theta universel} la proposition suivante:

\begin{proposition}[Laszlo] (1) On a $\mathrm{mult}_x(\Theta_E)\geq h^{0}(X,E\otimes
L)$;

(2) Il y a égalité si et seulement si pour $a\in H^{1}(X,\OO_X)$
assez général, l'application
$$
a\cup: ~~~ H^{0}(X,E\otimes L)\rightarrow H^{1}(X,E\otimes L)
$$
est bijective;

(3) Lorsque $\Theta_E$ est lisse en $x$, l'espace tangent à
$\Theta_E$ en $x$ est le noyau de
$$
H^{1}(X,\OO_{X})\rightarrow \mathrm{Hom}_k(H^{0}(X,E\otimes L),
H^{1}(X,E\otimes L))
$$
défini par $a\in H^{1}(X,\OO_{X})\mapsto a\cup\in
\mathrm{Hom}_k(H^{0}(X,E\otimes L), H^{1}(X,E\otimes L))$.
\end{proposition}

\subsection{Application à $B$.} Désormais dans ce chapitre, $X$
désigne une courbe propre lisse connexe de genre $g\geq 1$ sur
$k$, et $X_1$ l'image réciproque de $X/k$ par le Frobenius absolu
$F:\mathrm{Spec}(k)\rightarrow \mathrm{Spec}(k)$. Notons $J$
(resp. $J_1$) la jacobienne de $X/k$ (resp. de $X_1/k$), et
$\Theta\hookrightarrow J_1$ le diviseur thêta associé au fibré $B$
des formes différentielles localement exactes sur $X_1$.

Rappelons que l'on a un produit alterné sur $B$ à valeurs dans
$\Omega^{1}_{X_1/k}$:
$$
(\cdot,\cdot):~~ B\otimes_{\OO_{X_1}}B\rightarrow
\Omega^{1}_{X_1/k}.
$$
Pour $L$ un faisceau inversible sur $X_1$, via l'isomorphisme
$H^{0}(X_1,B\otimes L)\rightarrow
\mathrm{Hom}_{\OO_{X_1}}(L^{-1},B)$, on a donc un accouplement
$$
H^{0}(X_1,B\otimes L)\otimes_k H^{0}(X_1,B\otimes
L^{-1})\rightarrow H^{0}(X_1,\Omega^{1}_{X_1/k}).
$$
Par dualité de Serre, on déduit du résultat de Laszlo la
proposition suivante:

\begin{proposition} Soit $x\in \Theta\subset J_1$ un point fermé, $L$ le
faisceau inversible de degré $0$ sur $X_1$ associé à $x$. Alors

(1) On a $\mathrm{mult}_{x}(\Theta)\geq h^0(X_1, B\otimes L)$.

(2) Pour qu'il y ait égalité, il faut et il suffit que pour
$\alpha:H^{0}(X_1,\Omega^{1}_{X_1/k})\rightarrow k$, forme
linéaire assez générale, l'application bilinéaire
$$
\xymatrix{H^0(X_1,B\otimes L)\otimes H^0(X_1,B\otimes
L^{-1})\ar[r]& H^{0}(X_1,\Omega^{1}_{X_1/k})\ar[r]^<<<<<<{\alpha}&
k}
$$
soit non dégénérée. Cette condition (2) est automatiquement
assurée pour $p=2$.
\end{proposition}

\begin{corollaire}[Le critère de lissité]\label{critere de lissite}
Soient $x\in \Theta$ un point fermé, $L$ le faisceau inversible
sur $X_1$ correspondant. Alors $\Theta$ est lisse en $x$ si et
seulement si les deux conditions suivantes sont réalisées:

(1) $h^0(X_1,B\otimes L)=1$ (et donc aussi $h^{0}(X_1,B\otimes
L^{-1})=1$).

(2) L'application
$$
H^{0}(X_1,B\otimes L)\otimes_k H^0(X_1, B\otimes
L^{-1})\rightarrow H^{0}(X_1,\Omega^{1}_{X_1/k})
$$
donnée par $(\cdot,\cdot)$ est non nulle. Cette condition est
automatique pour $p=2$.

De plus, si ces conditions sont réalisées, notons $\omega$ un
générateur de la droite vectorielle, l'image dans
$H^{0}(X_1,\Omega^{1}_{X_1/k})$ de $H^{0}(X_1,B\otimes L)\otimes_k
H^0(X_1, B\otimes L^{-1})$, alors l'espace tangent à $\Theta$ en
$x$ est le noyau de l'application linéaire sur
$H^{1}(X_1,\OO_{X_1})$ définie par $\omega$.
\end{corollaire}

\begin{remarque}\label{Theta est singulier en point d'ordre divisiant 2}
Pour $p\geq 3$ impair, nons avons montré que $\Theta$ est
totalement symétrique au sens de Mumford (\ref{theta est fortement
symetrique}). Alors $\Theta$ est toujours singulier à un point $x$
d'ordre divisant $2$ de $J_1$ dès que $x\in \Theta$. De plus,
$\mathrm{mult}_x(\Theta)$ est paire.
\end{remarque}

\subsection{Etude de $\Theta$ aux points d'ordre $p$} Commençons
par quelques rappels:

\begin{itemize}


\item Soit $L$ un faisceau inversible de degré $0$ sur $X_1$.
Alors $L$ se plonge dans $F_{\ast}(\OO_X)$ si et seulement si
$F^{\ast}L\simeq \OO_X$, et donc si et seulement si $L$ est
d'ordre divisant $p$. On a la suite exacte suivante:
$$
0\rightarrow L\rightarrow F_{\ast}(\OO_X\otimes
F^{\ast}L)\rightarrow B\otimes L\rightarrow 0.
$$

\item Réciproquement, soit $L$ inversible sur $X_1$ d'ordre $p$,
il correspond à un élément non-nul de $H^{1}(X_1,\mu_p)$, et donc
à un torseur de base $X_1$ sous $\mu_p$, d'algèbre
$\mathcal{A}=\oplus_{i=0}^{p-1}L^{i}$, où la multiplication est
donnée par $L^{p}\simeq \OO_{X_1}$. Soient
$$
0\subset E_0\subset E_1\subset \cdots \subset E_{p-1}\subset
F_{\ast}(\OO_{X})
$$
les sous-faisceaux de $F_{\ast}(\OO_{X})$ obtenus à partir des
puissances symétriques ($\S$ \ref{puissance sym}). Comme $L$ se
plonge dans $F_{\ast}(\OO_{X})$, la suite exacte (pour la
définition de $E$, on renvoye à $\S$ \ref{puissance sym}
(\ref{fibre E}))
$$
0\rightarrow \OO_{X_1}\rightarrow E\rightarrow L\rightarrow 0
$$
est scindée. Par suite, $E_{p-1}\simeq \oplus_{i=1}^{p-1}L^{i}$.
Donc l'algèbre $\mathcal{A}$ est contenue dans
$F_{\ast}(\OO_X)$.\newline

\item Rappelons le lien entre faisceau inversible d'ordre $p$ et
formes différentielles de Cartier. Soient $Y/k$ une courbe propre
lisse connexe, $L$ un faisceau d'ordre $p$ sur $Y$ correspondant à
un point $x$ de $J_Y$, réalisé comme sous-faisceau du faisceau des
fonctions méromorphes sur $Y$, et défini par une section
méromorphe $(U_{\alpha},f_{\alpha})$ (où $\{U_{\alpha}\}$ est un
recouvrement ouvert de $Y$). Les
$g_{\alpha,\beta}:=f_{\alpha}/f_{\beta}\in \OO_{Y}(U_{\alpha}\cap
U_{\beta})^{\ast}$ forment un $1$-cocycle de $\OO_{Y}^{\ast}$.
Comme $L^{p}\simeq \OO_{Y}$, $(g_{\alpha,\beta}^{p})$ est un
$1$-cobord, il existe donc des sections locales $u_{\alpha}\in
\OO_{X_1}(U_{\alpha})^{\ast}$ telles que
$g_{\alpha,\beta}^{p}=u_{\alpha}/u_{\beta}$ dans
$\OO_{Y}(U_{\alpha}\cap U_{\beta})^{\ast}$. Donc
$f_{\alpha}^{p}/u_{\alpha}=f_{\beta}^{p}/u_{\beta}$ dans
$U_{\alpha}\cap U_{\beta}$. Ceci nous donne une section globale
non-nulle $(U_{\alpha},f_{\alpha}^{p}/u_{\alpha})$ de $L^{p}$.
Donc $(U_{\alpha},du_{\alpha}/u_{\alpha})$ définit une forme
différentielle holomorphe sur $Y$, que l'on note $\omega_x$. La
forme $\omega_x$ ne dépend pas du choix de la section méromorphe
$(U_{\alpha},f_{\alpha})$. Cette forme $\omega_x$ est \emph{la
forme de Cartier associée à $x$}.\newline

\item Pour $Y$ une courbe lisse connexe sur $k$ un corps
algébriquement clos de caractéristique $p>0$, on dispose d'un
opérateur de Cartier $C_Y$ sur $H^{0}(Y,\Omega^{1}_{Y/k})$ qui est
$Fr_{k}^{-1}$-linéaire. Le noyau de
$C_Y-\mathrm{Id}:H^{0}(Y,\Omega^{1}_{Y/k})\rightarrow
H^{0}(Y,\Omega^{1}_{Y/k})$ est un $\mathbf{F}_p$-espace de
dimension $d$ égale au $p$-rang de $Y$, et on a
$\ker(C_Y-\mathrm{Id})=\{0\}\cup \{\omega_x| x\in
J_Y~\mathrm{d'ordre}~p\}$. Les formes de Cartier sur $Y/k$
engendrent un $k$-sous-espace vectoriel de
$H^{0}(Y,\Omega^{1}_{Y/k})$ de dimension $d$. On renvoie à
\cite{Serre1} pour les détails.

\end{itemize}

\begin{proposition} Soient $L$ un faisceau d'ordre $p$ sur $X_1$,
$\omega$ la forme de Cartier associée à $L$. On considère
$\mathcal{A}=\oplus_{i=0}^{p-1}L^{i}$ comme sous-faisceau de
$F_{\ast}(\OO_X)$. Soient $\beta:L\hookrightarrow B$ et $\beta':
L^{-1}=L^{p-1}\hookrightarrow B$ les plongements ainsi obtenus.
Alors $1\in \Gamma(X_1,L\otimes L^{-1})$ s'envoie sur $-\omega\in
H^{0}(X_1,\Omega^{1}_{X_1/k})$ par
$(\cdot,\cdot)\circ(\beta\otimes \beta')$.
\end{proposition}

\begin{proof}  Soit $s$ un générateur
local de $L$, alors $s^{p}=f_1\in \OO_{X_1}^{\ast}$. Considérons
$s^{-1}=s^{p-1}/f_1$, c'est un générateur local de $L^{p-1}\simeq
L^{-1}$. Donc $1=s\otimes s^{-1}\in \OO_{X_1}\simeq L\otimes
L^{-1}$ s'envoie sur $C(s\cdot d(1/s))=-ds/s=-\omega$ par
$(\cdot,\cdot): F_{\ast}(\OO_X)\otimes F_{\ast}(\OO_X)\rightarrow
\Omega^{1}_{X_1/k}$. D'où le résultat.

\end{proof}

\begin{corollaire} Gardons les notations ci-dessus. Pour que $
\Theta$ soit lisse en $x\in \Theta$ qui correspond à $L$ d'ordre
$p$, il faut et il suffit que $h^{0}(X_1,B\otimes L)=1$. Dans ce
cas, l'espace tangent en $x$ à $\Theta$ est l'orthogonal de
$\omega_x$ dans $H^{1}(X_1,\OO_{X_1})$.
\end{corollaire}

\subsection*{Application aux composantes irréductibles de $\Theta$.}

Rappelons qu'une composante $D$ de $\Theta$ est dite
\emph{principale} si $D\cap \ker(V)\neq \emptyset$, où
$V:J_1\rightarrow J$ est le Verschiebung de $J$.

\begin{lemme}\label{lemme sur espace tangent}
Soient $x,\ y$ deux points d'ordre $p$ de $J_1$ où $\Theta$ est
lisse. Alors $T_{x}\Theta$ et $T_{y}\Theta$ coïncident (vus comme
$k$-sous-espace de $H^{1}(X_1,\OO_{X_1})$) si et seulement $x$ et
$y$ sont colinéaires dans $J_1[p](k)$.
\end{lemme}

\begin{proof} En fait, d'après le corollaire précédent, via les
identifications précédentes, on sait que $T_{x}\Theta$ est le même
que $T_{y}\Theta$ si et seulement si les formes de Cartier
$\omega_{x}$ et $\omega_y$ sont $k$-linéairement dépendantes. Donc
si et seulement si $x$ et $y$ sont colinéaire dans $J_1[p](k)$
(\cite{Serre1} page 41 proposition 10).
\end{proof}

\begin{proposition} \label{critere de ne pas d'etre un TAS}
Soit $D$ une composante irréductible de $\Theta$, qui contient
deux points $x,\ y$ d'ordre $p$ de $J_1$, où $\Theta$ est lisse.
Alors $D$ n'est pas un translaté d'une sous-variété abélienne de
$J_1$.
\end{proposition}

\begin{proof} Soit $A$ une sous-variété abélienne de $J_1$ qui laisse $D$
stable. Montrons qu'il n'y a pas de $a\in A(k)$ tel que $x=y+a$.
Sinon, l'espace tangent en $x$ et $y$ est le même, ce qui signifie
que $x$ et $y$ sont $\mathbf{F}_p$-colinéaires. C'est-à-dire, il
existe un $n\in \mathbf{N}$ tel que $1\leq n\leq p-1$, de sorte
que $x=ny$. Donc $a=(n-1)y$. Comme $x\neq y$, on obtient que
$n\neq 1$. Et comme $y$ est un point d'ordre $p$, on obtient que
$y\in A$, donc $0=y-y\in D\subset \Theta$. Par ailleurs, soit $f$
une équation locale de $\Theta$ en $x$, comme $\Theta$ est lisse
en $x$, $f$ est aussi une équation locale de $D$ en $x$, d'après
la propriété de Dirac, $f$ est nulle dans l'anneau local de
$\ker(V)$ en $x\in \ker(V)$ (avec $V:J_1\rightarrow J$ le
Verschiebung de $J$). Puisque $D$ est stable par $A$, $f$ est
encore une équation locale de $D$ en $0$, donc fait partie d'une
équation de $\Theta$ en $0$. Donc cette équation locale de
$\Theta$ est nulle dans l'anneau local de $\ker V$ en $0$. Ce qui
contredit à la propriété de Dirac pour $\Theta$. D'où le résultat.
\end{proof}

\begin{proposition} Soit $X/k$ une courbe lisse connexe
ordinaire de genre $g$. Supposons que pour toute droite affine
$\Delta$ du $\mathbf{F}_p$-espace vectoriel $J_1[p](k)\simeq
\mathbf{F}_{p}^{g}$ qui ne passe pas par l'origine, $\Theta$ soit
lisse en au moins deux points de $\Delta$. Alors toute composante
principale de $\Theta$ est ample.
\end{proposition}

\begin{proof} En fait, soit $\Theta'\subset \Theta$ une composante principale de
$\Theta$. Soit $A/k$ la plus grande sous-variété abélienne de
$J_1$ qui laisse stable $\Theta'$. Comme $X$ est ordinaire,
$\Theta$ ne passe pas par l'origine. Soit $x\in \Theta'$ un point
d'ordre $p$ de $J_1$, alors $x+A[p](k)\subset \Theta'$ contient
une droite affine $\Delta$ du $\mathbf{F}_p$-espace vectoriel
$J_1[p](k)$ qui ne passe pas par l'origine. Par hypothèse,
$\Theta$ est lisse en au moins deux points $a,~a'\in \Delta\subset
\Theta'$. Soit $z\in A(k)$ tel que $a=a'+z$. Alors l'espace
tangent à $\Theta$ en $a$ est le même que l'espace tangent à
$\Theta$ en $a'$, donc $a$ et $a'$ sont $\mathbf{F}_p$-colinéaires
dans $A[p](k)$ (lemme \ref{lemme sur espace tangent}), d'où une
contradiction.
\end{proof}

De la même façon, on montre:

\begin{proposition} Soit $X/k$ une courbe ordinaire lisse connexe de
genre $g$. Supposons que pour toute droite affine $\Delta$ du
$\mathbf{F}_p$-espace vectoriel $J_1[p](k)\simeq
\mathbf{F}_{p}^{g}$, $\Theta$ est lisse en au moins deux points de
$\Delta$. Alors toute composante principale n'est pas stable par
une sous-variété abélienne de $p$-rang $\geq 1$.
\end{proposition}

\begin{proposition} Soit $X/k$ une courbe lisse connexe
non-ordinaire. Soit $\Theta''$ la réunion des composantes
irréductibles qui passent par l'origine de $J_1$. Supposons que
$\Theta$ est lisse en tous les points d'ordre $p$, alors
$\Theta''$ est ample.
\end{proposition}

\begin{proof} Dans le cas où $p=2$, on sait que $\Theta$ est toujours
irréductible et ample ($\S$ \ref{cas p=2}), donc $\Theta''=\Theta$
est ample. Supposons maintenant $p\geq 3$. Comme $X$ est
non-ordinaire, $\Theta'' \neq \emptyset$. Soit $A$ la plus grande
sous-variété abélienne de $J_1$ qui laisse stable $\Theta''$.
D'après la propriété de Dirac (corollaire \ref{Dirac implique
ordinaire}), $A$ est ordinaire. Si $A\neq 0$, $\Theta''$ contient
au moins un point $x\in A(k)$ d'ordre $p$. Par hypothèse, $\Theta$
est lisse en $x$. A priori, $\Theta''$ est lisse sur $x$, et donc
lisse à l'origine de $J_1$ puisque $A$ laisse stable $\Theta''$.
D'où une contradiction puisque $\Theta$ est toujours singulier à
l'origine lorsque $p\geq 3$ (remarque \ref{Theta est singulier en
point d'ordre divisiant 2}).
\end{proof}

\subsubsection*{Un exemple où $\Theta$ est singulier en un point d'ordre $p$.}
Soit $k$ un corps algébriquement clos de caractéristique $p>0$. On
va construire une courbe ordinaire $X/k$ avec un faisceau $L$
d'ordre $p$, tel que $H^{0}(X_1,B\otimes L)\geq 2$, donc $\Theta$
n'est pas lisse en $L$. Comme $k$ est algébriquement clos, le
Frobenius absolu $Fr:\mathrm{Spec}(k)\rightarrow \mathrm{Spec}(k)$
est un isomorphisme. On peut donc identifier $X_{1}$ à $X$ de
sorte que $F$ s'identifie au Frobenius absolu $Fr:X\rightarrow X$.

Considérons $\mathbf{P}$ la droite projective, $\mathbf{A}\subset
\mathbf{P}$ l'ouvert droite affine de coordonnée $t$. Soit $f(t)$
un polynôme en $t$ de degré $n$ à racines simples. Considérons la
courbe projective $X$ définie par l'équation
$$
\frac{1}{v^{p}}-\frac{1}{v}=\frac{1}{f(t)},
$$
c'est-à-dire $v^{p}+f(t)v^{p-1}-f(t)=0$. Alors le morphisme
naturel $\pi:X\rightarrow \mathbf{P}$ est fini galoisien de degré
$p$, ramifié au-dessus de $t_i$ ($i\in I$ ) avec ramification
sauvage minimale. Donc, d'après la formule de
Crew-Deuring-Shafarevich (théorème 3.1 de \cite{Bouw}), $X$ est
une courbe ordinaire de genre $g=(n-1)(p-1)$. Soient $x_i\in X$
les points au-dessus de $t_i\in \mathbf{P}$ où $\pi$ est ramifié.
Soit $B_{X}$ (resp. $B_{\mathbf{P}}$) le faisceau des formes
différentielles localement exactes de $X$ (resp. de $\mathbf{P}$).
D'abord, on établit une relation entre $B_{X}$ et
$B_{\mathbf{P}}$. Commençons par un calcul local.

Soient $R=k[[t]]$, $R'=k[[t]][v]/(v^{p}+tv^{p-1}-t)=k[[v]]$.
Notons $\phi:R=k[[t]]\rightarrow
R'=k[[t]][v]/(v^{p}+tv^{p-1}-t)=k[[v]]$ le morphisme continu
naturel de $k$-algèbres, alors
$\phi(t)=v^{p}/(1-v^{p-1})=v^{p}\sum_{i=0}^{\infty} v^{i(p-1)}$.
Soient $F_t:k[[t]]\rightarrow k[[t]]$ (resp.
$F_v:k[[v]]\rightarrow k[[v]]$) le morphisme de Frobenius relatif:
c'est un morphisme continu de $k$-algèbres tel que $t\mapsto
t^{p}$ (resp. tel que $v\mapsto v^{p}$). Notons
$B_t=\oplus_{i=0}^{p-2}k[[t^{p}]]t^{i}$ (resp.
$B_v=\oplus_{i=0}^{p-2}k[[v^{p}]]v^{i}$), on a un diagramme à
lignes exactes:
$$
\xymatrix{0\ar[r] & k[[t]]\ar[r]^{F_t}\ar[d]^{\phi} &
k[[t]]\ar[r]^{d/dt}\ar[d]^{\phi}\ar[d]& B_t\ar[r]\ar@{.>}[d]^{\phi'}& 0\\
0\ar[r]& k[[v]]\ar[r]^{F_v}& F[[v]]\ar[r]^{d/dv}& B_v\ar[r]& 0}
$$
d'où un morphisme $\phi':B_t\rightarrow B_v$ rendant le carré de
droite commutatif.

\begin{lemme} Avec les notations ci-dessus, on a $\phi'(B_t)\subset
v^{p}B_{v}$.
\end{lemme}

\begin{proof} C'est un calcul direct. Par définition, pour $0\leq i\leq p-1$,
$$
\phi'(t^{i})=(\phi(t^{i+1})/(i+1))/dv=d\left(\frac{1}{i+1}\cdot
\left(\frac{v^{p}}{1-v^{p-1}}\right)^{i+1}\right)'=
\frac{(1-v^{p-1})^{p-2-i}v^{p(i+1)}v^{p-2}}{(1-v^{p-1})^{p}}\subset
v^{p}B_v.
$$
D'où le résultat.
\end{proof}

\begin{corollaire} Gardons les notations ci-dessus. Alors le morphisme
naturel $\pi^{\ast}B_{\mathbf{P}}\rightarrow B_{X}$ est injectif
et son image est contenue dans $B_{X}(-\sum_{i\in I}x_i)\subset
B_X$.
\end{corollaire}

Ecrivons $I=I'\coprod I''$ avec $\sharp I''=(p-1)\sharp I'$. Soit
$L=\OO_{X}(\sum_{i'\in I'}(p-1)x_{i'}-\sum_{i''\in I''}x_{i''})$,
alors $L^{p}=\pi^{\ast}\OO_{\mathbf{P}}(\sum_{i'\in
I'}(p-1)t_{i'}-\sum_{i''\in I''}t_{i''})\simeq
\pi^{\ast}\OO_{\mathbf{P}}=\OO_{X}$, et on a
\begin{eqnarray*}
B_X\otimes L  &=& B_X(\sum_{i'\in I'}(p-1)x_{i'}-\sum_{i''\in
I''}x_{i''}) \\ & \supset& \pi_{1}^{\ast}B_{\mathbf{P}}(\sum_{i\in
I}x_i+\sum_{i'\in I'}(p-1)x_{i'}-\sum_{i''\in I''}x_{i''}) \\
& =& \pi^{\ast}B_{\mathbf{P}}(\sum_{i'\in I'}px_{i'})\\
&=& \pi^{\ast}\left(B_{\mathbf{P}}(\sum_{i'\in I'}t_{i'})\right).
\end{eqnarray*}
Par ailleurs, sur la droite projective $B_{\mathbf{P}}\simeq
\OO_{\mathbf{P}}(-1)^{\oplus^{p-1}}$, donc $L\otimes B_{X}\supset
\OO_{\mathbf{P}^{1}_{k}}(\sharp I'-1)^{\oplus^{p-1}}$. Donc si
$\sharp I'\geq 2$ pour $p\geq 2$ (resp. si $\sharp I'\geq 1$ pour
$p\geq 3$), on a $H^{0}(X,B_X\otimes L)\geq (p-1)(\sharp I')\geq
2$.

\subsection{Etude différentielle du schéma $\mathcal{H}$} Rappelons
que l'on a noté $\mathcal{H}$ le schéma de Hilbert des
sous-faisceaux inversibles de $B$ de degré $0$ (pour la
définition, voir $\S$ \ref{le schema H}). On a une application
naturelle $\mathcal{H}\rightarrow \Theta$, qui est un
{\textquotedblleft fibré projectif\textquotedblright} sur $\Theta$
relativement à un certain faisceau cohérent sur $\Theta$
(proposition \ref{H est fibre}). Si $L$ est un faisceau inversible
sur $X_1$ de degré $0$, qui correspond à un point $x$ de $\Theta$.
La fibre de $\mathcal{H}\rightarrow \Theta$ au-dessus de $x$ est
l'espace projectif des droites de
$\mathrm{Hom}_{\OO_{X_1}}(L,B)=H^{0}(X_1,B\otimes L^{-1})$.

\begin{proposition}[\cite{FantechiGottsche} corollary 6.4.11 ]
\label{etude diff de H} Soit $\mathcal{L}^{-1}$ un faisceau
inversible de degré $0$ sur $X_1$ qui correspond à un point $x\in
\Theta$. Soit $\alpha:\mathcal{L}\hookrightarrow B$ un plongement
correspondant à un point $z$ de $\mathcal{H}$ au-dessus de $z$, et
soit $\mathcal{G}=\mathrm{coker}(\alpha)$.

(1) L'espace tangent à $\mathcal{H}/k$ en $x$ est canoniquement
isomorphe à $H^{0}(X_1,\mathcal{G}\otimes \mathcal{L}^{-1})$,
l'obstruction à la lissité de $\mathcal{H}$ en $z$ est dans
$H^{1}(X,\mathcal{G}\otimes \mathcal{L}^{-1})$.

(2) Posons $d=h^{0}(X_1,\mathcal{G}\otimes \mathcal{L}^{-1})$,
$r=h^{1}(X_1,\mathcal{G}\otimes \mathcal{L}^{-1})$. Alors
$d-r=g-1$, et
$$
g-1=d-r\leq \mathrm{dim}_{z}(\mathcal{H})\leq d.
$$
Si $\mathrm{dim}_{z}(\mathcal{H})=d-r$, $\mathcal{H}$ est une
intersection complète en $z$; si
$\mathrm{dim}_{z}(\mathcal{H})=d$, $\mathcal{H}$ est lisse en $z$.
\end{proposition}

Gardons les notations précédentes, et notons
$\overline{\alpha}:\overline{\mathcal{L}}\hookrightarrow B$ le
sous-fibré inversible de $B$ engendré par
$\alpha:\mathcal{L}\hookrightarrow B$. Soit
$\overline{\mathcal{G}}=\mathrm{coker}(\overline{\alpha})$. Notons
$\mathcal{L}_{\alpha}^{\bot}$ l'orthogonal de
$\alpha:\mathcal{L}\hookrightarrow B$ pour $(\cdot,\cdot)$.

\begin{proposition}\label{diff prop of H} Les conditions (1)-(4) suivantes sont équivalentes:

(1) $\mathcal{H}$ est lisse en $z$ de dimension $g-1$.

(2) $H^{1}(X_1,\mathcal{G}\otimes \mathcal{L}^{-1})=0$

(3) $H^{1}(X_1,\overline{\mathcal{G}}\otimes \mathcal{L}^{-1})=0$

(4)
$\mathrm{Hom}(\mathcal{L}^{-1},\mathcal{L}^{\bot}_{\alpha})=0$.

\noindent De plus, ces conditions sont entraînées par

(5) $H^{1}(X_1,\overline{\mathcal{G}}\otimes
\bar{\mathcal{L}}^{-1})=0$.
\end{proposition}

\begin{proof} D'après la proposition précédente, $\mathcal{H}$ est
lisse en $z$ de dimension $g-1$ si et seulement si
$H^{1}(X_1,\mathcal{G}\otimes \mathcal{L}^{-1})=0$, d'où
(1)$\Leftrightarrow$(2). Or par définition, $\mathcal{G}$ est une
extension de $\overline{\mathcal{G}}$ par son sous-module des
torsions $\mathcal{T}$, on a donc $H^{1}(X_1,\mathcal{G}\otimes
\mathcal{L}^{-1})\simeq H^{1}(X_1,\overline{\mathcal{G}}\otimes
\mathcal{L}^{-1})$, d'où (2)$\Leftrightarrow$(3). A partir de la
suite exacte suivante:
$$
\xymatrix{0\ar[r] &  \mathcal{L}\ar[r]^{\alpha} & B\ar[r]&
\mathcal{G}\ar[r]& 0},
$$
on sait que $H^{1}(X_1,\mathcal{G}\otimes \mathcal{L}^{-1})=0$ si
et seulement si le morphisme $\phi_{\alpha}:
H^{1}(X_1,\OO_{X_1})\rightarrow H^{1}(X_1,B\otimes
\mathcal{L}^{-1})$ est surjectif. Or $\phi_{\alpha}$ est dual de
$\psi_{\alpha}:\mathrm{Hom}_{\OO_{X_1}}(\mathcal{L}^{-1},B)\rightarrow
\mathrm{Hom}_{\OO_{X_1}}(\OO_{X_1},\Omega_{X_1/k}^{1})$ (donné par
$<\cdot,\cdot>$ et $\alpha:\mathcal{L}\rightarrow B$) par la
dualité de Serre. Donc dire que $\phi_{\alpha}$ est surjectif
équivaut à dire que le morphisme $\psi_{\alpha}$ est injectif.
Autrement dit, on a (3)$\Leftrightarrow$(4). Par définition, on a
une suite exacte:
$$
0\rightarrow \mathcal{L}\rightarrow
\overline{\mathcal{L}}\rightarrow \mathcal{T}'\rightarrow 0
$$
avec $\mathcal{T}'$ un $\OO_{X_1}$-module de torsion. On en déduit
un morphisme surjectif $H^{1}(X_1,\mathcal{G}\otimes
\bar{\mathcal{L}}^{-1})\rightarrow H^{1}(X_1,\mathcal{G}\otimes
\mathcal{L}^{-1})$. Donc $H^{1}(X_1,\overline{\mathcal{G}}\otimes
\bar{\mathcal{L}}^{-1})\simeq H^{1}(X_1,\mathcal{G}\otimes
\bar{\mathcal{L}}^{-1})=0$ implique que
$H^{1}(X_1,\mathcal{G}\otimes \mathcal{L}^{-1})=0$, d'où
(5)$\Rightarrow$(2). Ceci finit la preuve.
\end{proof}

\begin{corollaire}\label{lissite de point de codim 0} Soit $\eta$ un point générique de $\Theta$. Les
conditions suivantes sont équivalentes:

(1) $\Theta$ est lisse en $\eta$;

(2) Il existe $\alpha:\mathcal{L}_{\eta}\hookrightarrow B_{\eta}$
qui est un plongement d'un faisceau inversible de degré $0$
au-dessus de $\eta$, tel que
$\mathrm{Hom}(\mathcal{L}_{\eta}^{-1},\mathcal{L}_{\eta,\alpha}^{\bot})=0$.

\noindent De plus, sous ces conditions, $\mathcal{H}\rightarrow
\Theta$ est un isomorphisme au-dessus d'un voisinage ouvert de
$\eta$.
\end{corollaire}

\begin{proof} Clairement, (1) implique (2), et si $\Theta$ est
lisse en $\eta$, $\mathcal{H}$ est isomorphe à $\Theta$ sur un
voisinage de $\eta$. Supposons ensuite qu'il existe un prolongment
$\alpha:L_{\eta}\hookrightarrow B_{\eta}$ satisfaisant la
condition (2), Alors $\mathcal{H}$ est lisse de dimension $g-1$ en
le point $\xi\in \mathcal{H}$ correspondant à $\alpha$. Soit $W$
la composante irréductible (munie de la structure de sous-schéma
réduit) de $\mathcal{H}$ passant par $\xi$, c'est une composante
de dimension $g-1$ qui est génériquement lisse sur $k$. Comme
$\xi$ est au-dessus de $\eta$ qui est un point générique de
$\Theta$, on en déduit que $W_{\eta}=\{\xi\}$. En particulier,
pour $z\in \overline{\{\eta\}}\subset \Theta$ assez général,
$\mathrm{dim}_{k(z)}\left(\QQ\otimes_{\OO_{\Theta,z}}k(z)\right)=1$,
donc $\QQ$ est inversible sur un voisinage de $\eta\in \Theta$.
Donc $\mathcal{H}$ est isomorphe à $\Theta$ sur un voisinage de
$\eta$, à priori, $\Theta$ est lisse en $\eta$.
\end{proof}

%% file: chapitre3.tex
\section{Etude générique du diviseur thêta}
\markboth{chapitre3}{chapitre3}

\subsection{Préliminaires.}




\subsubsection{Rappels sur l'action de monodromie.}\label{Section de
Monodromie}

\subsubsection*{Rappels sur les familles de courbes semi-stables.}

Soient $g\geq 1$ un entier, $k$ un corps algébriquement clos.
Notons $\underline{H}_g$ le foncteur de la catégorie de
$k$-schémas vers la catégorie des ensembles défini de la façon
suivante: soit $T$ un $k$-schéma.
\begin{itemize}
\item Pour $g\geq 2$, $\underline{H}_g(T)$ est l'ensemble des
classes d'isomorphisme de courbes stables tricanoniquement
plongées de genre $g$ (\cite{DM}).

\item Pour $g=1$, $\underline{H}_1(T)$ est l'ensemble des classes
d'isomorphisme de courbes $E/T$ propres plates à fibres
géométriques intègres de genre arithmétique $1$ avec pour
singularité au plus un point double ordinaire, munie d'une section
$o\in E(T)$ dans le lieu lisse, plongées dans $\mathbf{P}^{2}$
comme cubique plane par $\OO_E(3o)$:
$$
\xymatrix{E\ar@{^{(}->}[r]^{i} \ar[d]_{\pi} & \mathbf{P}^{2}_{T}\ar[ld]\\
T &}.
$$
\end{itemize}

\begin{theorem}[pour $g\geq 2$, \cite{DM}] \label{moduli} (1) Pour $g\geq 1$, le foncteur $\underline{H}_g$ est
représentable par un schéma $H_g$ irréductible lisse sur $k$. Soit
$\mathcal{X}_g/H_g$ la courbe universelle (pour $g=1$, on note
aussi $\mathcal{E}=\X_1$ la courbe universelle sur $H_1$).

(2) Pour $g\geq 2$, soit $H_{g,\mathrm{sing}}$ le sous-schéma
fermé réduit de $H_g$ image du lieu où la courbe universelle
$\X_g$ n'est pas lisse sur $H_g$. Alors $H_{g,\mathrm{sing}}$ est
un diviseur à croisements normaux de $H_g$. Plus précisément, si
$t\in H_g$, soit $\mathfrak{O}_t$ un hensélisé strict de l'anneau
local de $H_{g,\mathrm{sing}}$ en $t$. Les branches de
$H_{g,\mathrm{sing}}$ passant par $t$ (qui correspondent aux
composantes irréductibles de $\mathfrak{O}_t$) sont en bijection
avec les points doubles de la fibre $\X_t$ de $\X$.

\end{theorem}

\subsubsection*{Théorème d'Igusa et d'Ekedahl.}
Pour $g\geq 1$, soit $H_{g}'$ (resp. $\widetilde{H_g'}$) le lieu
des $t\in H_g$ tels que la jacobienne de $\X_{g,t}$ soit une
variété abélienne ordinaire sur $t$ (resp. tel que $\X_{g,t}$ soit
une courbe lisse ordinaire). Alors $H_g'$ (resp.
$\widetilde{H_{g}'}$) est un ouvert non vide de $H_g$. Notons
$\mathcal{X}_g'=\X_g\times_{H_g}H_g'$, et
$\widetilde{\mathcal{X}_{g}'}=X_g\times_{\mathcal{H}_g}\widetilde{H_g'}$
(pour $g=1$, on note aussi $\mathcal{E}'=\X_1'$ et
$\widetilde{E'}=\widetilde{\X_1'}$),
$\mathcal{J}_{g}'=\mathrm{Pic}^{\circ}_{\X_g'/H_g'}$ sa
jacobienne. $\mathcal{J}_{g}'$ est un schéma abélien ordinaire sur
$H_g'$ de dimension relative $g$. Soient $\bar{s}\rightarrow H_g'$
un point géométrique, $N=p^{e}N'\geq 1$ un entier (où $N'$ est
premier à $p$). Notons $\mathcal{J}_{g}'[N]_{et}$ le plus grand
quotient étale de $\mathcal{J}_{g}'[N]]$ sur $H_{g}'$, et
$\rho_{g,N}:\pi_{1}(H_g',\bar{s})\rightarrow
\mathrm{Aut}_{\mathbf{Z}/N
\mathbf{Z}}\left(\left(\mathcal{J}_{g}'[N]\right)_{et,\bar{s}}\right)$
l'action de monodromie associée (c'est-à-dire, l'action de
$\pi_{1}(H_g',\bar{s})$ sur la fibre en $s$ de
$\mathcal{J}_{g}'[N]_{et,\bar{s}}$, vu comme faisceau localement
constant constructibe sur $H_g'$).

\begin{theorem}[Igusa, Ekedahl \cite{Ekedahl}]\label{monodromie} Supposons $g\geq 1$.

(1) Si $N=p^{e}$ est une puissance de $p$, le morphisme
$\rho_{g,p^{e}}$ est surjectif.

(2) Si $N=N'$ est premier à $p$, le morphisme $\rho_{g,N'}$ a
$\mathbf{Sp}(2g,(\mathbf{Z}/N\mathbf{Z}))$ pour image, où
$\mathbf{Sp}(2g,(\mathbf{Z}/N\mathbf{Z}))$ est le groupe
symplectique pour l'accouplement de Weil sur $\mathcal{J}_g'[N]$.
\end{theorem}

\begin{remark} On renvoie le lecteur à \cite{Ekedahl} pour une
preuve du théorème \ref{monodromie}. Le cas où $g=1$ et $N=p$ a
été d'abord prouvé par Igusa. Le cas où $N$ est premier à $p$ est
classique (\cite{DM}).
\end{remark}

\begin{remark}\label{lemme sur la mono} (1) Pour $g\geq 2$,
soit $\mathcal{X}_{g}'/H_g'$ la courbe universelle sur $H_g'$.
Alors $\mathcal{X}_{g}'$ représente le foncteur des courbes
semi-stables tricanoniquement plongées de genre $g$ munies d'un
point rationnel, dont la jacobienne soit une variété abélienne
ordinaire. Comme $\X_g'/H_g'$ est à fibres géométriques connexes,
étant donné un point géométrique $\bar{x}$ de $\X_g'$, le
morphisme naturel $\pi_1(\X_{g}', \bar{x})\rightarrow
\pi_1(H_g',\bar{x})$ est surjectif. Par suite, l'action de
monodromie associée à la famille universelle de $\X_g'$ est aussi
surjective en vertu du théorème \ref{monodromie}.

(2) Par conséquent, pour $g\geq 2$, il existe une courbe $C_g$
semi-stable définie sur un corps $K_1$, telle que $C_g$ soit
constituée de deux composantes irréductibles: (i) une courbe $X'$
lisse de genre $g-1$ définie sur $K_1$ munie d'un point rationnel;
(ii) une courbe elliptique $E$ sur $K_1$, qui se coupent en un
point rationnel. De plus l'action de
$\mathrm{Gal}(\overline{K_1}/K_1)$ sur
$J_{C_g}[p](\overline{K_1})\simeq J_{X'}[p](\overline{K_1})\times
J_{E}[p](\overline{K_1})$ a
$\mathrm{Aut}_{\mathbf{F}_p}(J_{X'}[p](\overline{K_1}))\times
\mathrm{Aut}_{\mathbf{F}_p}(J_{E}[p](\overline{K_1}))$ comme image
dans $\mathrm{Aut}_{\mathbf{F}_p}(J_{C_g}[p](\overline{K_1}))$.
\end{remark}

\subsubsection{Préliminaires sur les points doubles ordinaires.}

Dans ce $\S$, nous rappelons la notion de point double ordinaire
(en codimension $1$) et leur déformation.

\begin{definition}[point double ordinaire]\label{pt double}
Soient $F$ un corps, $X/F$ un schéma localement de type fini. Un
point singulier $x\in X$ de codimension $1$ est \textit{un point
double ordinaire} s'il existe un $F$-schéma lisse
$Y=\mathrm{Spec}(A)$ et un diagramme commutatif suivant:
$$
\xymatrix{X\ar[d]&&U\ar[ll]_{\alpha}\ar[d]^{\beta}\\
\mathrm{Spec}(F)& Y\ar[l]& \mathrm{Spec}(A[u,v]/(uv))\ar[l]}
$$
dans lequel $\alpha$ et $\beta$ sont des morphismes étales, et
$x\in \alpha(U)\subset X$.
\end{definition}

Soit $R$ un anneau de valuation discrète, de corps résiduel $k$ et
de corps de fractions $K$. Notons $\mathfrak{m}$ l'idéal maximal
de $R$, $\varpi\in \mathfrak{m}$ une uniformisante de $R$. Par
convention, on note $\varpi^{\infty}=0\in R$.

\begin{notation} Pour $n\in \mathbf{Z}_{\geq 1}\cup \{\infty\}$,
notons $C_n=\mathrm{Spec}(R[u,v]/(uv-\varpi^{n}))$.
\end{notation}

\begin{remark} Pour $n\geq 1$ et $n\neq \infty$, $C_n$ est un
schéma normal, tandis que $C_{\infty}$ a deux composantes
irréductibles qui se coupent transversalement le long de $u=v=0$.
\end{remark}

Notons $S=\mathrm{Spec}(R)$, $\eta$ (resp. $s$) le point générique
(resp. le point fermé) de $S$.

\begin{prop}[déformation des points doubles ordinaires]
\label{deformation de point double ordinaire} Soient $X/S$ un
$S$-schéma plat, $x\in X_s$ un point ordinaire double.

(1) Il existe alors un $n\in \mathbf{Z}_{\geq 1}\cup \{\infty\}$,
un $S$-schéma lisse $Y$, et un diagramme commutatif:
$$
\xymatrix{X\ar[d]& Z\ar[l]_{\alpha'}\ar[d]^{\beta'} \\ S &
Y\times_S C_n\ar[l]}
$$
avec $\alpha'$ et $\beta'$ des morphismes étales tels que $x\in
\alpha'(Z)$.

(2) Si $n\neq \infty$, il existe un voisinage ouvert $U$ de $x$
dans $X$ qui est normal et  tel que $U_{\eta}$ soit lisse sur
$\eta$.

(2)' (Cas équisingulier) Si $n=\infty$, il existe un voisinage
ouvert $U$ de $x$ dans $X$ dont le normalisé $\widetilde{U}$ est
lisse sur $S$ et a pour fibre spéciale le normalisé de $U_s$.
\end{prop}

\begin{proof}Vu la définition d'un point double ordinaire (\ref{pt double}), le corps résiduel $k(x)$ en $x\in X$
est une extension séparable de $k(s)=k$, donc une extension finie
étale d'une extension transcendante pure $k[t_1,\cdots,t_n]$ de
$k$. Quitte à restreindre à un voisinage ouvert de $x$ dans $X$,
on peut supposer que les $t_i$ se relèvent en des fonctions $T_i$
sur $X$. Considérons le morphisme $\phi:X\rightarrow
\mathrm{Spec}(R[t_1,\cdots,t_n])=\mathbf{A}_S^{n}$ défini par les
$T_i$, et soit $\xi\in \mathbf{A}_{S}^{n} $ le point générique de
la fibre spéciale de $\mathbf{A}^{n}_{S}$. Alors $\xi=\phi(x)$ et
chaque composante irréductible de $X$ domine $\mathbf{A}_S^{n}$
par $\phi$. Soit $R'$ l'anneau local de $\mathbf{A}^{n}_{S}$ en
$\xi$. Donc $R'$ est un anneau de valuation discrète. Soit
$X'=X\times_{\mathbf{A}^{n}_{S}}\mathrm{Spec}(R')$. Alors $X'$ est
maintenant une $R'$-courbe plate qui présente un point double
ordinaire en $x$. Quitte à faire une localisation étale
$Y\rightarrow \mathbf{A}_{S}^{n}$ avec $\xi'$ un point au-dessus
de $\xi$, on peut remplacer $\mathbf{A}_S^{n}$ par un $S$-schéma
lisse $Y$ de $R'$ tel que l'extension $k(\xi')\rightarrow k(x)$
soit triviale. On est donc ramené au cas d'une courbe relative, un
cas qui est bien connu. On renvoie à \cite{de Jong}($\S$ 2.23)
pour les détails. Les autres assertions sont immédiates.
\end{proof}

\subsection{Théorème d'irréductibilité et de normalité} Dans cette
section, $k$ est un corps algébriquement clos de caractéristique
$p>0$. Si $Y$ est un schéma, et $y\in Y$, on note $k(y)$ le corps
résiduel de $Y$ en $y$.

\begin{lemme}\label{lemme technique} Soit $g\geq 2$ un entier. Il existe un $k$-schéma
$S=\mathrm{Spec}(R)=\{\eta,s\}$ spectre d'un anneau de valuation
discrète complet et une courbe semi-stable $X/S$ de genre $g$
telle que

(1) La fibre générique $X_{\eta}$ de $X/S$ est lisse.

(2) La jacobienne $J_1$ de $X_1$ est propre sur $S$.

(3) La fibre spéciale $X_s$ est constituée de deux composantes:
(i) une courbe générique $Y$ de genre $g-1$, et (ii) une courbe
elliptique générique $E$ qui se coupent transversalement en un
point rationnel. De plus, l'action de
$\mathrm{Gal}(\overline{k(s)}/k(s))$ sur les points d'ordre $p$ de
$J_{1,s}\simeq J_{Y,1}\times E_1$ a pour image
$\mathrm{Aut}_{\mathbf{Z}/p\mathbf{Z}}(J_{Y,1}[p](\bar{s}))\times
\left(\mathbf{Z}/p\mathbf{Z}\right)^{\ast}$.

\end{lemme}

\begin{proof} On a $g\geq 2$. Reprenons la courbe $C_g/h$ au-dessus
de $h=\mathrm{Spec}(K_1)$ construite dans la remarque \ref{lemme
sur la mono}. Choisissons un plongement tricanonique de $C_{g}/h$,
ce qui définit un morphisme de $h$ dans $H_{g}'$. Soit $h'$ son
image. En $h'$, $H_{g,\mathrm{sing}}'$ est un diviseur lisse de
$H_{g}'$(théorème \ref{moduli}). Soit $s$ le point générique de ce
diviseur. On prend $R$ le complété de l'anneau local de $H_g'$ en
$s$. Alors $R$ est un anneau de valuation discrète complet. Posons
$S=\mathrm{Spec}(R)=\{\eta,s\}$, $X=\X_{g}'\times_{H_{g}'}S$.
Ainsi $X/S$ est une courbe stable de fibre générique la courbe
générique de genre $g$, et sa jacobienne est un schéma abélien sur
$S$. De plus, $X_s$ est constituée de deux composantes: (i) une
courbe générique $Y$ de genre $g-1$, et (ii) une courbe elliptique
générique $E$. Ces deux composantes se coupent transversalement en
un point rationnel. Il reste à vérifier les conditions sur
l'action de monodromie. Par construction (remarque \ref{lemme sur
la mono}), le morphisme
$\mathrm{Gal}(\overline{k(h)}/k(h))\rightarrow
\mathrm{Aut}\left(J_{C_g,1}[p](\overline{k(h)})\right)$ a
$\mathrm{Aut}_{\mathbf{Z}/p\mathbf{Z}}(J_{Y,1}[p](\bar{h}))\times
\left(\mathbf{Z}/p\mathbf{Z}\right)^{\ast}$ comme image. Soit
$U\subset H_g'$ le lieu lisse du diviseur passant par $h'\in
H_g'$, alors $h'\in U$. Considérons le diagramme commutative
suivant
$$
\xymatrix{\mathrm{Gal}(\overline{k(h)}/k(h))=\pi_1(h,\bar{h})\ar[r]\ar[rd]&
\pi_1(U,\bar{h})\ar[d]^{\rho'}\\ &
\mathrm{Aut}\left(J_{\mathcal{C}_g,1}[p](\overline{k(h)})\right)},
$$
on en déduit que $\mathrm{Aut}(J_{Y,1}[p](\bar{h}))\times
(\mathbf{Z}/p\mathbf{Z})^{\ast}\subset \mathrm{im}(\rho')$. Or
$X_h$ est une courbe stable sur $k(h)$ constituée de deux
composantes lisses se coupant transversalement en un point
rationnel, on a $\mathrm{Aut}(J_{Y}[p](\bar{h}))\times
(\mathbf{Z}/p\mathbf{Z})^{\ast}\supset \mathrm{im}(\rho')$. Donc
$\rho'$ a
$\mathrm{Aut}_{\mathbf{Z}/p\mathbf{Z}}(J_{Y,1}[p](\bar{h}))\times
\left(\mathbf{Z}/p\mathbf{Z}\right)^{\ast}$ comme image. Puisque
$s\in U$ est le point générique, par généralisation de $h'$ à $s$,
l'action de $\mathrm{Gal}(\overline{k(s)}/k(s))$ sur les points
d'ordre $p$ de $J_{1,s}\simeq J_{Y,1}\times E_1$ a pour image
$\mathrm{Aut}_{\mathbf{Z}/p\mathbf{Z}}(J_{Y,1}[p](\bar{s}))\times
\left(\mathbf{Z}/p\mathbf{Z}\right)^{\ast}$.

\end{proof}

\begin{theoreme}\label{irr de theta gen} Si $g=2$, le diviseur $\Theta_{\mathrm{gen}}$ de la courbe générique est
lisse, et donc géométriquement irréductible.
\end{theoreme}

\begin{proof}
Considérons la courbe $X/S$ construite dans le lemme précédent.
Soit $k'/k(s)$ une extension finie galoisienne de $k(s)$ de groupe
$(\mathbf{Z}/p\mathbf{Z})^{\ast}\times
(\mathbf{Z}/p\mathbf{Z})^{\ast}$ qui déploie les points d'ordre
$p$ de $J_{1,s}$. Comme $S$ est complet donc hensélien,
l'extension $k'/k(s)$ s'étend en un revêtement étale galoisien
$S'/S$ de groupe $(\mathbf{Z}/p\mathbf{Z})^{\ast}\times
(\mathbf{Z}/p\mathbf{Z})^{\ast}$. Soient $(b_i)_{i=1,\cdots,p-1}$
et $(b_i')_{i=1,\cdots,p-1}$ les points d'ordre $p$ de $E_{1}$ et
$E_1'$. Sur $s'$, le diviseur $\Theta_{s'}$ est tel que (remarque
\ref{theta pour les courbes stables})
$$\Theta_{s'}=
\Theta_{s}\times_{k(s)}k(s')
=\bigcup_{i=1,j=1}^{p-1}\left(\left(\{b_i\}\times
E_{1}'\right)\cup \left(E_1\times \{b_j'\}\right)\right).
$$
C'est une courbe semi-stable de points doubles $(b_i,b_j')_{i,j\in
\{1,\cdots, p-1\}}$ qui sont conjugués sous $G$. En particulier,
$\Theta_{s'}$ a seulement des points doubles ordinaires pour
singularités.

Soit $a\in \Theta_{s'}$ un point singulier, $a$ est donc un point
double ordinaire de $\Theta_{s'}$. D'après la déformation de
points doubles ordinaires (proposition \ref{deformation de point
double ordinaire}), il y a deux cas à distinguer:

(1) Ou bien il existe un voisinage ouvert $U$ de $a$ dans
$\Theta':=\Theta\times_S S'$ tel que $U_{\eta}'$ soit lisse sur
$\eta'$, auquel cas $\Theta_{\eta'}'$ et donc $\Theta_{\eta}$ sont
lisses.

(2) Ou bien, $\Theta'$ est équisingulier en $a$ sur $S'$, auquel
cas le normalisé $\widetilde{\Theta'}$ de $\Theta'$ est lisse sur
$S'$ et la fibre spéciale $\widetilde{\Theta'}_{s'}$ est le
normalisé de $\Theta_{s'}'$. Alors $\widetilde{\Theta'}_{s'}$ est
une somme disjointe de courbes irréductibles isomorphes à $E_1$ ou
$E_1'$. Par suite, les composantes de
$\widetilde{\Theta'}_{\eta'}$ (où $\eta'$ est le point générique
de $S'$) sont des courbes elliptiques. Or $J_{1}$ ne contient pas
de courbes elliptiques puisque
$\mathrm{NS}(J_{1,\eta}\times_{\eta} \bar{\eta})=\mathbf{Z}$
(\cite{Mori}), d'où une contradiction.

\end{proof}

\begin{theoreme}\label{nor de theta gen} Si $g\geq 3$, le diviseur $\Theta_{\mathrm{gen}}$ de la courbe générique
est géométriquement normal.
\end{theoreme}

\begin{proof} On raisonne par récurrence sur $g$. D'après le
théorème précédent, on sait que le diviseur thêta pour la courbe
générique de genre $2$ est lisse. Supposons maintenant $g\geq 3$
et que le théorème a été démontré pour les courbes génériques de
genre $\leq g-1$. Considérons la courbe $X/S$ construite dans le
lemme \ref{lemme technique}. Soit $K'/k(\eta)$ une extension
galoisienne de corps telle que les conditions suivantes soient
réalisée: notons $S'=\{\eta',s'\}$ le normalisé de $S$ dans $K'$,
alors (i) toutes les composantes irréductibles de $\Theta_{\eta'}$
sont géométriquement irréductibles; (ii) l'extension de corps
$k(s')/k(s)$ déploie les points d'ordre $p$ de $J_{s}$. Notons
$(b_i)_{i=1,\cdots,p-1}$ les points d'ordre $p$ de $E_{1,s'}$.

Montrons d'abord que $\Theta_{\eta'}$ est géométriquement intègre.
D'après la remarque \ref{theta pour les courbes stables}, la fibre
de $\Theta$ au-dessus de $s'$ est telle que
$$
\Theta_{s'}=\Theta_{Y_{s}}\times
E_{1,s'}\cup\left(\cup_{i=1}^{p-1}\left(J_{Y_1,s'}\times\{b_i\}\right)\right),
$$
où $\Theta_{Y_{s'}}$ est le diviseur thêta de la courbe générique
de genre $g-1$, et donc est normal. Par conséquent, $\Theta$ a une
fibre spéciale géométriquement réduite. Il reste à prouver que
$\Theta_{\eta'}$ est géométriquement irréductible. Soit
$\Theta_{\eta',i}~(i\in I)$ les composantes irréductibles de
$\Theta_{\eta'}$, munies de la structure de sous-schéma fermé
réduit. Comme $\mathrm{NS}(J_{1,\eta'})\simeq \mathbf{Z}$, elles
sont amples. Soit $\overline{\Theta_{\eta',i}}$ l'adhérence
schématique de $\Theta_{\eta',i}$ dans $\Theta':=\Theta\times_S
S'$ ($=$ le diviseur thêta pour la courbe $X'/S'$). Alors
$\overline{\Theta_{\eta',i}}$ a une fibre spéciale ample. Or
d'après la description explicite de $\Theta_{s'}$ donnée
ci-dessus, on a donc $\Theta_Y\times E_1\subset
(\overline{\Theta_i})_s$. Or par l'hypothèse de récurrence,
$\Theta_{s'}$ est réduit, ce qui implique que
$\mathrm{card}(I)=1$. Donc $\Theta_{\eta'}$ est géométriquement
irréductible.

Ensuite, on montre que $\Theta_{\eta'}$ est géométriquement
normal. Soit $Z$ le fermé de $\Theta_{s'}$ où $\Theta_{s'}$ n'est
pas lisse. Par hypothèse de récurrence, $\Theta_Y$ est normal,
donc les composantes de $Z$ de codimension $1$ dans $\Theta_{s'}$
sont les $\Theta_Y\times \{b_i\}$ ($1\leq i\leq p-1$). Soit
$\xi_i$ le point générique de $\Theta_Y\times \{b_i\}$. Comme les
$b_i$ ($1\leq i\leq p-1$) sont conjugués sous le groupe de
décomposition de $k(\eta')/k(\eta)$, les $\xi_i$ le sont aussi. En
$\xi_i$, $\Theta_{s'}$ présente un point double ordinaire de
codimension $1$. Soit $\pi:\widetilde{\Theta}'\rightarrow \Theta'$
la normalisation. Il y a deux cas à distinguer:

(1) $\Theta'$ est normal en chacun des $\xi_i$, auquel cas,
$\Theta'$ est normal, et en particulier,
$\Theta_{\eta'}=\Theta_{\eta'}'$ est normal. De plus,
$\Theta_{\eta'}$ est en fait lisse en codimension $\leq 1$ (lemme
\ref{deformation de point double ordinaire} (2)), donc est
géométriquement normal en vertu du critère de normalité de Serre.

(2) Il existe un point $\xi_{i_0}$ où $\Theta'$ est équisingulier
(au sens de lemme \ref{deformation de point double ordinaire}
(2')). Par l'action de monodromie, $\Theta'$ est équisingulier en
tous les $\xi_i$. Alors $\widetilde{\Theta}'$ est lisse sur $S'$
au-dessus d'un voisinage ouvert $V_i$ de $\xi_i$, et la fibre
spéciale $\widetilde{\Theta}_{s'}$ de $\widetilde{\Theta}'$ est la
normalisé de $\Theta_{s'}$ au-dessus de $V_{s'}$ (d'après la
proposition \ref{deformation de point double ordinaire}).

\begin{lemme} Il existe un fermé $\widetilde{W}$ de $\widetilde{\Theta}_{s'}$ de
codimension $\geq 2$ dans $\widetilde{\Theta}_{s'}$ tel que
$\widetilde{\Theta}_{s'}-\widetilde{W}$ ait au moins $2$
composantes connexes.
\end{lemme}

\begin{proof}[Démonstration du lemme] Soit
$W=Z-(\cup_{i=1}^{p-1}V_i)$,
$\widetilde{W}=\pi_{s'}^{-1}(W)\subset \widetilde{\Theta}$. Alors
$\widetilde{W}$ est de codimension $\geq 2$ dans $\Theta_{s'}$. De
plus, $\widetilde{\Theta}_{s'}-\widetilde{W}$ est le normalisé de
$\Theta_{s'}-W$, et il est aussi lisse en les points au-dessus de
$\xi_i$ ($1\leq i\leq p-1$). Par conséquent,
$\widetilde{\Theta}_{s'}-\widetilde{W}$ n'est pas connexe (en
effet, $\widetilde{\Theta}_{s'}-\widetilde{W}$ a $p$ composantes
connexes).
\end{proof}

On termine la preuve par la proposition ci-après que l'on applique
à $Z=\widetilde{\Theta}'$ et $F=\widetilde{W}$. On voit que comme
$\widetilde{\Theta}_{s'}-\widetilde{W}$ n'est pas connexe par le
lemme ci-dessus, $\widetilde{\Theta}_{\eta'}$ n'est pas connexe.
Or $\Theta_{\eta'}$ est géométriquement intègre d'après ce que
l'on a montré au début de la preuve, $\widetilde{\Theta}_{\eta'}$
l'est aussi, d'où une contradiction. Ceci finit la preuve.
\end{proof}

\begin{proposition}[Utilisation de SGA$2$]Soit
$S=\mathrm{Spec}(R)$ le spectre d'un anneau de valuation discrète
complet, de point fermé $s$, de point générique $\eta$ et
d'uniformisante $\varpi$. Soit $f:Z\rightarrow S$ un schéma propre
et plat à fibre générique équidimensionnelle de dimension $\geq
2$. On suppose que $Z$ est normal. Soit $F$ un fermé de $Z_s$ tel
que $\mathrm{Codim}(F,Z_s)\geq 2$, et soit $V=Z-F$. Notons
$\widehat{V}$ le complété formel de $V$ le long de $V_s$, alors

(1) L'application canonique $H^{0}(V,\OO_{V})\rightarrow
H^{0}(\widehat{V},\OO_{\widehat{V}})$ est bijective.

(2) $Z_{\eta}=V_{\eta}$ est connexe si et seulement si $V_s$ est
connexe.
\end{proposition}

\begin{proof} L'assertion (2) est une conséquence immédiate de (1) et
du fait que les composantes connexes de $\widehat{V}$ sont celles
de $V_s$. Prouvons (1) comme cas particulier de SGA2 IX corollaire
1.2. Notons $j:V\rightarrow Z$ l'immersion ouverte, montrons
d'abord que les faisceaux $(f\circ j)_{\ast}(\OO_{V})$ et
$\mathrm{R}^{1}(f\circ j)_{\ast}(\OO_V)$ sont des faisceaux
cohérents sur $Z$ comme conséquence de SGA2 VIII, théorème 3.1. En
effet, pour tout $x\in V$ tel que
$\mathrm{codim}(\overline{\{x\}}\cap (X-V),\overline{\{x\}})=1$,
on a $\mathrm{dim}(\OO_{X,s})\geq 2$. Comme $V$ est normal,
$\mathrm{prof}_x(\OO_{V})\geq 2$. Appliquons SGA2 VIII théorème
3.1 à la situation où $X=Z$, $Y=S$, $U=V$, $F=\OO_V$ et $n=2$, on
trouve que les faisceaux $(f\circ j)_{\ast}(\OO_{V})$ et
$\mathrm{R}^{1}(f\circ j)_{\ast}(\OO_U)$ sont des faisceaux
cohérents sur $Z$. Par SGA2 IX corollaire 1.2, on en conclut que
l'application canonique
$$
H^{0}(V,\OO_V)\rightarrow \varprojlim
H^{0}(V_n,\OO_{V_n})=H^{0}(\widehat{V},\OO_{\widehat{V}})
$$
est bijective (où $V_{n}=V\otimes_R R/\varpi^{n+1} $). D'où le
résultat.

\end{proof}

\begin{remarque} En utilisant le même genre de dégénérescence, on
peut majorer les multiplicités du $\Theta_{\mathrm{gen}}$. En
effet, si $g=2n$ (resp. $g=2n+1$), on dégénère la courbe générique
en une chaîne de $n$ courbes génériques de genre $2$ (resp. en une
chaîne de $n$ courbes génériques de genre $2$ et une courbe
elliptique), on en déduit que en un point de
$\Theta_{\mathrm{gen}}$, la multiplicité est $\leq n$ (resp. $\leq
n+1$). D'où le corollaire suivant:
\end{remarque}

\begin{corollaire} Pour une courbe générique de genre $g=2n$ (resp.
de genre $g=2n+1$), et tout $L$ inversible de degré $0$ sur $X_1$,
on a $h^{0}(X_1,B\otimes L)\leq n$ (resp. $h^{0}(X_1,B\otimes
L)\leq n+1$).
\end{corollaire}

\subsection{Compléments}
\subsubsection{Une propriété de $\QQ$.}

\begin{prop}\label{Q ne provient pas de J, cas generique}
Pour une courbe $X$ générique de genre $g\geq 2$, le faisceau
$\QQ$ ($\S$ \ref{accouplement}) ne provient pas d'un faisceau
inversible sur $J_1$.
\end{prop}

\begin{proof} C'est clair pour $g\geq 4$ car $\QQ$ n'est pas
inversible (proposition \ref{purete}). Supposons $g=2$ ou $3$ et
que $\QQ$ soit inversible sur $\Theta$ si $g=3$ (rappelons que
c'est automatique si $g=2$ puisque $\Theta_{\mathrm{gen}}$ est
lisse). On raisonne par l'absurde. Supposons que $\QQ$ provienne
d'un faisceau inversible $\mathcal{L}$ sur $J_1$. Comme
$\mathrm{NS}(J_1)\simeq \mathbf{Z}$ avec $\Theta_{\mathrm{class}}$
comme générateur (\cite{Mori}), il existe $r\in \mathbf{Z}$ tel
que $\mathcal{L}$ est algébriquement équivalent à
$\OO_{J_1}(r\Theta_{\mathrm{class}})$. On a aussi
$(-1)^{\ast}\mathcal{L}\equiv
\OO_{J_1}(r\Theta_{\mathrm{class}})$. Comme
$(-1_{\Theta})^{\ast}\QQ\otimes \QQ\simeq
\omega_{\Theta}=\OO_{J_1}(\Theta)|_{\Theta}$ ($\S$
\ref{accouplement}), $\OO_{J_1}(2r\Theta_{\mathrm{class}})$ et
$\OO_{J_1}((p-1)\Theta_{\mathrm{class}})$ ont des restrictions à
$\Theta$, qui sont algébriquement équivalentes. On en déduit que
$2r=p-1$ (car $\Theta_{\mathrm{class}}$ est ample), ce qui n'est
impossible que pour $p\geq 3$. Supposons $p\geq 3$, on trouve que
$r=(p-1)/2$. On en déduit que $\mathcal{L}$ et
$\left(\mathcal{O}_{J_1}(-\Theta)\otimes \mathcal{L}\right)^{-1}$
sont amples sur $J_1$. Considérons la suite exacte
$$
0\rightarrow \OO_{J_1}(-\Theta)\otimes \mathcal{L}\rightarrow
\mathcal{L}\rightarrow \mathcal{Q}\rightarrow 0,
$$
et appliquons lui la transformation de Fourier-Mukai (inverse), on
obtient la suite exacte suivante
$$
0=\mathcal{F}^0(\OO_{J_1}(-\Theta)\otimes \mathcal{L})\rightarrow
\mathcal{F}^{0}(\mathcal{L})\rightarrow
\mathcal{F}^{0}(\QQ)\rightarrow
\mathcal{F}^{1}(\OO_{J_1}(-\Theta)\otimes \mathcal{L}).
$$
Comme $(\OO_{J_1}(-\Theta)\otimes \mathcal{L})^{-1}$ est ample, et
$J_1$ est de dimension $g\geq 2$, on sait que
$\mathcal{F}^{1}(\OO_{J_1}(-\Theta)\otimes \mathcal{L})=0$. Donc
$\mathcal{F}^{0}(\mathcal{L})\simeq
\mathcal{F}^{0}(\mathcal{\QQ})\neq 0$. Or
$\mathcal{F}(\mathcal{Q})=\mathcal{F}\circ\mathcal{F}(B)[1]=(-1)^{\ast}B[1-g]$
(\cite{Mukai}), on en déduit que $\mathcal{F}^{0}(\mathcal{Q})=0$,
d'où une contradiction.
\end{proof}

\subsubsection{Etude des points de torsion de
$\Theta_{\mathrm{gen}}$.} D'après \ref{Dirac}, le diviseur
$\Theta$ satisfait à la propriété de Dirac, a fortiori, $\Theta$
contient tous les points d'ordre $p$ de $J_1(k)$. Pour une courbe
générique, on a le résultat suivant:

\begin{prop} Le diviseur $\Theta_{\mathrm{gen}}$ de la
courbe générique de genre $g\geq 1$ contient comme seuls points
d'ordre fini les points d'ordre $p$ de $J_1$. De plus,
$\Theta_{\mathrm{gen}}$ est lisse en ces points.
\end{prop}

\begin{proof}
Dans le cas où $g=1$, le résultat est clair. Supposons maintenant
que $g\geq 2$. Soit $\eta$ le point générique de l'espace de
modules des courbes de genre $g$. Soit $x\in \Theta(\bar{\eta})$
un point d'ordre fini égal à $N=p^{e}M$ avec $e,M$ des entier tels
que $e\geq 0, M\geq 1, (e,M)=1$. Fixons un isomorphisme
$J_{\eta,1}[N](\bar{\eta})\simeq
J_{\eta,1}[p^{e}](\bar{\eta})\times J_{\eta,1}[M](\bar{\eta})$
donné par $x\in J_{\eta,1}[N](\bar{\eta})\mapsto (x_p,x_M)\in
J_{\eta,1}[p^{e}](\bar{\eta})\times J_{\eta,1}[M](\bar{\eta})$.
D'après le théorème \ref{monodromie},
$\mathrm{Gal}(\bar{\eta}/\eta)$ agit transitivement sur les points
d'ordre $M$ (resp. sur les points d'ordre $p^{e}$), donc tout
point d'ordre $M$ (resp. d'ordre $p^{e}$) apparaît comme la
composante $M$-primaire (resp. $p$-primaire) d'un certain point
d'ordre $N$ contenu dans $\Theta_{\eta}=\Theta_{\mathrm{gen}}$. Or
on peut trouver $X/S$ une courbe semi-stable telle que $X_{\eta}$
soit générique, et que $X_s$ soit une chaîne de $g$ courbes
elliptiques ordinaires. $x$ se spécialise en un point d'ordre
exactement $N$ de $J_{1,s}$. Si $M\neq 1$, par conjugaison des
points d'ordre $M$, on peut supposer que la spécialisation
$(y_1,\cdots,y_g)\in E_{1,1}\times \cdots \times E_{1,g}\simeq
J_{1,s}$ de $x$ est telle que chacun des $y_i$ ($i=1,\cdots,g$)
ait sa composante $M$-primaire d'ordre exactement $M$. Par
ailleurs, par la description de $\Theta_s$ (remarque \ref{theta
pour les courbes stables}), les seuls points d'ordre fini sont les
points $(y_1,\cdots,y_g)\in E_{1,1}\times \cdots\times
E_{1,g}\simeq J_{1,s}$ dont au moins un des $y_i$ est exactement
d'ordre $p$. D'où une contradiction. Donc $M=1$ et $N=p^{e}$ est
une puissance de $p$. Si $e\geq 2$, de la même façon, par
conjugaison des points d'ordre $N=p^{e}$, on peut supposer que la
spécialisation $(y_1,\cdots,y_g)\in E_{1,1}\times \cdots\times
E_{1,g}$ de $x$ est telle que chacun des $y_i$ ($i=1,\cdots,g$)
soit d'ordre exactement $p^{e}$, ce qui contredit la description
explicite des points d'ordre fini de $\Theta_{s}$ (remarque
\ref{theta pour les courbes stables}). Par conséquent, on a $N=p$.

Pour la lissité de $\Theta_{\eta}$ en les points d'ordre $p$, il
suffit de remarquer que $\Theta_s$ contient au moins un point $y$
d'ordre $p$ où $\Theta_s$ est lisse. Soit $x\in \Theta_{\eta}$ le
point d'ordre $p$ qui se spécialise en $y\in \Theta_s$, alors
$\Theta_{\eta}$ est lisse en $y$, et donc lisse en tous les points
d'ordre $p$ par l'action de monodromie.
\end{proof}

%% file: chapitre4.tex
\section{Cas où $g=2$ ou $p=3$} On donne des résultats sur le \
diviseur thêta $\Theta$ dans le cas où $g=2$ ou $p=3$.
\markboth{chapitre4}{chapitre4}

\subsection{Diviseur thêta en caractéristique $3$.}

Dans cette section, $k$ est un corps algébriquement clos de
caractéristique $p=3$, $X/k$ est une courbe propre lisse connexe
sur $k$. Notons $B$ le fibré des formes différentielles localement
exactes sur $X_1$, et $\Theta=\Theta_B$ le diviseur thêta associé
à $B$. Comme $p=3$, $B$ est un fibré vectoriel de rang $2$ sur
$X_1$.

\subsubsection{Préliminaires.}

\subsubsection*{Quelques rappels.}

Soit $x$ un point de $\Theta$, notons $L$ le faisceau inversible
de degré $0$ sur $X_1$ correspondant. D'après le critère de
lissité, pour que $x$ soit un point singulier de $\Theta$, il faut
et il suffit que l'une des deux conditions suivantes soit
réalisée:
\begin{itemize}
\item soit $h^{0}(X_1,B\otimes L)\geq 2$,

\item soit $h^{0}(X_1,B\otimes L)=1$ (donc on a aussi
$h^{0}(X_1,B\otimes L^{-1})=1$). Et si l'on note
$\tau:L^{-1}\hookrightarrow B$ et $\tau':L\hookrightarrow B$ les
uniques (à une multiplication par un scalaire près) plongements de
$L^{-1}$ dans $B$ et de $L$ dans $B$, on a en plus que $L$ et
$L^{-1}$ sont orthogonaux par rapport à l'accouplement
antisymétrique $(\cdot,\cdot):B\otimes B\rightarrow
\Omega^{1}_{X_1/k}$ sur $B$ (\ref{critere de lissite}).

\end{itemize}

En général, soit $\tau:L^{-1}\hookrightarrow B$ un plongement de
$L^{-1}$ dans $B$, qui correspond à un élément non-nul de
$H^{0}(X_1,B\otimes L)$. Soit $M_{\tau}$ le sous-fibré en droites
de $B$ contenant $\tau(L^{-1})\subset B$. Puisque $B$ est un fibré
de rang $2$ sur $X_1$, $M_{\tau}$ est l'orthogonal dans $B$ de
$\tau(L^{-1})\subset B$ par rapport à l'accouplement
anti-symétrique non-dégénéré $(\cdot,\cdot)$ sur $B$. En
particulier, si $\tau':L\hookrightarrow B$ est un plongement de
$L$ dans $B$ tel que $\tau(L)^{-1}$ et $\tau'(L)$ soient
orthogonaux, alors $M_{\tau}$ contient à la fois $\tau(L^{-1})$ et
$\tau'(L)\subset B$.

Le plongement $\tau:L^{-1}\hookrightarrow B$ correspond aussi à un
point (noté encore $\tau$) de $\mathcal{H}$ (schéma de Hilbert des
faisceaux inversibles de degré $0$ plongés dans $B$, $\S$ \ref{le
schema H}) au-dessus de $x\in \Theta$. Pour que
$\mathrm{Hom}(L,M_{\tau})=0$ (c'est-à-dire, pour qu'il n'y ait pas
de $\tau':L\hookrightarrow B$ qui soit orthogonal à $\tau$), il
faut et il suffit que $\mathcal{H}$ soit lisse en $\tau$ de
dimension $g-1$ (proposition \ref{etude diff de H}).

\subsubsection*{{\textquotedblleft Le schéma des différences\textquotedblright}.}

\begin{definition}[Schéma des différences]\label{schema de diff}Pour $d>0$ un entier, on note $V_d$ l'image
schématique du morphisme
\begin{eqnarray*}
\phi_d:X_{1}^{d}\times_k X_{1}^{d}&\rightarrow& J_1 \\
(x_1,\cdots,x_{d}, y_1,\cdots,y_{d})&\mapsto&
\OO_{X}\left(x_1+\cdots x_{d}-(y_1+\cdots y_{d})\right).
\end{eqnarray*}
Définissons $V_{d}'$ par le carré cartésien suivant
$$
\xymatrix{\ar@{} |\Box[rd]V_{d}'\ar[r]^{i'}\ar[d] & J_1\ar[d]^{[2]}\\
V_d\ar[r]^{i} & J_1}
$$
\end{definition}

Les propriétés suivantes des schémas des différences sont faciles
à vérifier.
\begin{lemma} \label{prop de schema de diff}(1) Pour $1\leq d\leq g/2$,
$\mathrm{dim}(V_{d}')=\mathrm{dim}(V_d)=2d$. Pour $d>g/2$, on a
$\mathrm{dim}(V_d')=g$.

(2) Pour $d>0$, $V_d'$ est irréductible.

(3) Pour $d>0$, $V_d$ contient un translaté de $X_1$, par suite
$i':V_d'\hookrightarrow J_1$ induit une surjection sur les groupes
fondamentaux $\pi_1(V_d')\rightarrow \pi_1(J_1)$.
\end{lemma}

\subsubsection*{Une limitation des sous-faisceaux inversibles de
$B$.}\label{limitation}

Les considérations qui suivent valent pour toute caractéristique
$p\geq 2$, mais nous n'avons des applications que pour $p=3$, et
nous nous limitons à ce cas.

Pour $m\geq 0$ un entier, soit $\Theta_m$ le fermé de
$J_{1}^{[m]}$ formé des faisceaux inversibles $M$ sur $X_1$ de
degré $m$ qui admettent un plongement $M\hookrightarrow B$. En
particulier, $\Theta_0$ est l'ensemble sous-jacent à $\Theta$.
Nous allons donner une limitation à priori de $\Theta_m$.

Soit $M$ un faisceau inversible sur $X_1$ de degré $m$. Alors un
plongement $M\hookrightarrow B\hookrightarrow
F_{\ast}\Omega^{1}_{X/k}$ donne par adjonction un plongement
$F^{\ast}M\hookrightarrow \Omega^{1}_{X/k}$, et donc une
réalisation $F^{\ast}M\simeq \Omega^{1}_{X/k}(-D)$ où $D$ est un
diviseur positif ou nul sur $X$ de degré $2g-2-3m$. Les diviseurs
$D$ en question sont paramétrés par $X^{(2g-2-3m)}$. Son image
schématique dans $J^{[2g-2-3m]}$ est un fermé irréductible normal
de dimension $\mathrm{min}(2g-2-3m, g)$. Par suite les faisceaux
$F^{\ast}M\simeq \Omega^{1}_{X_1/k}(-D)$ sont paramétrés par un
fermé irréductible et normal de $J^{[3m]}$ de dimension
$\mathrm{min}(2g-2-3m,g)$, que nous notons $A_{3m,\mathrm{max}}'$.
Le passage $M\mapsto F^{\ast}M$ est donné par le Verschiebung
$V:J^{[m]}_{1}\rightarrow J^{[3m]}$. Finalement $M$ appartient à
$A_{m,\mathrm{max}}:=\left(V^{-1}(A_{3m,\mathrm{max}}')\right)_{\mathrm{red}}$.

\begin{lemma} $A_{m,\mathrm{max}}$ est équidimensionnel de
dimension $\mathrm{min}(2g-2-3m,g)$, irréductible lorsque
$2g-2-3m>0$.
\end{lemma}

\begin{proof} En effet, $V:J_{1}^{[m]}\rightarrow J^{[3m]}$ se
factorise en un morphisme radiciel et un morphisme $V_{et}$. Pour
établir le lemme, il suffit de considérer $V_{et}$.
L'irréductibilité résulte de \cite{Milne} (proposition 9) et du
fait que pour $2g-2-3m>0$, $A_{m,\mathrm{max}}$ contient des
translatés de $X_1$.
\end{proof}

\vspace{2mm}

On a donc $\Theta_{m}\subset A_{m,\mathrm{max}}$, ce qui fournit
une limitation à priori de la taille de $\Theta_m$. La discussion
dans les sections suivantes s'organise autour de la taille de
$\Theta_m$ dans $A_{m,\mathrm{max}}$.

Si $i:M\hookrightarrow B$ est un plongement d'un faisceau
inversible de degré $m\geq 0$, $M(-D)$ est contenu dans $B$ pour
tout $D$ effectif de degré $m$. Considérons l'application
$$
\alpha_m:X^{(m)}_1\times A_{m,\mathrm{max}}\rightarrow J_1
$$
donnée par $(D,M)\mapsto M(-D)$, et soit $B_{m,\mathrm{max}}$ son
image. Notons $\mathrm{Sat}(\Theta_m)$ l'image de
$X^{(m)}_{1}\times \Theta_{m}$ par $\alpha_m$. Comme $X_1$
engendre $J_1$, le fermé $B_{m,\mathrm{max}}$ est irréductible de
dimension $2g-2-2m$ du moins, tant que $2g-2-2m\leq g$,
c'est-à-dire $m\geq (g-2)/2$.


\subsubsection{$\Theta$ est réduit en caractéristique $3$ pour
$g\geq 2$.} \label{Theta est reduit}

Dans ce numéro, on montre que $\Theta$ est toujours réduit en
caractéristique $3$ pour $g\geq 2$. Supposons que $\Theta$ ne soit
pas réduit. Alors il existe une composante irréductible
$\widetilde{\Theta}$ (munie de la structure de sous-schéma fermé
réduit) de $\Theta$ de multiplicité $\geq 2$. D'après la
proposition \ref{polarisation}, on a alors
$\Theta=2\widetilde{\Theta}$, et donc $\widetilde{\Theta}$ est
algébriquement équivalent à $\Theta_{\mathrm{class}}$.

Soit $\eta$ le point générique de $\widetilde{\Theta}$. Comme
$\Theta$ est singulier en $\eta$, d'après le corollaire
\ref{lissite de point de codim 0}, il existe un plongement
$\tau:L_{\eta}^{-1}\hookrightarrow B_{\eta}$ (sur $X_1\times_k
\eta$) de saturation $M_{\tau}$, et un plongement
$\tau':L_{\eta}\hookrightarrow B_{\eta}$ de sorte que
$\tau'(L_{\eta})\subset M_{\tau}$. Soit
$m=\mathrm{deg}(M_{\tau})$. Comme $g\geq 2$, $L_{\eta}$ n'est pas
d'ordre divisant $2$, donc $m>0$. Il existe donc des diviseurs
effectifs $D_{\eta}$ et $D_{\eta}'$ de degré $m$ sur la courbe
$X_{1,\eta}:=X_1\times_k \eta$ tels que
$L_{\eta}(D_{\eta})=M_{\tau}$ et
$L^{-1}_{\eta}(D_{\eta}')=M_{\tau}$. En particulier, on trouve que
$M_{\tau}^{2}=\OO_{X_{1,\eta}}(D_{\eta}+D_{\eta}')$ et
$L^{2}_{\tau}=\OO_{X_{1,\eta}}(D_{\eta}-D_{\eta}')$.

Cette dernière condition signifie que $L_{\eta}^{2}$ est dans le
{\textquotedblleft schéma des différences\textquotedblright} $V_m$
(définition \ref{schema de diff}). Donc
$\overline{\{\eta\}}=\widetilde{\Theta}\subset V_m'$. Par suite
$g-1\leq \mathrm{dim}(V_{m}')=\mathrm{min}(2m,g)$ (\ref{prop de
schema de diff}), donc $2m\geq g-1$. Par ailleurs, $M_{\tau}\in
\Theta_{m}$, et les $L$ de $\widetilde{\Theta}$ de la forme
$M(-D)$, avec $M\in \Theta_m$ et $D$ diviseur effectif de degré
$m$ sur $X_1$, sont dans $B_{m,\mathrm{max}}$ de dimension
$\mathrm{min}(2g-2-2m,g)$. Comme $\widetilde{\Theta}$ est de
dimension $g-1$, on a $2m\leq g-1$.

Finalement, on a $2m=g-1$, ce qui achève la démonstration pour $g$
pair.

Pour $g$ impair, on a $m=(g-1)/2>0$. Comme $V_{m}'$ est
irréductible de dimension $g-1$, on a $V_{m}'=\widetilde{\Theta}$.
Or $V_{m}'$ est algébriquement équivalent à $2^{2}V_m=4V_m$. Comme
$\widetilde{\Theta}$ est algébriquement équivalent à
$\Theta_{\mathrm{class}}$ (remarque \ref{cas p=3 facile}), et que
la classe de la polarisation principale de
$\Theta_{\mathrm{class}}$ est non divisible dans
$\mathrm{NS}(J_1)$, on a une contradiction. Donc $\Theta$ est
réduit.

\subsubsection{Etude du lieu singulier de $\Theta$ de dimension
$g-2$.} \label{Theta est sans TAS} Dans ce numéro, on recherche
les composantes du lieu singulier $\Theta_{\mathrm{sing}}$ de
$\Theta$ de dimension $g-2$, qui existent si et seulement si
$\Theta$ est non normal, et on montre que que $\Theta$ ne contient
pas de composante irréductible qui est un translaté d'une
sous-variété abélienne de $J_1$ pour une courbe de genre $g\geq 3$
(le cas où $g=2$ sera traité dans $\S$ \ref{cas p=3 g=2}, où l'on
montrera que $\Theta$ est toujours intègre pour une courbe lisse
de genre $2$ en caractéristique $3$).

\begin{lemma} Supposons $p=3$. Soit $t$ un point de $\Theta$ de codimension
$1$ en lequel $\Theta$ n'est pas lisse. Soit $L_t$ le faisceau
inversible sur $X_{1,t}:=X_1\times_k t$ correspondant, alors il
existe $\xi\in \mathcal{H}$ ($\S$ \ref{le schema H}) au-dessus de
$t$, où $\mathcal{H}$ n'est pas lisse de dimension $g-1$. En
particulier, $\xi$ correspond à un plongement
$L_{t}^{-1}\hookrightarrow B_{t}$ de saturation $M_{\xi}$ avec
$L_{t}\subset M_{\xi}$ pour un plongement $L_t\hookrightarrow B_t$
convenable.
\end{lemma}

\begin{proof} Considérons le morphisme naturel surjectif
$\mathcal{H}\rightarrow \Theta$, et considérons la fibre
$\mathcal{H}_t$ au-dessus de $t$. Distinguons deux cas:

(1) Si $h^{0}(X_{1,t},B\otimes L_{t})=1$ (et donc
$h^{0}(X_{1,t},B\otimes L_{t}^{-1})=1$), alors le morphisme
$\mathcal{H}\rightarrow \Theta$ est un isomorphisme au-dessus d'un
voisinage de $t$, et en particulier, $\mathcal{H}$ n'est pas lisse
de dimension $g-1$ en l'unique point $\xi\in \mathcal{H}$
au-dessus de $t\in \Theta$. D'après proposition \ref{diff prop of
H}, l'unique plongement $L_{t}\hookrightarrow B_{t}$ est donc tel
que $L^{-1}_{t}\hookrightarrow M_{t}$.

(2) Supposons que $h^{0}(X_{1,t}, B\otimes L_{t})>1$, de sorte que
la fibre de $\mathcal{H}$ au-dessus de $t$ est de dimension $\geq
1$. Soit $\Theta_i$ une composante irréductible de $\Theta$
contenant $t$, $\eta_i$ son point générique. Comme $\Theta$ est
réduit (\ref{Theta est reduit}), le morphisme
$\mathcal{H}\rightarrow \Theta$ est un isomorphisme au-dessus d'un
voisinage de $\eta_i\in \Theta$. Soit $\tau_i$ l'unique point de
$\mathcal{H}$ au-dessus de $\eta_i$. Notons $Y$ l'adhérence
schématique de $\tau_i$ dans $\mathcal{H}$. Comme
$\mathcal{H}\rightarrow \Theta$ est propre, $Y\rightarrow
\Theta_i$ est surjectif. Soit $\xi$ un point de $Y$ au-dessus de
$t$. Comme $t$ est de codimension $1$ dans $\Theta$, $Y\rightarrow
\Theta_i$ est quasi-fini en $\xi$. Or $\mathcal{H}$ a une fibre
irréductible de dimension $\geq 1$ au-dessus de $t$, il y a donc
au moins deux composantes irréductibles de $\mathcal{H}$ qui
passent par $\xi$, et donc $\mathcal{H}$ n'est pas lisse en $\xi$.
Alors $\xi$ correspond à un plongement $\xi:L_{t}^{-1}\rightarrow
B_{t}$ de saturation $M_{\xi}$, tel qu'il existe un plongement
$L_{t}\hookrightarrow B_{t}$ de sorte que $L_{t}\hookrightarrow
M_{\xi}$ (proposition \ref{diff prop of H}), d'où le résultat.
\end{proof}

\vspace{2mm}La preuve du lemme suivant est immédiate.

\begin{lemma}\label{lemme facile}Soit $k$ un corps.

(1) Soient $A$ une variété abélienne, $Y=x+B$ un translaté d'une
sous-variété abélienne de dimension $<\mathrm{dim}(A)$ de $A$.
Alors le morphisme $\pi_{1}(Y)\rightarrow \pi_1(A)$ est de conoyau
infini.

(2) Soit $C$ une courbe lisse de genre $g\geq 2$ sur $k$, avec
$J_C$ sa jacobienne. Soit $F\subset J_C$ un fermé irréductible de
$J_C$ muni de la structure de sous-schéma réduit. Supposons que
$F$ contienne un translaté $\widetilde{C}$ de $C$, alors le
morphisme naturel $\pi_1(F)\rightarrow \pi_1(J_C)$ est surjectif.
En particulier, $F$ n'est pas un translaté d'une sous-variété
abélienne de $J_C$.
\end{lemma}

Soient $\Sigma$ une composante irréductible réduite de dimension
$g-2$ de $\Theta_{\mathrm{sing}}$, $t\in \Sigma$ son point
générique avec $L_t$ le faisceau inversible correspondant sur
$X_{1,t}:=X_1\times_k t$. D'après le lemme ci-dessus, il existe
$\xi:L_{t}^{-1}\hookrightarrow B_{t}$ de saturation $M_{\xi}$, et
$\xi':L_{t}\hookrightarrow B_t$ de sorte que $\xi'(L_t)\subset
M_{\xi}$. Soit $m$ le degré de $M_{\xi}$. Comme $g\geq 3$, $L_{t}$
n'est pas d'ordre divisant $2$, donc $m>0$. Par suite, il existe
des diviseurs $D_{t}$ et $D_t'$ positifs de degré $m$ de $X_{1,t}$
tels que $L_t(D_t)=M_{\xi}$ et $L^{-1}_{t}(D_t')=M_{\xi}$. En
particulier, $L^{2}_{t}=\OO_{X_{1,t}}(D_t-D_t')$,
$M_{\xi}^{2}=\OO_{X_{1,\xi}}(D_t+D_t')$. On a donc $t\in
V_{m}'\cap B_{m,\mathrm{max}}$, par suite $\Sigma\subset V_m'\cap
B_{m,\mathrm{max}}$. Par conséquent $(g-2)/2\leq m\leq g/2$ pour
une raison de dimension. Il y a deux cas à distinguer: le cas où
$g$ est pair et le cas où $g$ est impair.

\subsubsection*{Cas où $g$ est pair.}
On suppose donc $g\geq 4$, le cas où $g=2$ sera traité plus tard
au $\S$ \ref{cas p=3 g=2}. Alors ou bien $m=(g-2)/2$, ou bien
$m=g/2$.

Examinons d'abord le cas $m=(g-2)/2$. Alors $V_m'$ est un
sous-schéma fermé irréductible de $J_1$ de dimension $g-2$
(\ref{prop de schema de diff}). Donc $\Sigma =V_m'$. Alors une
composante irréductible $\Theta_i$ de $\Theta$ qui contient
$\Sigma=V_m'$ ne peut être un translaté d'une sous-variété
abélienne de $J_{1}$ car son image par la multiplication par $2$
serait aussi un translaté d'une sous-variété abélienne, qui
contiendrait $V_m$. Or $V_{m}$ pour $m>0$, contient un translaté
de $X_1$, d'où une contradiction (lemme \ref{lemme facile}).

Examinons maintenant le cas $m=g/2$. Alors $A_{m,\mathrm{max}}$
est équidimensionnel de dimension $2g-2-3m=\frac{g}{2}-2$,
irréductible pour $g\geq 6$, et $B_{m,\mathrm{max}}$ est
équidimensionnel de dimension $g-2$. Donc $\Sigma$ est une
composante irréductible de $B_{m,\mathrm{max}}$. En particulier,
$\Sigma$ contient un translaté de $X_1$. Donc, d'après le lemme
\ref{lemme facile}, une composante irréductible $\Theta_i$ de
$\Theta$ qui contient $\Sigma$ ne peut être un translaté d'une
sous-variété abélienne de $J_{1}$

Réciproquement, si une composante irréductible de
$A_{m,\mathrm{max}}$ est une composante de $\Theta_m$, alors
$\mathrm{Sat}(\Theta_m)$ sera une composante $\Sigma$ singulière
de $\Theta$, de dimension $g-2$, à condition que si $M\in
\Theta_{m}$, on ait $h^{0}(X_1,M^{2})>m$, et on donnera plus loin
un exemple pour $g=4$.

\subsubsection*{Cas où $g$ est impair.}
Supposons $g\geq 3$ impair. On doit examiner le cas $m=(g-1)/2$.
Alors $A_{m,\mathrm{max}}$ est irréductible de dimension
$(g-1)/2$, et $\mathrm{dim}(\Theta_{m})\leq (g-1)/2$. Par
ailleurs, $\Sigma$ est contenu dans $\mathrm{Sat}(\Theta_m)$, donc
$$
g-2=\mathrm{dim}(\Sigma)\leq
\mathrm{dim}(\mathrm{Sat}(\Theta_{m}))=\mathrm{dim}(\Theta_m)+(g-1)/2\leq
\mathrm{dim}(B_{m,\mathrm{max}})=g-1.
$$
Distinguons à nouveau deux cas:

\vspace{2mm}

Cas (i): $\mathrm{dim}(\Theta_m)=\frac{g-1}{2}-1$. Soient
$\Theta_{m,i}$ ($i\in I$) les composantes irréductibles de
$\Theta_{m}$, de dimension $\frac{g-1}{2}-1$. Alors il existe
$i_0\in I$ tel que $\Sigma=\alpha_m(X^{(m)}_1\times
\Theta_{m,i_0})$. En particulier, comme $m>0$, $\Sigma$ contient
un translaté de $X_1$, et aucune composante irréductible contenant
$\Sigma$ n'est un translaté d'une sous-variété abélienne de $J_1$
(lemme \ref{lemme facile}).

Cas (ii): $\mathrm{dim}(\Theta_m)=(g-1)/2$. Alors
$\Theta_m=A_{m,\mathrm{max}}$ et $\Theta':=\mathrm{Sat}(\Theta)$
est une composante irréductible de $\Theta$.

\begin{lemma} $\Theta'$ est la spécialisation ensembliste d'un
diviseur qui existe sur la jacobienne générique, par suite
$\Theta'$ est algébriquement équivalent à
$r\cdot\Theta_{\mathrm{class}}$ avec $r\geq 1$ un entier.
\end{lemma}

\begin{proof} Gardons les notations ci-dessus. Soient $S=\{\eta,s\}$
le spectre d'un anneau de valuation discrète complet contenant
$k$, $\X/S$ une courbe propre lisse de genre $g$ dont la fibre
générique est la courbe générique de genre $g$, et dont la fibre
$\X_s$ est $X_s=X$. Notons $\mathcal{J}$ la jacobienne de $\X/S$.
Soit $\mathcal{D}$ l'image schématique du morphisme
$$
\X^{(g-1)/2}\rightarrow \mathcal{J}^{3(g-1)/2}
$$
donné par $(x_1,\cdots,x_{(g-1)/2})\mapsto
\Omega^{1}_{\X/S}(-x_1-\cdots-x_{(g-1)/2})$. Soit $\mathcal{D}'$
l'image réciproque de $\mathcal{D}$ par
$V:=F^{\ast}:\mathcal{J}^{(g-1)/2}\rightarrow
\mathcal{J}^{3(g-1)/2}$. On a
$(\mathcal{D}'_{s})_{\mathrm{red}}=\Theta'$. Par ailleurs, comme
$\X_{\eta}$ est générique,
$\mathrm{NS}(\mathcal{J}_{1,\eta})\simeq \mathbf{Z}$, ce qui
implique que $\mathcal{D}'_{\eta}$ est algébriquement équivalent à
$r'\Theta_{\mathrm{class}}$ avec $r'\geq 1$ un entier. Par
spécialisation, $\Theta'=(\mathcal{D}_s')_{\mathrm{red}}$ est donc
algébriquement équivalent à $r\Theta_{\mathrm{class}}$ avec $r\leq
r'$ un entier $\geq 1$. D'où le résultat.
\end{proof}
\vspace{2mm}

Revenons au diviseur $\Theta'$. Alors $\Theta'$ est algébriquement
équivalent à $r\Theta_{\mathrm{class}}$ avec $r$ un entier
strictement positif. Comme $\Theta$ est algébriquement à
$2\Theta_{\mathrm{class}}$, on sait que $1\leq r\leq 2$
(proposition \ref{polarisation}). Donc ou bien $r=1$, et ceci
implique que $\Theta$ est une somme de deux composantes translatés
de $\Theta_{\mathrm{class}}$; ou bien $r=2$, et alors
$\Theta=\Theta'$ est irréductible (proposition
\ref{polarisation}). En tous cas, $\Theta$ ne contient pas de
composante translatée d'une sous-variété abélienne de $J_1$.

En résumé, on a montré le théorème suivant:

\begin{theorem}\label{cas p=3} Soient $k$ un corps algébriquement clos de caractéristique
$3$, $X/k$ une courbe propre lisse connexe de genre $g\geq 2$.
Alors: (1) $\Theta$ est réduit; (2) si $g\geq 3$, aucune
composante de $\Theta$ n'est un translaté d'une sous-variété
abélienne de $J_1$.
\end{theorem}

\begin{remark}(Complément au cas $g$ impair) Dans le cas $g$ impair,
il paraît probable que le cas où $\Theta$ est une somme de deux
composantes translatées de $\Theta_{\mathrm{class}}$ ne puisse pas
se produire. En effet, sur les complexes, il est connu ($\S$ 3 de
\cite{Beauville}, voir aussi \cite{GeemenIzadi}) que, pour une
courbe $X$ lisse sur $k$, l'application $\mathcal{M}\rightarrow
|2\Theta_{\mathrm{class}}|$ qui à un fibré vectoriel semi-stable
$V$ de rang $2$ de pente $g-1$ sur $X$, associe son diviseur
$\Theta_V$, est une immersion de l'espace de modules des fibrés
semi-stables de rang $2$ de pente $g-1$ vers le système linéaire
de $2\Theta_{\mathrm{class}}$ si $X$ n'est pas hyperelliptique, et
induit une immersion $\mathcal{M}/\{1,\tau\}\rightarrow
|2\Theta_{\mathrm{class}}|$ si $X$ est hyperelliptique (où $\tau$
est l'involution hyperelliptique de $X$). Si un tel résultat est
encore vrai en caractéristique $p=3$, $\Theta_{B}$ ne peut pas
être somme de deux translatés de $\Theta_{\mathrm{class}}$, en
raison de la stabilité de $B$.
\end{remark}

\subsubsection*{Un exemple en genre $4$ où $\Theta$ n'est pas normal.}
\label{exemple ou theta n'est pas normal} Reprenons la courbe de
Tango $X$ construite dans l'exemple \ref{courbe de tango} en
prenant $d=4$. Gardons les notations de \ref{courbe de tango}.
Alors $X$ est une courbe de genre $4$, et
$M=\OO_{X}(\sum_{i=1}^{10}b_i-\pi^{-1}(\sum_{j=1}^{4}c_j))$ est un
sous-faisceau inversible de $B$ (où $\pi:X\rightarrow
\mathbf{P}_{k}^{1}$ est le revêtement fini étale de degré $2$,
$\{b_{i}\}_{i=1,\cdots,10}$ sont les points de $X$ de ramification
de $\pi$, $\{c_{j}\}_{j=1,2,3,4}\subset
\mathbf{P}^{1}_{k}=\mathbf{P}$ tels que $c_j\neq a_i:=\pi(b_i)$
pour tout $i$). Alors $\Sigma:=\{M(-D)|D\in X^{(2)}\}$ est un
fermé de dimension $2$ de $\Theta$. Montrons que $\Theta$ est
singulier en tous les points de $\Sigma$. En effet, comme
$M^{2}=\OO_{X}\left(2\left(\sum_{i=1}^{10}b_i-\pi^{-1}(\sum_{j=1}^{4}c_j)\right)\right)
\simeq \pi^{\ast}\left(\OO_{\mathbf{P}}\left(\sum_{i=1}^{10}a_i-2
(\sum_{j=1}^{4}c_j)\right)\right)\simeq
\pi^{\ast}(\OO_{\mathbf{P}_{k}^{1}}(2))$, on a $h^{0}(X,M^{2})\geq
3$. Donc $h^{0}(X,M^{2}(-D))\geq 1$ pour tout diviseur $D$
effectif de degré $2$. Par suite, pour tout $D$ diviseur effectif
de degré $2$ sur $X$, il existe un $D'$ diviseur effectif de degré
$2$ tel que $M^{2}=\OO_{X}(D+D')$. Donc $L:=M(-D)\hookrightarrow
N\hookrightarrow B$ est tel que $L^{-1}\simeq
M(-D')\hookrightarrow M$, c'est-à-dire, ces deux plongements sont
orthogonaux par rapport au produit $(\cdot,\cdot)$ sur $B$. Donc,
d'après le critère de lissité (proposition \ref{critere de
lissite}), $\Theta$ est singulier en $L$. Donc $\Theta$ contient
un fermé de points singuliers de codimension $1$, ce qui implique
que $\Theta$ n'est pas normal.

\subsection{Le diviseur thêta en genre $2$ et en toute
caractéristique (positive)}

Soient $k$ un corps algébriquement clos de caractéristique $p>0$,
$X$ une courbe propre lisse connexe de genre $g=2$ sur $k$. Notons
$\Theta$ le diviseur thêta associé au faisceau des formes
différentielles localement exactes sur $X_1$. Le but de cette
section est de montrer qu'au moins dans le cas ordinaire, aucune
composante de $\Theta$ n'est un translaté d'une sous-variété
abélienne de $J_1$.

\subsubsection{Analyse de la lissité aux points d'ordre $p$.}

Sur $X_1$, on dispose de la suite exacte suivante:
$$
\xymatrix{0\ar[r] &  B\ar[r] & F_{\ast}\Omega^{1}_{X/k}\ar[r]^{c}
& \Omega^{1}_{X_1/k}\ar[r] & 0}.
$$
Pour $L$ un faisceau inversible de degré $0$, on a donc
$h^{0}(X_1,B\otimes L )\leq
h^{0}(X_1,F_{\ast}(\Omega^{1}_{X/k})\otimes
L)=h^{0}(X,\Omega^{1}_{X/k}\otimes F^{\ast}L)=1$ sauf si
$F^{\ast}L=L^{p}=\OO_{X}$, où l'on pourrait avoir
$h^{0}(X_1,B\otimes L)=2$. On a donc le lemme suivant:

\begin{lemma} Avec les notations ci-dessus, soit $x\in \Theta$ un point
d'ordre $p$ avec $L$ le faisceau inversible de degré $0$
correspondant. Alors $\Theta$ est singulier en $x$ si et seulement
si $h^{0}(X_1,B\otimes L)=2$. Cette condition équivaut au fait que
le morphisme $c\otimes 1_L:H^{0}(X, \Omega^{1}_{X/k}\otimes
F^{\ast}L)\rightarrow H^{0}(X_1,\Omega^{1}_{X_1/k}\otimes L)$ est
nul.
\end{lemma}

\begin{remark} Soit $x=[L]\in \Theta$ un point d'ordre $p$, et
notons $\omega_{x}$ la forme de Cartier associée. On voit $L$
comme un sous-faisceau du faisceau des fonctions méromorphes de
$X_1$. Soit $(U_{\alpha},f_{\alpha})$ une section méromorphe de
$L$, on a donc $f_{\alpha}/f_{\beta}\in \OO_{X_1}(U_{\alpha}\cap
U_{\beta})^{\ast}$. Comme $L$ est d'ordre $p$, il existe
$u_{\alpha}\in \OO_{X_1}(U_{\alpha})^{\ast}$ pour chaque $\alpha$
tel que $f_{\alpha}^{p}/f_{\beta}^{p}=u_{\alpha}/u_{\beta}$. La
forme $\omega_x$ est donc la forme différentielle holomorphe
$\left(U_{\alpha},d(u_{\alpha})/u_{\alpha}\right)$. Soit $\omega$
une autre section globale de $\Omega^{1}_{X_1/k}$ tel que $\omega$
et $\omega_x$ forment une $k$-base de $H^{0}(X,\Omega^{1}_{X/k})$.
Ceci nous fournit une section $(U_{\alpha},\omega\otimes
f^{p}_{\alpha}u_{\alpha})$ de $\Omega^{1}_{X/k}\otimes L^{p}$. Par
l'opérateur de Cartier, elle donne la section
$(U_{\alpha},c(u_{\alpha}\omega)\otimes f_{\alpha})$ de
$\Omega^{1}_{X_1/k}\otimes_{\OO_{X_1}}L$. Par le lemme, on voit
que $\Theta$ est singulier en $x$ si et seulement si cette forme
est nulle. Ceci équivaut à dire que $u_{\alpha}\omega$ est
localement exacte.
\end{remark}

\begin{theorem}\label{lissite de points d'ordre p,cas ordinaire}
Supposons que $X$ est ordinaire de genre $2$ en caractéristique
$p>0$. Soit $H$ une droite affine dans le $\mathbf{F}_p$-espace
$J_1[p](k)\simeq \mathbf{F}_{p}^{2}$ qui ne passe pas par
l'origine de $J_{1}[p](k)$, alors $\Theta$ est lisse en au moins
deux points de $H$.
\end{theorem}

\begin{proof} Soient $x=[L],~x'=[L']\in \Theta$ deux points d'ordre
$p$ qui ne sont pas $\mathbf{F}_p$-colinéaires dans $J_{1}[p](k)$,
et notons
$\omega_{x}=\left(U_{\alpha},d(u_{\alpha})/u_{\alpha}\right)$ et
$\omega_{x'}=\left(U_{\alpha},d(v_{\alpha})/v_{\alpha}\right)$ les
formes de Cartier associées (où $\{U_{\alpha}\}$ est un
recouvrement ouvert de $X_1$, $u_{\alpha}, ~v_{\alpha}\in
\OO_{X_1}(U_{\alpha})^{\ast}$ tels que $(U_{\alpha},u_{\alpha}) $
(resp. $(U_{\alpha},v_{\alpha})$) est un $1$-cobord de
$\OO_{X_1}^{\ast}$ correspondant à $L^{p}$ (resp. à $L'^{p}$)).
Pour le faisceau inversible $L\otimes L'^{\otimes^{i}}$
($i=0,\cdots ,p-1$), sa forme de Cartier s'écrit
$\omega_{x}+i\omega_{x'}=(U_{\alpha},\frac{du_{\alpha}}{u_{\alpha}}
+i\frac{dv_{\alpha}}{v_{\alpha}})$. Comme $x$ et $x'$ ne sont pas
colinéaires, on peut prendre $\{\omega_{x+ix'},\omega_{x}\}$ comme
$k$-base des sections globales de $\Omega^{1}_{X/k}$ pour $i\neq
0$. D'après ce qui précède, $\Theta$ est singulier en $x+ix'$ si
et seulement si la forme $(U_{\alpha},
u_{\alpha}v_{\alpha}^{i}du_{\alpha}/u_{\alpha})$ est localement
exacte.

Comme $\omega_{x'}\neq 0$ dans $H^{0}(X,\Omega^{1}_{X/k})$, on
sait que $v_{\alpha}$ n'est pas une puissance $p$-ième d'un
élément de $\OO_{X}(U_{\alpha})$. Quitte à diminuer $U_{\alpha}$,
on peut supposer que $u_{\alpha}$ s'écrit sous la forme
$u_{\alpha}=\sum_{j=0}^{p-1}\lambda_{j}^{p}v_{\alpha}^{j}$ avec
$\lambda_j\in \OO_{X}(U_{\alpha})$ . Donc
$du_{\alpha}=\sum_{j=1}^{p-1}j\lambda_{j}^{p}v_{\alpha}^{j-1}dv_{\alpha}$,
et $v_{\alpha}^{i}du_{\alpha}=\sum_{j=1}^{p-1}j\lambda_{j}^{p}
v_{\alpha}^{i+j-1}dv_{\alpha}$. C'est une forme localement exacte
si et seulement si $\lambda_{p-i}=0$. Donc, pour que $\Theta$ soit
singulier en $x+iy$ pour $1\leq i\leq p-1$, il faut
$\lambda_{p-i}=0$, ceci entraîne que $u_{\alpha}$ est une
puissance $p$-ième, d'où une contradiction.

Ceci étant, il existe donc au moins un $i\in \{1,\cdots,p-1\}$ tel
que $\Theta$ soit lisse en $x+iy$. Supposons qu'il y en a un seul
$i\in \{1,\cdots,p-1\}$ tel que $\Theta$ soit lisse en $x+iy$, et
montrons que $\Theta$ est lisse en $x$. Quitte à remplacer $L'$
par une puissance convenable, on peut supposer que $i=1$. Donc
$u_{\alpha}=\lambda_0^{p}+\lambda_{p-1}^{p}v_{\alpha}^{p-1}$. Si
$\Theta$ est singulier en $x$, ceci implique que
$v_{\alpha}du_{\alpha}/u_{\alpha}$ est localement exacte. Or
$v_{\alpha}du_{\alpha}/u_{\alpha}=(p-1)\lambda_{p-1}^{p}v_{\alpha}^{p-1}dv_{\alpha}/
(\lambda_{0}^{p}+\lambda_{p-1}^{p}v_{\alpha}^{p-1})$. Si c'est une
forme localement exacte, $\lambda_0=0$, ce qui implique que
$(p-1)dv_{\alpha}/v_{\alpha}=du_{\alpha}/u_{\alpha}$, d'où une
contradiction puisque $L$ et $M$ ne sont pas colinéaires.
\end{proof}

Complètons le théorème \ref{lissite de points d'ordre p,cas
ordinaire} en considérant les courbes de $p$-rang $1$.
\begin{theorem}\label{lissite de points d'ordre p,cas non-ordinaire}
Soit $X$ de genre $2$ et de $p$-rang $1$. Alors $\Theta$ est lisse
en au moins un point d'ordre $p$ (et donc en deux points d'ordre
$p$ si $p\geq 3$ par symétrie).
\end{theorem}

\begin{proof} Soit $x\in J_1$ un point d'ordre $p$. Soit
$\omega_x=(U_{\alpha},du_{\alpha}/u_{\alpha})$ la forme de Cartier
associée à $x$. Soit $\omega$ une autre forme différentielle
holomorphe sur $X$ telle que $\omega$ et $\omega_x$ forment une
$k$-base de $H^{0}(X_1,\Omega^{1}_{X_1/k})$. Comme
$\Omega^{1}_{X_1/k}$ est un $\OO_{X_1}$-module de rang $1$, il
existe une fonction méromorphe $h$ de $X$ telle que
$\omega=h\omega_x$. Donc il suffit de montrer qu'il existe un
$i\in \{1,\cdots,p-1\}$ tel que
$(U_{\alpha},hu_{\alpha}^{i-1}du_{\alpha})$ ne soit pas localement
exacte. Quitte à diminuer $U_{\alpha}$, on peut supposer que
$h=\sum_{j=0}^{p-1}\lambda_{j}^{p}u_{\alpha}^{j}$. Alors
$hu_{\alpha}^{i-1}du_{\alpha}=\sum_{j=0}^{p-1}\lambda_{j}^{p}u_{\alpha}^{i+j-1}du_{\alpha}$
est une forme localement exacte si et seulement si
$\lambda_{p-i}=0$. Donc si $\Theta$ est singulier en tout $ix$
quelque soit $i\in \{1,\cdots,p-1\}$, on obtient que
$\omega=\lambda_0^{p}\omega_{x}$. Or $\omega$ et $\omega_x$ sont
des formes holomorphes sur $X_1$ de genre $2$, les diviseurs de
$\omega$ et $\omega_x$ sont effectifs de degré $2$. Ce qui
implique que $\lambda_0\in k-\{0\}$: en effet, si $\lambda_0$
n'est pas un scalaire, $\lambda_0$ a un seul pôle d'ordre $1$,
mais comme $X$ est de genre $g=2$, un tel $\lambda_0$ n'existe
jamais. Donc $\lambda_0\in k-\{0\}$. Ceci contredit le fait que
$\omega$ et $\omega_x$ forment une $k$-base de
$H^{0}(X_1,\Omega^{1}_{X_1/k})$.
\end{proof}

\subsubsection{Composantes principales de $\Theta$ et amplitude.}

\begin{corollary}\label{principal} Supposons $X$ ordinaire de genre $2$. Alors
toute composante principale (définition \ref{composante
principale}) est ample.
\end{corollary}

\begin{proof} En effet, soit $D\subset \Theta$ une composante principale de
$\Theta$. Supposons qu'elle est de la forme $x+E$, avec $E\subset
J_1$ une courbe elliptique, $x$ un point d'ordre $p$. Comme $E$
est ordinaire, $D$ contient au moins deux points $x,y$ d'ordre
$p$. Comme $X$ est ordinaire, $\Theta$ ne passe pas par l'origine.
Ceci entraîne que $x+E[p](k)$ est une droite affine de
$J_1[p](k)\simeq \mathbf{F}_{p}^{2}$ qui ne passe pas par
l'origine. D'après le théorème \ref{lissite de points d'ordre
p,cas ordinaire}, on trouve que $\Theta$ est lisse en au moins
deux points d'ordre $p$ qui sont contenus dans $D$. Or $D$ est un
translaté d'une sous-variété abélienne de $J_1$, il existe un
$z\in E(k)$ tel que $x=y+z$, d'où une contradiction avec
\ref{critere de ne pas d'etre un TAS}.
\end{proof}

Complètons le corollaire \ref{principal} en considérant les
courbes non-ordinaires.

\begin{corollary} Soit $X$ une courbe de genre $2$ de
$p$-rang $\leq 1$. Soit $\Theta'$ la réunion des composantes
passant par l'origine. Alors $\Theta'$ est un diviseur ample de
$J_1$.
\end{corollary}

\begin{proof} Si $X$ est
de $p$-rang $0$, comme $\Theta'$ est la réunion des composantes
principales de $\Theta$, elle est ample (corollaire \ref{Dirac
implique ample}). Il reste à traiter le cas où $X$ est de $p$-rang
$1$. Si $\Theta'$ n'est pas ample, soit $E$ la plus grande
sous-variété abélienne qui laisse stable $\Theta'$. D'après la
propriété de Dirac, on sait que $E$ est une courbe elliptique
ordinaire (corollaire \ref{Dirac implique ordinaire}). En
particulier, on sait que $\Theta'$ contient tous les points
d'ordre $p$ de $J_1$. Or $\Theta$ est lisse en au moins un point
d'ordre $p$ (théorème \ref{lissite de points d'ordre p,cas
non-ordinaire}), disons $x\in J_1$. Soit $f$ une équation locale
de $\Theta$ en $x$. La propriété de Dirac nous dit que $f$ est
nulle dans $\OO_{\ker(V),x}$. Considérons l'automorphisme $T_{x}$
de $J_1$ défini par $y\mapsto y+x$. Il laisse stable $\Theta'$.
Donc $T_{x}^{\ast}(f)$ est une équation locale de $\Theta'$ en
$0$, par suite elle fait partie de l'équation locale de $\Theta$
en $0$. En particulier, si $g$ est une équation locale de $\Theta$
en $0$, elle est nulle dans $\OO_{\ker(V),0}$. Ceci nous fournit
une contradiction avec la propriété de Dirac (\ref{Dirac}).
\end{proof}

\subsubsection{Composantes non principales en genre $2$.} Dans ce
numéro, on étudie les composantes non principales. En particulier,
on montre que lorsque la courbe $X$ est ordinaire (de genre $2$),
$\Theta$ ne contient pas de composante irréductible qui est un
translaté d'une sous-variété abélienne de $J_1$.


Commençons par un résultat facile, qui découle directement de la
classification des fibrés vectoriels sur une courbe elliptique.

\begin{lemma} Soit $E$ une courbe elliptique sur $k$.

(1) Soit $V$ un fibré vectoriel sur $E$. Supposons que
$h^{0}(E,V\otimes L)=1$, et $h^{1}(E, V\otimes L)=0$ quelque soit
$L\in \mathrm{Pic}^{\circ}_{E/k}(k)$. Alors $V$ est stable.

(2) Soit $W$ un fibré vectoriel de $E$ tel que $h^{0}(E,W\otimes
L)=h^{1}(E, W\otimes L)=1$ quelque soit $L\in
\mathrm{Pic}^{\circ}_{E/k}(k)$. Alors $W$ n'est pas stable, et sa
filtration de Harder-Narasimhan s'écrit $0=W_0\subset W_1\subset
W_2=W$ avec $W_1$ un sous-fibré stable de degré $1$, et $W/W_1$ un
fibré stable de degré $-1$.
\end{lemma}

\begin{proof} Soit
$$
0=V_0\subset V_1\subset \cdots\subset V_{i-1}\subset V_i\subset
\cdots\subset V_r=V
$$
la filtration de Harder-Narasimhan de $V$, telle que
$V_{i}/V_{i-1}$, pour $1\leq i\leq r$ soit un fibré stable
non-trivial de pente $\lambda_i$, et que l'on ait $\lambda_1\geq
\lambda_2\geq \cdots\geq \lambda_{r-1}\geq \lambda_{r}$. On
raisonne par l'absurde. Supposons que $r\geq 2$. Montrons d'abord
que $\lambda_1>0$. Sinon, $\lambda_i\leq 0$ pour $1\leq i\leq r$.
Les $V_i/V_{i-1}$ sont stables donc indécomposables de $E$.
D'après Atiyah et Oda (\cite{Oda} section 2, page 60), il existe
exactement un $L_i\in J_{E}$ tel que
$h^{0}(E,(V_{i}/V_{i-1})\otimes L_i)=h^{1}(E,
(V_{i}/V_{i-1})\otimes L_i)>0$. En particulier, ceci entraîne
qu'il n'y a qu'un nombre fini de $L\in J_E$ tels que
$h^{0}(E,V\otimes L)\neq 0$, ce qui contredit l'hypothèse que
$h^{0}(E, V\otimes L)=1$ quelque soit $L\in J_E$. Donc
$\lambda_1>0$. Ensuite, montrons que $\lambda_2<0$. En fait, on a
la suite exacte suivante:
$$
0\rightarrow V_1\rightarrow V_2\rightarrow V_2/V_1\rightarrow 0.
$$
Comme $\lambda_2\geq 0$, il existe au moins un faisceau inversible
$L$ tel que $h^{0}(E, (V_{2}/V_1)\otimes L)>0$. La suite exacte
longue associée à la cohomologie nous montre que
$h^{0}(E,V_2\otimes L)\geq 2$ (rappelons que $\lambda_1>0$, donc
$h^{1}(E,V_1\otimes L)=0$). A priori, $h^{0}(E,V\otimes L)\geq 2$,
d'où une contradiction. Donc $\lambda_{r}<0$ puisque $r\geq 2$.
Considérons la suite exacte suivante:
$$
0\rightarrow V_{r-1}\rightarrow V\rightarrow V/V_{r-1}\rightarrow
0,
$$
on trouve un morphisme surjectif: $H^{1}(E, V)\rightarrow H^{1}(E,
V/V_{r-1})$ avec $H^{1}(E,V/V_{r-1})\neq 0$ (car $V/V_{r-1}$ est
stable de pente $\lambda_r<0$). Donc $h^{1}(E, V)\geq 1$, d'où une
contradiction. Donc $r=1$, c'est-à-dire, $V$ est stable. D'où (1).
Le même raisonnement nous montre la partie (2) du lemme. Ceci
achève la démonstration.
\end{proof}

\begin{prop} Soit $D$ une composante irréductible (réduite) de
$\Theta$, et supposons que $D$ ne passe par aucun point d'ordre
divisant $p$. Alors $D$ n'est pas une courbe de genre $1$.
\end{prop}

\begin{proof} On raisonne par l'absurde. Supposons que $D$ soit
une courbe de genre $1$. Comme $D$ ne passe pas par aucun point
d'ordre divisant $p$, il existe un morphisme fini de courbes
lisses $f:X_1\rightarrow E$ avec $E$ une courbe elliptique, et un
faisceau inversible $M$ de degré $0$ sur $X_1$, tel que
$M\otimes_{\OO_{X_1}} f^{\ast}L\in D$ quelque soit $L\in
J_E=\mathrm{Pic}^{\circ}_{E/k}(k)$. Comme $D$ ne passe pas par
aucun point d'ordre divisant $p$, on a
$h^{0}(X_1,F_{\ast}\Omega^{1}_{X/k}\otimes M\otimes
f^{\ast}L)=h^{0}(X_1,\Omega^{1}_{X/k}\otimes F^{\ast}(M\otimes
f^{\ast}L))=1$. Donc $h^{0}(X_1,B\otimes M\otimes f^{\ast}L)=1$
quel que soit $L\in J_{E}$. Posons $\mathcal{E}=f_{\ast}(B\otimes
M)$, $\mathcal{F}=f_{\ast}(F_{\ast}\Omega^{1}_{X/k}\otimes M)$ et
$\mathcal{G}=f_{\ast}(\Omega^{1}_{X/k}\otimes M)$. Comme
$f:X_1\rightarrow E$ est fini, on a la suite exacte suivante:
$$
0\rightarrow \mathcal{E}\rightarrow \mathcal{F}\rightarrow
\mathcal{G}\rightarrow 0,
$$
et $\mathcal{E}$, $\mathcal{F}$ et $\mathcal{G}$ sont des fibrés
sur $E$ avec les propriétés suivantes: $h^{0}(E,\mathcal{E}\otimes
L)=h^{0}(E,\mathcal{F}\otimes L)=h^{1}(E,\mathcal{E}\otimes L)=1$
quelque soit $L\in J_E$, et $h^{1}(E,\mathcal{F}\otimes L)=0$ pour
tout $L\in J_E$. Donc, d'après le lemme précédent, on sait que
$\mathcal{F}$ est un fibré stable sur $E$, et $\mathcal{E}$ admet
une filtration de Harder-Narasimhan de la forme $0\subset
\mathcal{E}_1\subset \mathcal{E}_2=\mathcal{E}$ telle que
$\mathcal{E}_1$ soit stable de pente $>0$. En particulier,
$\mathcal{F}$ admet un sous-fibré de pente $>0$. Mais
$\mathcal{F}$ est un fibré de degré $1=h^{0}(E,\mathcal{F})$, donc
la pente de $\mathcal{E}_1$ est strictement plus grande que celle
de $\mathcal{F}$, ceci contredit la stabilité de $\mathcal{F}$. Le
résultat s'en déduit.
\end{proof}

Comme corollaire direct, on a
\begin{corollary}\label{theta est ample si g=2} Supposons $X$ ordinaire de genre $2$. Alors toute
composante irréductible de $\Theta$ est ample.
\end{corollary}

\subsection{Cas où $g=2$ et $p=3$}

\subsubsection{Un résultat classique.} Le résultat suivant est
classique, qui découle sans difficulté du théorème d'indice de
Hodge.

\begin{prop}\label{resultat classique} Soit $k$ un corps algébriquement clos de caractéristique
quelconque, $X/k$ une courbe propre lisse connexe de genre $2$.
Soit $D$ un diviseur effectif sur $J$, qui est algébriquement
équivalent à $2\Theta_{\mathrm{class}}$. Alors:

(1) Ou bien $D$ est irréductible,

(2) Ou bien $D=X_1+X_2$ où $X_i$ est un translaté de
$\Theta_{\mathrm{class}}$.

(3) Ou bien $X$ est un revêtement de degré $2$ d'une courbe
elliptique, et $D=(x_1+E_1)+(x_2+E_2)$ où $E_i$ est une courbe
elliptique $x_i\in J$. De plus, $E_1\cdot E_2=4$.
\end{prop}

\subsubsection{Irréductibilité de $\Theta$ pour $g=2$ et $p=3$.}

\begin{theorem} \label{cas p=3 g=2}Soit $X$ une courbe lisse propre connexe, de genre $g=2$
sur un corps $k$ algébriquement clos de caractéristique $p=3$.
Soit $\Theta$ son diviseur thêta associé au faisceau $B$. Alors le
faisceau $\QQ$ ($\S$ \ref{accouplement}) est inversible de degré
$4$ sur $\Theta$. De plus $\Theta$ est intègre, et ses seuls
points singuliers sont ses points de torsion d'ordre divisant $2$.
\end{theorem}

Voici un lemme qui est extrait de \cite{BGN} (proposition 3.1 de
\cite{BGN}).

\begin{lemma} Soit $E$ un fibré stable de rang $r\geq 2$ et de pente
$\lambda\leq 1$. Soit $F$ le sous-faisceau engendré par les
sections globales $H^0(X,E)$ de $E$. Alors $F$ est un sous-fibré
de $E$ qui est isomorphe à $\mathcal{O}_{X}^n$ avec $n=h^0(X,E)$.
\end{lemma}

\begin{corollary} Soit $E$ est un fibré stable de rang $r\geq 2$
et de pente $\leq 1$, alors $h^0(X,E)<r$.
\end{corollary}

\begin{proof}[Démonstration du théorème \ref{cas p=3 g=2}]
Comme $p=3$ et $g=2$, le fibré vectoriel stable $B$ est de pente
$1$ et de rang $2$. D'après le lemme précédent, pour tout $L$
faisceau inversible de degré sur $X_1$ qui correspond à un point
$x\in \Theta$, on a $h^{0}(X_1,B\otimes L)=1$. En particulier, le
faisceau $\QQ$ est inversible sur $\Theta$ (lemme \ref{lemme
trivial}). Soit $d$ son degré. Comme $(-1)^{\ast}\QQ\otimes
\QQ\simeq \omega_{\Theta}\simeq \OO_{J_1}(\Theta)|_{\Theta}$ ($\S$
\ref{accouplement}), on a $d=\mathrm{deg}(\omega_{\Theta})/2=4$.
Comme $B$ est stable de pente $1$, soit $L^{-1}\hookrightarrow B$
l'unique plongement (à multiplication par un scalaire près) de
$L^{-1}$ dans $B$, alors $L^{-1}$ est égal à sa saturation. Donc
si $x\in \Theta$ est un point singulier, d'après le critère de
lissité \ref{critere de lissite}, il faut et il suffit que
$\mathrm{Hom}(L, L^{-1})\neq 0$. Ceci équivaut au fait que
$L^{2}=\OO_{X_1}$.

Il reste à montrer que $\Theta$ est intègre. Comme $p=3$, $\Theta$
est un diviseur effectif réduit algébriquement équivalent à
$2\Theta_{\mathrm{class}}$. Donc on a l'une des trois possibilités
pour $\Theta$ décrites dans la proposition \ref{resultat
classique}. D'après \cite{R2}, $\Theta$ contient au moins une
composante irréductible qui n'est pas un translaté d'une
sous-variété abélienne, ceci exclut le cas (3). Si
$\Theta=\Theta_1+\Theta_2$ où
$\Theta_i=x_i+\Theta_{\mathrm{class}}$ est un translaté de
$\Theta_{\mathrm{class}}$ avec $x_i\in J_1(k)$ un point fermé.
Comme $\Theta$ est symétrique et réduit, $x_1=-x_2$ et $x_i$ n'est
pas un point d'ordre $2$. Montrons que ceci entraîne que $B$ n'est
pas stable. En fait, notons $S_{X_1}$ l'espace de modules des
fibrés semi-stables de rang $2$ et de déterminant
$\Omega^{1}_{X_1/k}$. Rappelons qu'on a l'identification
$\phi:S_{X_1}\rightarrow |2\Theta_{\mathrm{class}}|$ qui envoie
$E\in S_{X_1}$ sur le diviseur thêta associé à $E$ (\cite{NR}). De
plus, le diagramme suivant est commutatif:
$$
\xymatrix{J_1\ar[r]^{b}\ar[rd]_{K_{X_1}} & S_{X_{1}}\ar[d]^{\phi} \\
& |2\Theta_{\mathrm{class}}|}
$$
dans lequel $K_{X_1}:J_1\rightarrow |2\Theta_{\mathrm{class}}|$
est défini par $a\mapsto
(a+\Theta_{\mathrm{class}})+(-a+\Theta_{\mathrm{class}})$, et
$b:J_1\rightarrow S_{X_1}$ est la flèche naturelle associée au
faisceau $(\mathcal{P}\otimes \delta )\oplus
(\mathcal{P}^{-1}\otimes \delta)$ avec $\mathcal{P}$ un faisceau
de Poincaré de $X_1$ et $\delta$ une caractéristique thêta de
$X_1$. Alors $K_{X_{1}}$ induit une immersion fermée $J_1/\{\pm
1\}\rightarrow |2\Theta_{\mathrm{class}}|$. Comme
$\Theta=(x+\Theta_{\mathrm{class}})+(-x+\Theta_{\mathrm{class}})=K_{X_1}(x)=\phi(b(x))$,
$B=b(x)$ est donc un fibré non-stable, ce qui nous donne une
contradiction. Ceci exclut le cas (2). Donc d'après le théorème
précédent, $\Theta$ est irréductible, donc intègre.
\end{proof}

\subsubsection{Lien avec les variétés de Prym.} \label{Prym}



On suppose $p=3$. On considère une courbe lisse $X/k$ de genre
$2$, tel que $\Theta$ soit lisse sur $k$, ce qui est réalisé si
$X$ est assez générale. Par suite, $\Theta/k$ est une courbe lisse
de genre $5$, et $\QQ$ est inversible sur $\Theta$ (théorème
\ref{cas p=3 g=2}). Notons $J_1$ la jacobienne de $X_1$. Alors
$\Theta$ se descend en un diviseur de Cartier effectif $D$ sur la
variété de Kummer associée à $J_1$. Comme $\Theta$ est lisse, il
ne passe pas par un point d'ordre divisant $2$ de $J_1$. Donc $D$
est une courbe lisse connexe de genre $3$ de la surface de Kummer,
et $\pi:\Theta\rightarrow D$ est un revêtement étale de degré $2$.
Soient $\mathrm{Nm}:J_{\Theta}\rightarrow J_{D}$ le morphisme de
norme, et $P:=\left(\ker(\mathrm{Nm})\right)^{\circ}$ la variété
de Prym associée, qui est une variété abélienne de dimension $2$
sur $k$. Alors $\ker(\mathrm{Nm})/P\simeq \mathbf{Z}/2\mathbf{Z}$
(\cite{Prym}).

Comme le morphisme naturel $J_{1}^{\vee}\rightarrow J_{\Theta}$
est une immersion fermée,\footnote{Cet énoncé est du type
{\textquotedblleft Lefschetz\textquotedblright} (SGA2
\cite{SGA2}). Il n'est pas automatique pour toute surface en
caractéristique $p>0$, par suite du défaut de validité du
{\textquotedblleft vanishing de Kodaira\textquotedblright}, mais
ne pose pas de problème pour les surfaces abéliennes.} et que
$J_{1}^{\vee}\subset \ker(\mathrm{Nm})$, on a $P=J_{1}^{\vee}$.
Par ailleurs, d'après $\S$ \ref{accouplement}, on a un
isomorphisme $(-1)^{\ast}\QQ\otimes \QQ\simeq \omega_{\Theta}$
avec $\omega_{\Theta}=\OO_{J_1}(\Theta)|_{\Theta}$ le faisceau
canonique de $\Theta$. Soit $\theta$ un diviseur thêta classique
\textit{symétrique} de $J_1$, et posons $\QQ'=\QQ\otimes
(\OO_{J_1}(-\theta)|_{\Theta})$, alors $(-1)^{\ast}\QQ'\otimes
\QQ'\simeq \OO_{\Theta}$ (proposition \ref{theta est sym}). Donc
$\QQ'\in \ker(\mathrm{Nm})$.

\begin{proposition} Avec les hypothèses ci-dessus,
le faisceau $\QQ'$ est dans la composante connexe de
$\ker(\mathrm{Nm})$ qui n'est pas la composante neutre.
\end{proposition}

\begin{proof} Soit $S$ le spectre d'un anneau de valuation discrète complet,
strictement hensélien, de point générique $\eta$ et de point fermé
$s$. Soit $\X/S$ une courbe propre lisse sur $S$, dont la fibre
spéciale est $X$ (à un changement de base sur $k$ près), et dont
la fibre générique est la courbe générique de genre $2$. Soient
$\mathcal{B}$ le faisceau des formes différentielles localement
exactes de $\X$, $\Theta_{\mathcal{B}}$ le diviseur thêta
(relatif) associé à $\mathcal{B}$. Notons comme avant
$\mathcal{Q}_{\mathcal{B}}$ le faisceau sur $\Theta_{\mathcal{B}}$
à partir de $\mathcal{B}$ après un choix d'une section de $\X/S$
($\S$ \ref{accouplement}). Soit $\mathcal{J}_1$ la jacobienne de
$\X/S$. Comme $\Theta_{\mathcal{B}}$ est totalement symétrique
(\ref{theta est fortement symetrique}), il se descend en un
diviseur $\mathcal{D}$ sur la surface de Kummer $S$ associée à
$\mathcal{J}_1$. Par l'hypothèse, le schéma $\Theta_{\mathcal{B}}$
est lisse sur $S$, donc $\Theta_{\mathcal{B}}$ ne passe pas par un
point d'ordre divisant $2$. Ceci entraîne que $\mathcal{D}$ est
lisse sur $S$, et que le morphisme
$\Theta_{\mathcal{B}}\rightarrow \mathcal{D}$ est fini étale de
degré $2$. Comme $g=2$ et $p=3$, $\Theta_{\mathcal{B}}/S$ (resp.
$\mathcal{D}/S$) est une courbe relative de genre $5$ (resp. de
genre $3$). Soit $J_{\Theta_{\mathcal{B}}}$ (resp.
$J_{\mathcal{D}}$) la jacobienne de $\Theta_{\mathcal{B}}$ (resp.
de $\mathcal{D}$). On a un morphisme naturel
$\mathrm{Nm}:J_{\Theta_{\mathcal{B}}}\rightarrow J_{\mathcal{D}}$.
Soit $P:=(\ker(\mathrm{Nm}))^{\circ}$ la variété de Prym associée
à $\Theta_{\mathcal{B}}\rightarrow \mathcal{D}$. Alors $P$ est un
schéma abélien sur $S$ de dimension relative $2$, et
$\Pi:=\ker(\mathrm{Nm})/P\simeq \mathbf{Z}/2\mathbf{Z}$ d'après la
théorie des variétés de Prym.

Par ailleurs, le morphisme naturel
$\mathcal{J}_{1}^{\vee}\rightarrow J_{\Theta_{\mathcal{B}}}$ est
une immersion fermée, par suite $\mathcal{J}_{1}^{\vee}=P$. Notons
$\Theta_{\mathrm{class,sym}}$ un diviseur thêta classique
symétrique de $\mathcal{J}_1$, on a
$\OO_{\mathcal{J}_1}(\Theta)=\OO_{\mathcal{J}_1}(2\Theta_{\mathrm{class,sym}})$.
Donc
$(-1)^{\ast}\QQ_{\mathcal{B}}(-\Theta_{\mathrm{class,sym}})\otimes
\QQ_{\mathcal{B}}(-\Theta_{\mathrm{class,sym}})\simeq
\OO_{\Theta_{\mathcal{B}}}$ (\ref{accouplement}). Par conséquent,
$\QQ_{\mathcal{B}}':=\QQ_{\mathcal{B}}(-\Theta_{\mathrm{class,sym}})$
est un faisceau inversible de degré $0$ sur $S$, donc il
appartient à $J_{\Theta_{\mathcal{B}}}$. En effet, comme
$(-1)^{\ast}\QQ_{\mathcal{B}}'\otimes \QQ_{\mathcal{B}}'\simeq
\OO_{\Theta_{\mathcal{B}}}$, il appartient à $\ker(\mathrm{Nm})$.
Comme $\QQ_{\mathcal{B}\eta}$ donc $\QQ_{\mathcal{B}\eta}'$ ne
provient pas d'un faisceau inversible sur $\mathcal{J}_{1,\eta}$
(proposition \ref{Q ne provient pas de J, cas generique}), le
morphisme naturel $S\rightarrow \ker(Nm)\rightarrow \Pi$ induit
par $\QQ_{\mathcal{B}}'$ n'est pas la section nulle. Comme $S$ est
connexe, ce qui implique que $\QQ_{\mathcal{B},s}\notin
\mathcal{J}_{1,s}^{\vee}$. D'où le résultat.
\end{proof}

\begin{remarque} On sait comment décrire le faisceau $\QQ$ dans ce
cas. En fait, $\QQ$ est seulement défini à une tensorisation près
avec un faisceau inversible algébriquement trivial qui provient
d'un faisceau inversible de la jacobienne. Il suffit donc de
caractériser l'image de $\QQ(-\Theta_{\mathrm{class,sym}})$ dans
$J_{\Theta}/J_1$. Et d'après la preuve de la proposition
ci-dessus, on peut caratériser $\QQ$ comme le faisceau qui
appartient à la composante connexe de $\ker(Nm)/J_1\subset
J_{\Theta}/J_1$ autre que la composante neutre. Avec ce résultat,
et en appliquant la transformation de Fourier-Mukai inverse, on
peut retrouver le faisceau $B$ et la courbe $X$ à partir de la
variété abélienne $J_1$ et du diviseur $\Theta\subset J_1$.
\end{remarque}

%% file: chapitre5.tex
\section{Groupe fondamental et diviseur thêta}
\markboth{chapitre5}{chapitre5}

Soient $k$ un corps algébriquement clos de caractéristique $p>0$,
$X/k$ une courbe propre lisse connexe. Soit $\bar{x}$ un point
géométrique de $X$, on note $\pi_1(X,\bar{x})$ le groupe
fondamental de $X$ associé à $x$. Par définition (\cite{SGA1},
exposé V), $\pi_{1}(X,\bar{x})$ classifie les revêtements finis
étales de $X$. De plus, soit $\bar{x}'$ un autre point géométrique
de $X$, il existe un isomorphisme $\pi_1(X,\bar{x})\simeq
\pi_{1}(X,\bar{x}')$. Donc, s'il n'y a pas de besoin de préciser
le point géométrique, on note simplement $\pi_1(X)$ le groupe
fondamental de $X$.

Le groupe fondamental $\pi_1(X)$ est mal connu. Concernant sa
variation dans l'espace de modules des courbes (\cite{Pop-Saidi},
\cite{R2}, \cite{Tamagawa}), le résultat le plus frappant est le
suivant dû à A. Tamagawa.

\begin{theo}\textbf{\emph{(Tamagawa \cite{Tamagawa})}}  Soient $S=\mathrm{Spec}k[[T]]$ de point générique
$\eta$, et de point spécial $s$, $X/S$ une courbe relative propre
et lisse, de genre $g\geq 2$, dont la fibre spéciale $X_s$ est
définissable sur un corps fini. Alors si le morphisme de
spécialisation $\mathrm{sp}:\pi_1(X_{\bar{\eta}}) \rightarrow
\pi_1(X_{\bar{s}})$ est bijectif, la courbe relative $X/S$ est
constante.
\end{theo}

La démonstration de ce théorème utilise le diviseur $\Theta$, mais
celui-ci ne permet de contrôler qu'un certain quotient métabélien
du $\pi_1$, à savoir $\pi_{1}^{\mathrm{new,ord}}$ (on renvoye à
\cite{R2} pour la définition de $\pi_{1}^{\mathrm{new,ord}}$, et
la relation entre $\Theta$ et $\pi_{1}^{\mathrm{new,ord}}$). On
peut se demander si l'analogue du résultat de Tamagawa est encore
vrai quand on remplace le $\pi_1$ par son quotient
$\pi_{1}^{\mathrm{new, ord}}$. En fait, on est très loin de
pouvoir répondre à cette question faute de renseignements
suffisants sur la géométrie de $\Theta$ et sur la
{\textquotedblleft saturation\textquotedblright} de la torsion.
Toutefois, on donne une réponse positive lorsque la fibre spéciale
$X_s$ est supersingulière.

Plus précisément, soient $S$ le spectre d'un anneau de valuation
discrète, strictement hensélien, de point fermé $s$ et de point
générique $\eta$, $X/S$ une $S$-courbe propre lisse à fibres
géométriques connexes. On a, d'après Grothendieck, un morphisme
surjectif de spécialisation
$\mathrm{sp}:\pi_{1}(X_{\bar{\eta}})\rightarrow
\pi_1(X_{\bar{s}})$. Il induit un morphisme de spécialisation sur
$\pi^{\mathrm{new,ord}}_{1}$. La question suivante se pose
naturellement:

\begin{quest1}\label{question} Gardons les notations ci-dessus. On suppose que
$X_s$ est définissable sur un corps fini, et que le morphisme de
spécialisation:
$$
\mathrm{sp}:\pi^{\mathrm{new,ord}}_{1}(X_{\bar{\eta}})\rightarrow
\pi^{\mathrm{new,ord}}_{1}(X_{\bar{s}})
$$
est bijectif. Alors $X$ est-elle constante?
\end{quest1}

Commençons par traduire cette question au moyen de $\Theta$.

\begin{defn1} Soit $A$ une variété abélienne sur un corps $k$
algébriquement clos de caractéristique $p>0$. Soit $x\in A$ un
point de $A$ d'ordre $n$ premier à $o$, on note $\mathrm{Sat}(x)$
l'orbite de $x$ sous l'action naturelle du groupe
$(\mathbf{Z}/n\mathbf{Z})^{\ast}$ des unités de l'anneau
$\mathbf{Z}/n\mathbf{Z}$. De plus, soit $X$ un ensemble de points
de $p'$-torsion, on note $\mathrm{Sat}(X)=\cup_{x\in
X}\mathrm{Sat}(x)$.
\end{defn1}

Pour tout $n>0$ entier tel que $(n,p)=1$, les points de
$n$-torsion de $J_1$ sont décomposés sur $S$. Donc le morphisme de
spécialisation $\eta\rightarrow s$ induit une bijection sur les
points de $p'$-torsion $J_{1,\eta}\{p'\}\simeq J_{1,s}\{p'\}$
(pour une variété abélienne $A/k$ dont les points de $p'$-torsion
sont décomposés, on note $A\{p'\}$ l'ensemble des points de
$p'$-torsion de $A$). Sous cette identification, on a
$\Theta_{\eta}\{p'\}:=\Theta_{\eta}\cap J_{1,\eta}\{p'\}\subset
\Theta_{s}\{p'\}:=\Theta_{s}\cap J_{1,s}\{s\}$. Par suite,
$\mathrm{Sat}(\Theta_{\eta}\{p'\})\subset
\mathrm{Sat}(\Theta_{s}\{p'\})$. L'observation suivante est
importante dans les applications de $\Theta$ à l'étude de la
variation de $\pi_1$.

\begin{proposition1}[\cite{R2}, proposition 2.2.4]\label{sp et theta} Soit $X$ une
courbe propre et lisse sur $S$, à fibres géométriques connexes
telle que les points de $p'$-torsion de la jacobienne $J_1$ de
$X_1$ soient constants. Alors pour que le morphisme de
spécialisation
$$
\mathrm{sp}:\pi^{\mathrm{new,ord}}_{1}(X_{\bar{\eta}})\rightarrow
\pi^{\mathrm{new,ord}}_{1}(X_{\bar{s}})
$$
soit un isomorphisme, il faut et il suffit que $\mathrm{sp}$
induise une bijection $\mathrm{Sat}(\Theta_{\eta}\{p'\})\simeq
\mathrm{Sat}(\Theta_{s}\{p'\})$.
\end{proposition1}

Donc l'étude de la variation de $\pi_{1}^{\mathrm{new,ord}}$ est
ramenée à l'étude de la variation de la saturation des points de
$p'$-torsions de $\Theta$.\newline

\noindent\textbf{Rappels et notations.} Soit $k$ un corps
algébriquement clos.

(1) Soit $k\subset K$ une extension de corps, et soit $A/K$ une
variété abélienne. Il existe une plus grande sous-variété
abélienne $B$ de $A$, appelée \textit{la $k$-partie fixe} ou
\textit{la $k$-trace}, qui est isogène à une variété abélienne
définie sur $k$. Le {\textquotedblleft
supplémentaire\textquotedblright} $C$ de $B$ dans $A$ est
\textit{la $k$-partie mobile}.

(2) On dit qu'une sous-variété fermée $X\subset A$ est \emph{de
$p'$-torsion} si $X\{p'\}:=X\cap A\{p'\}$ est dense dans $X$.

(3) On dit que $A$ est \emph{supersingulière}, si $A$ est isogène
à un produit de courbes elliptiques supersingulières. Si $X$ est
une courbe lisse, on dit que $X$ est \emph{supersingulière} si sa
jacobienne $J_X$ l'est.    \hfill

\begin{theoreme1}\label{application de theta}Soit $S$
le spectre d'un anneau de valuation discrète strictement hensélien
(de caractéristique $p>0$), de point fermé $s$ et de point
générique $\eta$. Soit $X/S$ une $S$-courbe propre lisse, à fibres
géométriques connexes de genre $g\geq 2$. Notons $J/S$ la
jacobienne de $X/S$, $\Theta\subset J_1$ le diviseur thêta.
Supposons que $X_s$ soit une courbe supersingulière définissable
sur un corps fini, et que le morphisme de spécialisation
$$
\mathrm{sp}:\pi^{\mathrm{new,ord}}_{1}(X_{\bar{\eta}})\rightarrow
\pi^{\mathrm{new,ord}}_{1}(X_{\bar{s}})
$$
soit un isomorphisme. Alors $X$ est constante.


\end{theoreme1}



\begin{lemme1}\label{lemme sur sing} Soient $k$ un corps algébriquement clos, $C/k$ une
courbe projective lisse connexe. Supposons qu'une composante
irréductible $\Theta_i$ de $\Theta_C$ soit de $p'$-torsion, et
soit translatée d'une hypersurface abélienne $H$. Notons $E$ la
courbe elliptique $J_{C_1}/H$, et soit $a\in J_{C_1}$ tel que
$\Theta_i=a+H$. Notons $a'$ l'image de $a$ dans $E$. Alors si $E$
est ordinaire, $a'\neq 0$.
\end{lemme1}

\begin{proof} Notons $V_H$ (resp. $V_J$) le Verschiebung de $H$ (resp. de
$J:=J_{C_1}$). Alors si $E:=J/H$ est ordinaire,
$\ker(V_H)=\ker(V_J)$. Supposons que $a'=0$, en d'autres termes,
$\Theta_i=H$. En particulier, $\ker(V_J)=\ker(V_H)\subset
\Theta_i\subset \Theta_C$, on obtient donc une contradication avec
la propriété de Dirac (\ref{Dirac}), d'où le résultat.
\end{proof}

Le lemme ci-après est le point clé où l'on utilise la
supersingularité pour contrôler la saturation de la $p'$-torsion.

On dit qu'une composante irréductible d'un diviseur dans une
variété abélienne est \emph{abélienne} si elle est une translatée
d'une hypersurface abélienne.

\begin{lemme1}\label{lemme sur sat} Soient $k=\overline{\mathbf{F}_p}$
une clôture algébrique de $\mathbf{F}_p$, $C/k$ une courbe lisse
connexe projective \emph{supersingulière}. Soient $Y\subset
\Theta$ un sous-ensemble fermé de $\Theta$ tel que
$\mathrm{Sat}(\Theta\{p'\})=\mathrm{Sat}(Y\{p'\})$. Notons $D$ la
réunion (muni de la structure de sous-schéma fermé réduit) des
composantes de $\Theta$ autre que les composantes abéliennes non
principales (\ref{composante principale}) de $\Theta$. Alors
$D\neq \emptyset$, et $D$ est aussi ample.
\end{lemme1}

\begin{proof} On raisonne par l'absurde. Si $D$ n'est pas ample,
comme $J$ est supersingulière, il existe une courbe elliptique $E$
de $J$ qui laisse $D$ stable par translations. Notons
$\pi:A\rightarrow J/E=:B$ la projection naturelle ($J=J_{C_1}$ la
jacobienne de $C_1$). Et soit $\Delta$ le diviseur de $B$ tel que
$D'=\pi^{\ast}(\Delta)$. Notons $\Theta'$ la réunion des
composantes abéliennes de $\Theta$ qui ne passent pas par
l'origine. Alors $\mathrm{Sat}(\Theta')$ est encore une réunion
finie de translatées d'hypersurfaces abéliennes. Ecrivons
$$
\mathrm{Sat}(\Theta')=\mathrm{Sat}_1\cup \mathrm{Sat}_2,
$$
où $\mathrm{Sat}_1$ est la réunion des composantes de
$\mathrm{Sat}(\Theta')$ qui sont laissées stable par $E$. Alors
$\pi|_{\mathrm{Sat}_2}:\mathrm{Sat}_2\rightarrow B$ est un
morphisme fini surjectif, et il existe un diviseur $\Delta_1$ de
$B$ tel que $\mathrm{Sat}_1=\pi^{\ast}(\Delta_1)$. La réunion des
composantes principales de $\Theta$ est ample (\ref{Dirac implique
ordinaire}). Il existe donc une composante $\Theta_i$ de $\Theta$
telle que $\pi|_{\Theta_i}:\Theta_i\rightarrow B$ soit surjectif.
Puisque $\Theta_i$ est principal, elle n'est pas contenue dans
$\mathrm{Sat}(\Theta')$. On peut donc trouver un ouvert non-vide
$V\subset B$, tel que $V\cap
\left(\Delta\cup\Delta_1\right)=\emptyset$, et que pour tout $b\in
V$, on ait $\pi^{-1}(b)\cap \Theta_i\nsubseteq \pi^{-1}(b)\cap
\mathrm{Sat}_2$. Notons $Y'$ la réunion des composantes
irréductibles de $Y$ qui ne sont pas composantes de
$\mathrm{Sat}_2$. Et notons $G=\pi(Y')$. Alors $G$ est un fermé de
$B$ tel que $G\neq B$. Posons $G'=G\cup \Delta\cup\Delta_1$. C'est
un fermé propre de $B$. Il existe donc $b\in V(k)$ tel que
$b\notin \mathrm{Sat}(G')$ (lemme 4.3.5 de \cite{R1}). Soit $a\in
\Theta_i$ tel que $\pi(a)=b$. Comme
$\mathrm{Sat}(\Theta)=\mathrm{Sat}(Y)$, il existe $a'\in Y$ tel
que $a\in \mathrm{Sat}(a')$. Comme $\pi(a)\notin
\mathrm{Sat}(G')$, on a $\pi(a')\notin \mathrm{Sat}(G')$. Donc
$a'\in \mathrm{Sat}_2$, d'où $a\in \mathrm{Sat}_2$ (car
$\mathrm{Sat}_2$ est saturé). On en déduit que $\pi^{-1}(b)\cap
\Theta_i\subset \mathrm{Sat}_2$, ceci nous donne une contradiction
avec la définition de $V$, et finit donc la démonstration.

\end{proof}


Pour montrer le théorème \ref{application de theta}, rappelons
d'abord, d'après Hrushovski, la structure d'une sous-variété de
$p'$-torsion d'une variété abélienne.

\begin{defn1}[\cite{R2}] Soient $K$ une extension de
$\overline{\mathbf{F}_{p}}$, $A$ une variété abélienne sur $K$.

(1) On dit que $A$ est constante (resp. constante à isogénie près)
si $A$ est isomorphe (resp. isogène) à une variété abélienne
définissable sur $\overline{\mathbf{F}_{p}}$.

(2) Lorsque $A$ est constante à isogénie près, on dit qu'une
sous-variété fermé $Z\subset A$ est constante à isogénie près s'il
existe une isogénie
$u:A_{o}\times_{\overline{\mathbf{F}_{p}}}K\rightarrow A$ avec
$A_{o}$ une variété abélienne sur $\overline{\mathbf{F}_p}$, et si
$Z$ est l'image par $u$ d'une sous-variété réduite $Z_{o}$ de
$A_o\times_{\overline{\mathbf{F}_{p}}}K$ qui est définie sur
$\overline{\mathbf{F}_{p}}$.
\end{defn1}

\begin{theoreme1}[Hrushovski, \cite{Hrushovski}] \label{Hrushovski}
Soient $K$ un corps de caractéristique $p>0$ qui contient une
clôture algébrique $\overline{\mathbf{F}_p}$ de $\mathbf{F}_{p}$,
$A$ une $K$-variété abélienne, $X$ un $K$-sous-schéma fermé de
$A$, intègre, tel que l'ensemble des points de $X$ de $p'$-torsion
soit Zariski dense dans $X$. Notons $B$ le plus grand sous-schéma
abélien de $A$ qui laisse stable $X$ et soit $C$ la
($\overline{\mathbf{F}_p}$-)partie fixe de $A/B$. Alors $X$ est le
translaté par un point de $p'$-torsion d'une sous-variété de
$p'$-torsion à isogénie près de $C$.
\end{theoreme1}

\begin{remarque1} \emph{La preuve d'Hrushovski utilise la théorie des modèles.
R.Pink et D. Rössler donnent dans \cite{Pink-Roessler} une autre
preuve dans le language de la géométrie algébrique.}
\end{remarque1}


\begin{proof}[Démonstration du théorème \ref{application de theta}]
Notons $D$ contenu dans $J_s$ introduit dans le lemme \ref{lemme
sur sat}. Alors $D$ est ample. Soit $D_{\eta}'$ la réunion des
composantes $\Theta_{i,\eta}$ de $\Theta_{\eta}$ qui sont de
$p'$-torsion et dont l'adhérence schématique $\Theta_i$ a une
fibre spéciale ayant au moins une composante contenue dans $D$.
Donc $D_{\eta}'$ est ample (\ref{lemme sur sat}). Soit
$D_{i,\eta}'=a_{\eta}+H_{\eta}$ une composante abélienne de
$D_{\eta}'$, son adhérence schématique s'écrit
$\overline{D_{i,\eta}'}=a+H$, avec $H=\overline{H_{\eta}}$. Par
définition de $D$, $a_{s}+H_{s}$ est principal, donc $a_{s}\in
H_{s}$, d'où $a\in H$. Donc $E:=J/H$ est supersingulière
(\ref{lemme sur sing}). Notons $M_{\eta}$ la partie mobile de
$J_{\eta}$. On sait donc que $D_{i,\eta}'$ est trivial sur
$M_{\eta}$. Soit $D_{j,\eta}'$ une composante non abélienne de
$D_{\eta}'$, d'après Hrushovski (\ref{Hrushovski}), $D_{j,\eta}'$
est l'image réciproque d'un diviseur de la partie fixe de
$J_{\eta}$, est donc aussi trivial sur $M_{\eta}$. Finalement
$D_{\eta}'$ est à la fois ample et trivial sur $M_{\eta}$, donc
$M_{\eta}=0$. En d'autre terme, la partie fixe de $J_{\eta}$ est
égale à $J_{\eta}$.

Soit ensuite $\Theta_{\eta}' $ la réunion des composantes
irréductibles de $p'$-torsion. Par \ref{lemme sur sat}, le
diviseur $\Theta_{\eta}'$ est ample. Comme $J_{\eta}/\eta$ est
égale à sa partie fixe, il existe $\phi:A\rightarrow J_{\eta}$ une
$\eta$-isogénie de variétés abéliennes avec $A$ une variété
abélienne constante, telle que $N:=\ker(\phi)$ est un groupe
radiciel. On va montrer que $N$ est constant. Sinon, notons
$D':=\phi^{\ast}\Theta_{\eta}'$ l'image réciproque de
$\Theta_{\eta}'$ dans $A$. Alors $D'$ est un diviseur constant de
$p'$-torsion. Soit $G$ le plus grand sous-schéma en groupes
radiciel de $A$ qui laisse stable $D'$. Le fait que $D'$ est
constant implique que $G$ l'est aussi. Et par définition de $D'$,
on a $N\subset G$. Puisque $N$ n'est pas constant, $G':=G/N$ est
un groupe radiciel non trivial. Donc $\Theta_{\eta}'$ peut se
descendre en un diviseur $\Delta$ sur $J_{\eta}/G'$. Or le nombre
d'intersection $(\Delta^{g})$ est $\geq g!$, il en résulte que
$(\Theta_{\eta}'^{g})\geq p^{g}g!$, mais comme $\Theta_{\eta}$ est
algébriquement équivalent à $p-1$ fois le diviseur thêta classique
(\ref{theta est sym}), on a $(\Theta_{\eta}'^{g})\leq
(\Theta_{\eta}^{g})=(p-1)^{g}g!$. On obtient donc $p^{g}\leq
(p-1)^{g}$, d'où une contradiction. Par conséquent, le groupe
radiciel $N$ est en fait constant. En particulier, $J_{\eta}$ est
constante, ce qui implique que le groupe de Néron-Severi de
$J_{\eta}$ est un groupe étale consant. Donc sa polarisation
principale définie par le diviseur thêta classique est aussi
constante (à équivalence algébrique près). Finalement, il suffit
d'utiliser le théorème de Torelli (\cite{Oort}) pour conclure que
$X/S$ est constante.
\end{proof}

\begin{remarque1}\emph{ Revenons à la question \ref{question}. La réponse
serait positive si les deux conditions suivantes étaient
satisfaites:}

\emph{(C1): Pour $C$ une courbe lisse projective connexe de genre
$g\geq 2$, le diviseur $\Theta$ ne contient pas de composantes
abéliennes.}

\emph{(C2): Soient $A$ une variété abélienne sur une clôture du
corps premier, $D$ diviseur  ample  de  $p'$-torsion, et $F$ un
fermé de $D$ tel que $\mathrm{Sat}(F) = \mathrm{Sat}(D)$. Alors la
réunion des composantes de $D$ contenues dans $F$ est un diviseur
ample.}

\emph{En fait, soit $\Theta_{\eta}'$ la réunion des composantes de
$p'$-torsion, il résulte alors de (C1) et (C2) que
$\Theta_{\eta}'$ est ample. Ensuite, d'après Hrushovski et par
(C1), $J_{\eta}$ est égale à sa partie fixe. On conclut que $J$,
et puis $X$ sont constantes par le même argument que dans le cas
supersingulier.}
\end{remarque1}

%% file: cmguide.bbl
\begin{thebibliography}{99}


\bibitem{Beauville} Beauville, A., \textit{Vector bundles on curves and theta
functions}, in Moduli spaces and arithmetic geometry (Kyoto 2004),
Advanced Studies in Pure Math. 45 (2006), 145--156.


\bibitem{BLR1} Bosch, S., Lütkebohmert, W., Raynaud, M., \textit{Néron models}, Ergebnisse der
Mathematik und ihrer Grenzgebiete 3. Folge, Band 21.,
Springer-Verlag, 1990.


\bibitem{Bouw} Bouw, I., \textit{The $p$-rank of curves and covers of
curves}, in Courbes semi-stables et groupe fondamental en
géométrie algébrique, Progr. Math. 187, Birkhauser, 2000.

\bibitem{BGN} Brambila-Paz, L., Grzegorczyk, I., Newstead, P. E.,
\textit{Geography of Brill-Noether loci for small slopes}, J.
Algebraic Geometry 6 (1997), 645--660.


\bibitem{de Jong} De Jong, A. J., \textit{Smoothness, semi-stability and
alterations}, Inst. Hautes Etudes Sci. Publ. Math. 83 (1996),
51--93.

\bibitem{DM} Deligne, P., Mumford, D., \textit{The irreducibility of the space of curves of given
genus}, Inst. Hautes Etudes Sci. Publ. Math. 34 (1969), 75--110.

\bibitem{Demazure-Gabriel} Demazure, M., Gabriel, P., \textit{Groupes algébriques, tom
I}, MASSON et CIE, Paris North-Holland, Amsterdam.

\bibitem{Eisenbud} Eisenbud, D., \textit{Commutative algebra with a view toward algebraic
geometry}, Graduate Texts in Mathematics 150, Springer-Verlag,
1994.

\bibitem{Ekedahl} Ekedahl, T., \textit{The action of monodromy on torsion points of
Jacobians}, in Arithmetic algebraic geometry (Texel, 1989),
41--49, Progr. Math. 89, Birkhäuser Boston, Boston, MA, 1991.

\bibitem{FantechiGottsche} Fantechi, B., Göttsche, L., \textit{Local properties
and Hilbert schemes of points}, Part 3 \textit{in} Fundamental
algebraic geometry: Grothendieck's FGA explained, 139--178,
American Mathematical Society, 2007.


\bibitem{EGA} Grothendieck, A., \textit{Eléments de géométrie algébrique
II, III, IV}, Inst. Hautes Etudes Sci. Publ. Math. 8, 11, 17, 20,
24, 28, 32 (1961--1967).

\bibitem{SGA1} Grothendieck, A., \textit{Revêtements étales et groupe
fondamental}, Lect. Notes in Math. 224, Springer,
Berlin-Heidelberg-New York, 1971.

\bibitem{SGA2} Grothendieck, A., \textit{Cohomologie locale des faisceaux
cohérents}, Augmenté d'un exposé par Michèle Raynaud, Advanced
Studies in Pure Mathematics Vol. 2, North-Holland Publishiing Co.,
Amsterdam; Masson et cie, Editeur, Paris, 1968.

\bibitem{SGA3} Grothendieck, A., \textit{Schémas en groupes I},
Lect. Notes in Math. 151, Springer, Berlin-Heidelberg-New York,
1970.


\bibitem{SGA5} Grothendieck, A., \textit{Cohomologie $l$-adique et fonction
$L$}, Lect. Notes in Math. 589, Springer, Berlin-Heidelberg-New
York, 1977.


\bibitem{Quot} Grothendieck, A., \textit{Techniques de construction et théorèmes
d'existence en géométrie algébrique IV. Les schémas de Hilbert},
Séminaire Bourbaki Vol 6, Exp. No. 221, Soc. Math. France., Paris,
1995.

\bibitem{Hartshorne} Hartshorne, R., \textit{Algebraic geometry},
Graduate Texts in Mathematics 52, Springer-Verlag, 1977.

\bibitem{RD} Hartshorne, R., \textit{Residues and duality}, Lect. Notes
in Math. 20, Springer-Verlag, 1966.

\bibitem{Hirschowitz} Hirschowitz, A., \textit{Problèmes de Brill-Noether en rang
supérieur}, disponsible depuis la page web:
http://math1.unice.fr/\~{}ah/

\bibitem{Hrushovski} Hrushovski, E., \textit{The Mordell-Lang conjecture for function
fields}, Jour. of the American Math. Soc. 9 (1996), 667--690.

\bibitem{Joshi} Joshi, K., \textit{Stability and locally exact differentials on a
curve}, C. R. Math. Acad. Sci. Paris 338 (2004), 869--872.

\bibitem{KatzMazur} Katz, N., Mazur, B., \textit{Arithmetic moduli of elliptic
curves}, Ann. of Math. Studies 108, Princeton University Press,
1985.

\bibitem{Kempf} Kempf, G., Appendix to: \textit{Varieties defined by quadratic
equations}, by Mumford, D., in Questions on algebraic varieties,
95--100, CIME, Roma, 1970.

\bibitem{Laszlo} Laszlo, Y., \textit{Un théorème de Riemann pour
les diviseurs thêta sur les espaces de modules de fibrés stables
sur une courbe}, Duke Math. J. 64 (1991), 333--347.

\bibitem{Laumon} Laumon, G., \textit{Transformation de Fourier
généralisée}, arXiv:alg-geom/9603004v1

\bibitem{Milne} Milne, J., \textit{Jacobian varieties}, in Arithmetic
geometry (Proc. conference on arithmetic geometry, Storrs, August
1984), 167--212, Springer, 1986.

\bibitem{Mori} Mori, S., \textit{The endomorphism rings of some abelian
varieties}, Japan. J. Math. (N.S.) 2 (1976), 109--130.

\bibitem{Mukai} Mukai, S., \textit{Duality between $\mathrm{D}(X)$
and $\mathrm{D}(\hat{X})$ with its application to Picard sheaves},
Nagoya Math. J. 81 (1981), 153--175.

\bibitem{AV} Mumford, D., \textit{Abelian varieties}, Tata Inst. Fund. Res.
Studies in Math. 5, Bombay, Tata Inst. Fund. Res., 1970.

\bibitem{Eq} Mumford, D., \textit{On the equations defining abelian
varieties. I}, Invent. Math. 1 (1966), 287--354.

\bibitem{theta} Mumford, D., \textit{Theta characteristic of an algebraic
curve}, Ann. Sci. Ecole Norm. Sup. 4 (1971), 181--192.

\bibitem{Prym} Mumford, D., \textit{Prym varieties}, Contributions
to analysis (a collection of papers dedicated to Lipman Bers),
325-350, Academic Press, New York, 1974.

\bibitem{NR} Narasimhan, M. S., Ramanan, S., \textit{Moduli of vector bundles on a compact Riemann
surface}, Ann. of Math. (2) 89 (1969), 14--51.

\bibitem{Oda} Oda, T., \textit{Vector bundles on an elliptic
curve}, Nagoya Math. J. 43 (1971), 41--72.

\bibitem{Oort} Oort, F., \textit{Finite group schemes, local moduli for abelian varieties and lifting
problems}, Compos. Math. 23 (1971), 265--296.

\bibitem{Pink-Roessler} Pink, R., Rössler, D., \textit{On $\psi$-invariant
subvarieties of semiabelian varieties and the Manin-Mumford
conjecture}, J. Algebraic Geometry 13 (2004), 771--798.

\bibitem{Pop-Saidi} Pop, F., Saïdi, M., \textit{On the specialisation homomorphism of fundamental groups of curves
in positive characteristics}, Math. Sci. Res. Ins. Pub. 41 (2003),
107--118.


\bibitem{R1} Raynaud, M., \textit{Sections des fibrés vectoriels sur une
courbe}, Bull. Soc. Math. France 110 (1982), 103--125.

\bibitem{R3} Raynaud, M., \emph{Revêtements des courbes
en caractéristique $p>0$ et ordinarité}, Compos. Math. 123 (2000),
73--88.

\bibitem{R2} Raynaud, M., \emph{Sur le groupe fondamental d'une courbe
complète en caractéristique $p>0$}, in Arithmetic fundamental
groups and noncommutative algebra (Berkeley, 1999), 335--351,
Proc. Sympos. Pure Math. 70, Amer. Math. Soc., 2002.

\bibitem{R4} Raynaud, M., \emph{Contre-exemple au {\textquotedblleft Vanishing theorem\textquotedblright} en caractéristique
$p>0$}, in C. P. Ramanujam-A Tribute, 273--278, Tata Inst. Fund.
Res. Studies in Math. 8, Bombay, Tata Inst. Fund. Res. 1978.


\bibitem{Serre1} Serre, J. -P., \emph{Sur la topologie des
variétés algébriques en caractéristique $p$}, 1958 Symposium
internacional de topología algebraica, 24-53., Universidad
Nacional Autónoma de México and UNESCO, Mexico City.


\bibitem{Szpiro} Szpiro, L., \emph{Travaux de Kempf, Kleiman, Laksov
sur les diviseurs exceptionnels}, Séminaire Bourbaki 14
(1971-1972), Exposé No. 417, Lect. Notes in Math. 317,
Springer-Verlag, 1973.


\bibitem{Tamagawa} Tamagawa, A., \emph{Finiteness of isomorphism classes of curves in
positive characteristic with prescribed fundamental groups}, J.
Algebraic Geometry 13 (2004), 675--724.

\bibitem{Tango} Tango, H., \textit{On the behavior of extensions
of vector bundles under the Frobenius map}, Nagoya Math. J. 48
(1972), 73--89.

\bibitem{GeemenIzadi} van Geemen, B., Izadi, E., \textit{The tangent space to
the moduli space of vector bundles on a curve and the singular
locus of the theta divisor of the jacobian}, J. Algebraic Geometry
10 (2001), 133--177.

\end{thebibliography}
